\documentclass{amsart}
\usepackage{amssymb, amsmath,mathtools}
\usepackage{mathrsfs}
\usepackage{amscd}
\usepackage{bbm}
\usepackage{verbatim}
\usepackage{stmaryrd}
\usepackage[dvipsnames]{xcolor}
\usepackage{a4wide}
\usepackage{enumitem}
\usepackage[backend=biber,maxbibnames=199,doi=false,isbn=false]{biblatex}
\usepackage[colorlinks,linkcolor={blue},citecolor={blue},urlcolor={purple},]{hyperref}

\renewcommand{\O}{\Omega}
\renewcommand{\div}{\normalfont{\text{div}}}
\newcommand{\pr}{\mathbb{P}}
\newcommand{\Ls}{\mathbb{L}}

\newcommand{\Hs}{\mathbb{H}}

\newcommand{\wt}{\widetilde}
\newcommand{\Dom}{\mathcal{O}}
\newcommand{\lb}{\langle}
\newcommand{\rb}{\rangle}
\newcommand{\calL}{{\mathcal L}}

\newcommand{\dd}{\,\mathrm{d}}

\renewcommand{\MR}{\mathrm{MR}}
\newcommand{\UH}{\mathcal{L}_2(U,H)}
\newcommand{\one}{\mathbbm{1}}
\newcommand{\<}{\langle}
\newcommand{\?}{\rangle}
\newcommand{\nn}{|\!|\!|}

\newcommand{\into}{\hookrightarrow}

\DeclareMathOperator*{\esssup}{\textup{ess\,sup}}

\newcommand{\supp}{\mathrm{supp}}

\theoremstyle{plain}
\newtheorem{theorem}{Theorem}[section]
\theoremstyle{remark}
\newtheorem{remark}[theorem]{Remark}

\theoremstyle{plain}
\newtheorem{corollary}[theorem]{Corollary}
\newtheorem{lemma}[theorem]{Lemma}
\newtheorem{proposition}[theorem]{Proposition}
\newtheorem{definition}[theorem]{Definition}

\newtheorem{assumption}[theorem]{Assumption}

\numberwithin{equation}{section}

\def\N{{\mathbb N}}

\def\R{{\mathbb R}}

\newcommand{\E}{\mathbb{E}}

\newcommand{\BB}{\mathcal{B}}

\newcommand{\F}{\mathcal{F}}
\newcommand{\G}{\mathcal{G}}
\renewcommand{\P}{\mathbb{P}}
\newcommand{\PP}{\mathcal{P}}

\newcommand{\om}{\omega}
\newcommand{\Om}{\Omega}
\newcommand{\eps}{\varepsilon}

\usepackage{stmaryrd}

\allowdisplaybreaks

\begin{document}

\author[E. S. Theewis]{Esm\'ee Theewis}
\address[E. S. Theewis]{Delft Institute of Applied Mathematics\\
Delft University of Technology \\ P.O. Box 5031\\ 2600 GA Delft\\The
Netherlands}
\email{E.S.Theewis@tudelft.nl}

\author[M. C. Veraar]{Mark Veraar}
\address[M. C. Veraar]{Delft Institute of Applied Mathematics\\
Delft University of Technology \\ P.O. Box 5031\\ 2600 GA Delft\\The
Netherlands}
\email{M.C.Veraar@tudelft.nl}

\thanks{The authors are supported by the VICI subsidy VI.C.212.027 of the Netherlands Organisation for Scientific Research (NWO)}

\date{February 20, 2026}

\title[Large Deviations for Stochastic Evolution Equations]{Large Deviations for Stochastic Evolution Equations in the Critical Variational Setting}



\begin{abstract}
Using the weak convergence approach, we prove the large deviation principle (LDP) for solutions to quasilinear stochastic evolution equations with small Gaussian noise in the critical variational setting, a recently developed general variational framework. No additional assumptions are made apart from those required for well-posedness. In particular, no monotonicity is required, nor a compact embedding in the Gelfand triple. Moreover, we allow for flexible growth of the diffusion coefficient, including gradient noise. This leads to numerous applications for which the LDP was not established yet, in particular equations on unbounded domains with gradient noise. 
Since our framework includes the 2D Navier--Stokes and Boussinesq equations with gradient noise and unbounded domains, our results resolve an open problem that has remained unsolved for over 15 years.  
\end{abstract}

\maketitle
\addtocontents{toc}{\protect\setcounter{tocdepth}{2}}

\tableofcontents

\section{Introduction}
In this paper we study large deviations for solutions to small-noise stochastic evolution equations of the form
\begin{equation}\label{eq: stoch ev eq}
  \dd Y^\eps(t)=-A(t,Y^\eps(t))\dd t+\sqrt{\eps}B(t,Y^\eps(t))\dd W(t)
\end{equation}
in the new variational framework of \cite{AV22variational} by Agresti and the second author. This framework, the \emph{critical variational setting}, has been developed to extend the classical variational approach to stochastic evolution equations originating from \cite{bensoussan72}, \cite{pardoux},  \cite{krylov81}. In the classical variational approach, the drift and diffusion coefficients $A$ and $B$ need to satisfy several conditions to ensure well-posedness of \eqref{eq: stoch ev eq}. The usual weak monotonicity condition is especially restrictive. It is therefore no surprise that efforts have been made to weaken the monotonicity condition, e.g.\ in \cite[\S 5.2]{liurockner15}  and very recently \cite{rockner22} with a much weaker local monotonicity condition. One of the advantages of the critical variational setting of \cite{AV22variational} used in this paper, is that no form of monotonicity is assumed. In return, $A$ and $B$ are of a slightly less (but still very) general form:
\[
A(t,v)=A_0(t,v)v-F(t,v)+f(t), \quad B(t,v)=B_0(t,v)v+G(t,v)+g(t)
\]
for $t\in\R_+$ and $v\in V$, where $(V,H,V^*)$ is a Gelfand triple belonging to the stochastic evolution equation.
That is, $(A,B)$ contains a quasilinear part  $(A_0,B_0)$ and a semilinear part $(F,G)$ and it is assumed that both parts satisfy certain critical local Lipschitz conditions, where the Lipschitz constant may depend arbitrarily on $\|v\|_H$ and, allowing even more flexibility,  polynomially on interpolation norms $\|v\|_{V_\beta}$, where $V_\beta=[V^*,V]_\beta$ denotes the complex interpolation space. Besides the absence of any  monotonicity assumption, another major improvement of the critical variational setting is the weakening of the usual growth conditions on the diffusion coefficient $B$, allowing e.g.\ for gradient noise. Lastly, a special feature is that critical nonlinearities are allowed (see \eqref{eq: critical}), which is not the case in other settings. The critical variational  setting covers many semilinear and some quasilinear equations that were not covered by more classical variational settings. In particular, this holds for many equations that require an (analytically) strong setting, in which monotonicity often fails, for example the Cahn--Hilliard equation, the tamed Navier--Stokes equations and the Allen--Cahn equation. See \cite[\S 5]{AV22variational} for details. The exact assumptions in the critical variational setting can be found in Section \ref{sec: Main result}. Finally, it should be stressed that unlike the settings in \cite{rockner22}, the critical variational setting does not require a compact (Sobolev) embedding $V\into H$ and is thus suited to treat equations on unbounded spatial domains.

The goal of this paper is to establish the LDP for solutions to small-noise stochastic evolution equations in the critical variational setting. Large deviations have been studied for SPDEs in many different frameworks. The first results for SPDEs were inspired by the pioneering paper for SDEs by Freidlin and Wentzell \cite{freidlin12} (see also \cite[\S 5.6]{dembo10}), relying on discretizations and the contraction principle. These techniques were extended to several SPDE settings  with Gaussian noise, notably in \cite{daprato14}, \cite{peszat94}, \cite{cerrai04} (stochastic
reaction-diffusion equations), \cite{chow92} (semilinear parabolic equations) and \cite{rockner06} (stochastic porous media equations). However, for less regular $A$ and $B$ such techniques are difficult to use in general settings.
In 2001, Budhiraja and Dupuis proved a substantially generalized contraction principle, the so-called weak convergence approach to large deviations \cite{budhidupuis01}.  This approach turned out to be extremely powerful for SPDEs and subsequently, it was applied to many SPDEs with less regular coefficients, e.g.\ in \cite{sritharansundar06}, \cite{chueshovmillet10} (2D Navier--Stokes and hydrodynamical models), \cite{duanmillet09} (Boussinesq equations), as well as \cite{renzhang08} and \cite{liu09} (general classical variational settings). A more detailed discussion on applications to fluid dynamics can be found below. 

Also for the recent variational settings with even weaker conditions on the coefficients $A$ and $B$,  the weak convergence approach has led to new LDP proofs. In \cite{hongliliu21} the LDP is obtained for McKean-Vlasov quasilinear stochastic evolution equations, in \cite{pan22} for a setting from \cite{rockner22}, in \cite{kumarmohan22} for the same setting extended to L\'evy noise, and most recently, \cite{pan24} obtained the LDP for the strongest setting of \cite{rockner22}. The latter allows flexible growth bounds on $B$,   including gradient noise. Still, the combination of flexible growth of $B$ and unbounded spatial domains (that is, no compact embedding $V\into H$) has not been covered in any of the papers so far. The main improvement of our work is that we allow for both. In fact, no additional bounds on $A$ and $B$ are assumed for the LDP apart from those in \cite{AV22variational} required for well-posedness, nor do we assume a compact embedding in the Gelfand triple. New techniques are used to replace the usual  compactness arguments. The paper contains new approaches for
\begin{itemize}[label=$-$]
  \item well-posedness of the skeleton equation and  compact sublevel sets of the rate function in the LDP, by means of maximal regularity theory and a strong approximation argument,
  \item the stochastic continuity criterion from the weak convergence approach, using critical estimates for the nonlinearities and an effective combination of deterministic and  stochastic Gronwall inequalities.
\end{itemize}

The LDP result in this paper opens up many new applications.
In particular, the following examples are included on bounded and unbounded domains in $\R^d$ and with gradient noise:
\begin{itemize}[label=$-$]
  \item Navier--Stokes equations for $d=2$ \cite[App.\ A]{AV24SNS},
  \item tamed Navier--Stokes equations for $d=3$ \cite[\S 5.2]{AV22variational},
  \item Cahn--Hilliard equation for $d=1,2$ \cite[\S 5.1]{AV22variational},
  \item Swift--Hohenberg equations for $d=1,2,3$  \cite[\S 5.6]{AV22variational},
 \item many reaction-diffusion equations, e.g.\ for $d\leq 4$:
 \begin{itemize}[label=--]
   \item Allen--Cahn equations \cite[\S 5.4]{AV22variational},
   \item symbiotic Lotka--Volterra equations \cite[Th.\ 3.11]{AV24reacdiff-arXiv},
   \item coagulation equations \cite[Th.\ 3.9]{AV24reacdiff-arXiv}.
 \end{itemize}
 \end{itemize}
 This list is far from extensive.
 
To make our results concrete for some of the models discussed above, we present an application to a general fluid dynamics model in Section \ref{ss:fluid}. 
Specifying further, in Subsection \ref{ss:NS}, we derive the LDP for the 2D Navier--Stokes equations with gradient noise and do not assume that the underlying domain is bounded.  

The LDP for the 2D Navier--Stokes equations with gradient noise and unbounded domains was already considered in the pioneering work \cite{sritharansundar06}. However, the proof of \cite[Lem.\ 4.5]{sritharansundar06} is incomplete -- specifically, the argument in the final line. 
The gap was already indicated in \cite[p.\ 2054]{duanmillet09} and concerns the compactness of the sublevel sets of the rate function. 
In \cite{chueshovmillet10, duanmillet09}, the gap is resolved for the 2D Navier--Stokes and Boussinesq equations under the assumption that the noise is gradient-free (see \cite[Th.\  3.2]{chueshovmillet10}, \cite[Ass.\ A Bis, p.\ 2072]{duanmillet09}), but the gradient noise case has remained open since then.  
The possibly unbounded spatial domains rule out alternative arguments based on compact Sobolev embeddings. 
We have now  covered the gradient noise case in Theorems \ref{thm:fluidabstract} and \ref{thm:SNS} -- extending the results of \cite{chueshovmillet10, duanmillet09} and completing the proof of \cite{sritharansundar06}. 

Another application that we would like to highlight are the 3D tamed Navier--Stokes equations, for which a large deviation principle was established in \cite{RoZhZh}. However, gradient noise was not considered in that work, and it is far from straightforward to extend their approach to settings where such noise is present. 
Our main result, Theorem \ref{th: main LDP theorem},  now includes the gradient noise case and also applies to a broad class of other models (see the list above). 

Closing the above indicated gap requires several intricate approximation techniques, which are detailed in Subsection \ref{ss:weakskeleton}. Furthermore, in the full abstract setting, the stochastic continuity criterion also necessitates new ideas, which we develop in Subsection \ref{ss:stochcont}.  

\subsubsection*{Acknowledgement}
The authors thank Antonio Agresti and Sebastian Bechtel for their helpful comments.

\subsection*{Notation}
We let $\R_+\coloneqq [0,\infty)$.
For $T>0$ and a normed space $X$ we let $C([0,T];X)$ denote the space of continuous functions from $[0,T]$ to $X$ equipped with supremum norm $\|f\|_{C([0,T];X)}\coloneqq \sup_{t\in[0,T]}\|f(t)\|_X$. For $(S,\mathcal{A},\mu)$ a measure space, we denote by $L^0(S;X)$ the space of strongly measurable functions $f\colon S\to X$, with identification of a.e.\ equal functions. For $p\in (0,\infty]$, we let $L^p(S;X)$ denote the subset consisting of all  $f\in L^0(S;X)$ for which $\|f\|_{L^p(S;X)}<\infty$, where
\[
\|f\|_{L^p(S;X)}\coloneqq
\begin{cases}
\left(\int_S\|f(s)\|_X^p \dd \mu(s)\right)^{\frac{1}{p}}, & p<\infty,\\
\esssup_{s\in S}\|f(s)\|_X, & p=\infty.
\end{cases}
\]
We write $L^p(S)\coloneqq L^p(S;\R)$ and if $S=[0,T]\subset \R$, we write $L^p(0,T;X)\coloneqq L^p(S;X)$. Moreover, we let $L^p_{\mathrm{loc}}(\R_+;X)\coloneqq \{u\colon \R_+\to X : u|_{[0,T]}\in L^2(0,T;X) \text{ for all } T\in \R_+\}$.

For Hilbert spaces $U$ and $H$ we let $\mathcal{L}(U,H)$ and $\UH$ denote the continuous linear operators and Hilbert-Schmidt operators from $U$ to $H$, respectively.
For brevity, we write
\[
\nn \cdot\nn_H\coloneqq \|\cdot\|_{\UH}.
\]
Furthermore, we denote the dual of a Hilbert space $V$ by $V^*$ and for $\beta\in(0,1)$, we denote the complex interpolation space at $\beta$ by
\[
V_\beta\coloneqq [V^*,V]_\beta, \quad \|\cdot\|_\beta\coloneqq \|\cdot\|_{V_{\beta}}.
\]

For a metric space $M$ we denote its Borel $\sigma$-algebra by $\BB(M)$. The unique product measure space of two $\sigma$-finite measure spaces $(S_1,\mathcal{A}_1,\mu_1)$ and $(S_2,\mathcal{A}_2,\mu_2)$ is denoted by $(S_1\times S_2,\mathcal{A}_1\otimes \mathcal{A}_2,\mu_1\otimes\mu_2)$.
Let $I=[0,T]$ or $I=\R_+$ and let $X$ be a Banach space. A process $(\Phi(t))_{t\in I}$ is a strongly measurable function $\Phi\colon I\times\Om\to X$. It is called strongly
progressively measurable if for every $t \in I$, $\Phi|_{[0,t]\times \Om}$ is strongly $\mathcal{B}([0, t]) \otimes \F_t$-measurable.  For  $I=\R_+$, we denote the
$\sigma$-algebra generated by the strongly progressively measurable processes by $\PP$.

We write $a\vee b\coloneqq \max(a,b)$ and $a\wedge b\coloneqq \min(a,b)$ for $a,b\in\R$.

\section{Main result}\label{sec: Main result}

We specify our setting for stochastic evolution equations  and recall the definition of the large deviation principle before we state our main result, Theorem \ref{th: main LDP theorem}.

\subsection{The critical variational setting}

We let $(V,H,V^*)$ be a Gelfand triple of real Hilbert spaces. That is, $(V,(\cdot,\cdot)_V)$ and $(H,(\cdot,\cdot)_H)$ are real Hilbert spaces such that there exists a continuous and dense embedding $\iota\colon V\into H$. Then, $j\colon H\into V^*\colon x\mapsto (x, \iota(\cdot))_H$ is a continuous embedding and $j(H)$ is dense in $V^*$ by reflexivity of $V$ ($j=\iota^*$ under Riesz' identification $H\cong H^*$). From now on we identify  $x\in V$ with $\iota(x)\in H$ and $x\in H$ with $j(x)\in V^*$. Then, if $\<\cdot,\cdot\?$ denotes the duality pairing between the abstract dual $V^*$ and $V$, one has
\[
\<x,v\?=(v,x)_H \quad\text{for all } x\in H, v\in V.
\]

For convenience of the reader, we recall that in applications, one does not work with the abstract dual $V^*$ but with a space $V'$ which, under some assumptions and with the correct duality pairing, is isomorphic to $V^*$, see also \cite[p.\ 1244]{krylov81}. One starts with reflexive Banach (or Hilbert) (sub)spaces $V\subset H\subset V'$, where each inclusion is dense and continuous and one defines $j\colon H\into V^*\colon x\mapsto (x, \cdot)_H$. Then, provided that
\begin{equation}\label{eq: j continuous}
|(x,v)_H|\leq \|x\|_{V'}\|v\|_V \quad\text{for all }x\in H, v\in V,
\end{equation}
there exists a unique continuous extension to a map $j_1\colon V'\to V^*$.  Furthermore, if $j_1$ is bijective, then it follows that $j_1\colon V'\cong V^*$ as normed spaces, although not necessarily isometrically. The duality pairing is then given by $\<v',v\?\coloneqq j_1(v')(v)$ and for $x\in H,v\in V$ we have $\<x,v\?=(x,v)_H$ since $j_1$ is the extension of $j$. The triple $(V,H,V')$ is also called a Gelfand triple and simply denoted by $(V,H,V^*)$, where as explained, the correct duality pairing $\<\cdot,\cdot\?\colon V'\times V\to \R$ is given by $\<v',v\?\coloneqq j_1(v')(v)$.

In fact, bijectivity of $j_1$ holds if and only if there exists $\alpha>0$ such that
\begin{equation}\label{eq: bijective criterion}
\alpha\|x\|_{V'}\leq\sup_{v\in V, \|v\|_V\leq 1}|(x,v)_H|\eqqcolon \|j_1(x)\|_{V^*} \quad \text{for all }x\in H.
\end{equation}
The equivalence follows from \cite[Th.\ 4.48]{rynne08}, density of $H\subset V'$ and continuity of $j_1$, and density of $\mathrm{Im}(j_1)\subset V^*$. The latter holds since $j_1(V')\supset j_1(H)=j(H)$ and one can verify that $j(H)$ is dense in $V^*$ using reflexivity of $V$. In conclusion, provided that $V\subset H\subset V'$ continuously and densely, one only has to verify \eqref{eq: j continuous} and \eqref{eq: bijective criterion} to have $j_1\colon V'\cong V^*$.

Popular choices for the Gelfand triple are the weak and strong setting for a given differential operator. For example, if $A(t,u)\coloneqq\Delta u$ on $\R^d$, then one can use
\begin{align*}
&V=H^1(\R^d),\; H=L^2(\R^d),\; V'=H^{-1}(\R^d)= V^*  &\text{(weak setting)},\\
  & V=H^2(\R^d),\; H=H^1(\R^d),\; V'=L^2(\R^d)\cong V^* &\text{(strong setting)}.
\end{align*}
See also \cite[Ex.\ 2.1, Ex.\ 2.2]{AV22variational}.

Recall that $H=[V^*,V]_{\frac{1}{2}}$ \cite[\S 5.5.2]{arendt02} and the following interpolation estimate holds for $\beta\in(\frac{1}{2},1)$:
\begin{equation}\label{eq:interpolation estimate}
\|v\|_\beta\leq K\|v\|_H^{2-2\beta}\|v\|_V^{2\beta-1}, \qquad v\in V.
\end{equation}

Since strong solutions are required to be strongly measurable, see Definition \ref{def:strong sol} below, one can assume without loss of generality that $V$ and $H$ are separable, see also \cite[p.\ 1244]{krylov81}.
Thus, from now on we assume that $V$ and $H$ are separable.

As mentioned in the introduction, we work with the critical variational setting from \cite{AV22variational}.
We consider stochastic evolution equations of the form
\begin{equation}\label{eq: original stoch ev eq}
  \begin{cases}
  &\dd u(t)=-A(t,u(t))\dd t+B(t,u(t))\dd W(t), \quad t\in[0,T], \\
  &u(0)=x,
\end{cases}
\end{equation}
where $x\in H$, $T>0$ and $W$ is a $U$-cylindrical Brownian motion (see Definition \ref{def: cylindrical BM}).

If $\Phi:[0,T]\times \Om\to \UH$ is strongly progressively measurable and $\Phi\in L^2(0,T;\UH)$ a.s., then one can define the stochastic integral $\int_0^t \Phi(s)\dd W(s)$ for $t\in[0,T]$, see \cite[\S 5.4 $(p=0)$]{NVW15}.

We now specify what we mean by a strong solution to \eqref{eq: original stoch ev eq}. In our definition we also  allow for $L^1(0,T;H)$-valued integrands, which is only needed to treat the skeleton equation associated to \eqref{eq: original stoch ev eq}, see Definition \ref{def: skeleton} below.

\begin{definition}\label{def:strong sol}
For $T>0$, we define the \emph{maximal regularity space} by
\[
\MR(0,T)\coloneqq C([0,T];H)\cap L^2(0,T;V), \quad \|\cdot\|_{\MR(0,T)}\coloneqq \|\cdot\|_{C([0,T];H)}+\|\cdot\|_{L^2(0,T;V)}.
\]
Let $A\colon\R_+\times V\to V^*$, $B\colon\R_+\times V\to \UH$ and let $x\in H$. Let $W$ be a $U$-cylindrical Brownian motion on a filtered probability space $(\Om,\mathcal{F},(\mathcal{F}_t)_{t\geq 0},\P)$ and let $T>0$. We say that a strongly progressively measurable process $u\colon [0,T]\times \Om\to V$ is a \emph{strong solution} to \eqref{eq: original stoch ev eq} if a.s.
\[
u\in\MR(0,T),\; A(\cdot,u(\cdot))\in L^2(0,T;V^*)+L^1(0,T;H),\; B(\cdot,u(\cdot))\in L^2(0,T;\UH)
\]
and a.s.
\begin{align}\label{eq:strong sol}
  u(t)=x-\int_0^t A(s,u(s))\dd s+\int_0^t B(s,u(s))\dd W(s) \: \text{ in }V^* \text{ for all }t\in[0,T].
\end{align}
A strong solution $u$ is \emph{unique} if for any other strong solution $v$ we  have a.s. $u=v$ in $\MR(0,T)$.

If $B=0$, we write $u'(t)=-A(t,u(t))$ instead of $\dd u(t)=-A(t,u(t))\dd t$ in \eqref{eq: original stoch ev eq} and we call $u\in \MR(0,T)$ a strong solution  if $A(\cdot,u(\cdot))\in L^2(0,T;V^*)+L^1(0,T;H)$ and \eqref{eq:strong sol} holds.
\end{definition}

For the weak convergence approach to large deviations it is necessary to let $A$ and $B$ be defined on $\R_+\times V$ rather than $\R_+\times\Om\times V$, meaning that stochasticity enters $A$ and $B$  through the solution $u$ in \eqref{eq: original stoch ev eq} and not separately. Also, the initial value $x$ in \eqref{eq: stoch ev eq} has to be deterministic. Other than that, we make exactly the same assumptions as those required for global well-posedness \cite[Th.\ 3.5]{AV22variational}. Let us introduce these assumptions.

\begin{assumption}\label{ass: crit var set}
We assume that:
\begin{enumerate}[label=\emph{(\arabic*)},ref=\ref{ass: crit var set}\textup{(\arabic*)}]
\item \label{it:ass1} $A(t,v)=A_0(t,v)v-F(t,v)-f$ and  $B(t,v)=B_0(t,v)v+G(t,v)+g$, where
\[
A_0\colon \R_+ \times H\to\mathcal{L}(V,V^*) \text{ and } B_0\colon \R_+ \times H \to \mathcal{L}(V,\UH)),
\]
are $\BB(\R_+)\otimes \mathcal{B}(H)$-measurable, and
\[
F\colon \R_+\times V\to V^* \text{ and } G\colon \R_+\times V\to\UH
\]
are $\BB(\R_+)\otimes \mathcal{B}(V)$-measurable, and $f\colon \R_+\to V^*$ and $g\colon \R_+\to \UH$ are $\BB(\R_+)$-measurable maps with
\[
f\in L_{\mathrm{loc}}^2(\R_+;V^*) \text{ and } g\in L_{\mathrm{loc}}^2(\R_+;\UH).
\]

\item \label{it:ass2} For all $T>0$ and $n\in\R_+$, there exist $\theta_{n,T}, M_{n,T}>0$ such that for any $t\in[0,T]$,
$u\in H$, $v\in V$ with $\|u\|_H\leq n$, we have
\[
\<A_0(t,u)v,v\?-\frac{1}{2}\nn B_0(t,u)v\nn_H^2\geq \theta_{n,T}\|v\|_V^2-M_{n,T}\|v\|_H^2.
\]

\item \label{it:ass3} There exist $\rho_j\geq 0$ and $\beta_j\in(\frac{1}{2},1)$ such that 
\begin{flalign}\label{eq: critical}
&\qquad\qquad2\beta_j\leq 1+\frac{1}{1+\rho_j},\quad j\in\{1,\ldots, m_F+m_G\},& \text{\emph{((sub)criticality)}}&
\end{flalign}
for some $m_F, m_G\in\N$ and for all $T>0$, $n\in\R_+$ there exists a constant $C_{n,T}$ such that for all $t\in[0,T]$ and $u,v,w\in V$ with $\|u\|_H,\|v\|_H\leq n$, we have
\begin{align*}
\|A_0(t,u)w\|_{V^*}&\leq C_{n,T}\|w\|_V,\\
\|A_0(t,u)w-A_0(t,v)w\|_{V^*}&\leq C_{n,T}\|u-v\|_H\|w\|_V,\\
\nn B_0(t,u)w\nn_{H}&\leq C_{n,T}\|w\|_V,\\
\nn B_0(t,u)w-B_0(t,v)w\nn_{H}&\leq C_{n,T}\|u-v\|_H\|w\|_V,\\
\|F(t,u)\|_{V^*}&\leq C_{n,T}\sum_{j=1}^{m_F} (1+\|u\|_{\beta_j}^{\rho_j+1}),\\
\|F(t,u)-F(t,v)\|_{V^*}&\leq C_{n,T}\sum_{j=1}^{m_F} (1+\|u\|_{\beta_j}^{\rho_j}+\|v\|_{\beta_j}^{\rho_j})\|u-v\|_{\beta_j},\\
\nn G(t,u)\nn_{H}&\leq C_{n,T}\sum_{j=m_F+1}^{m_F+m_G} (1+\|u\|_{\beta_j}^{\rho_j+1}),\\
\nn G(t,u)-G(t,v)\nn_{H}&\leq C_{n,T}\sum_{j=m_F+1}^{m_F+m_G} (1+\|u\|_{\beta_j}^{\rho_j}+\|v\|_{\beta_j}^{\rho_j})\|u-v\|_{\beta_j}.
\end{align*}
\end{enumerate}
Without loss of generality, we assume that the constants $C_{n,T}$ are non-decreasing in $n$ and $T$.
\end{assumption}

Because the coefficients are defined on $\R_+\times V$ instead of $\R_+\times\Om\times V$, the measurability in Assumption \ref{it:ass1} is different than in \cite[Ass.\ 3.1]{AV22variational}. However, $(A,B)$ satisfies our assumption if and only if $(\bar{A},\bar{B})$ satisfies \cite[Ass.\  3.1]{AV22variational}, where  $\bar{A}(t,\om,v)\coloneqq A(t,v)$ and $\bar{B}(t,\om,v)\coloneqq B(t,v)$ are trivial extensions.

$A_0$ and $B_0$  determine the leading differential order of the drift and the noise, respectively, and constitute the principal part of the quasilinear equation. In the semilinear case, $A_0(t,u)$ and $B_0(t,u)$  do not depend on $u$, rendering the principal part linear.

Condition \eqref{eq: critical} describes a balance between the growth rate $\rho_j+1$ of the nonlinearities $F$ and $G$
and the regularity coefficient $\beta_j$ (whose value is usually determined by Sobolev embeddings). In case of equality in \eqref{eq: critical} for some $j$, the nonlinearity is called \emph{critical}.

From \eqref{eq:interpolation estimate} and Assumption \ref{it:ass3}, it is clear that  $\|F(t,v)\|_{V^*}+\nn G(t,v)\nn_H\leq \tilde{C}_{\|v\|_H,T}(1+\|v\|_V)$ for all $t\in[0,T]$ if $\|v\|_H\leq n$, where $\tilde{C}_{\|v\|_H,T}$ is a constant. Thus we have integrability of $F(\cdot,u(\cdot))$ and $G(\cdot,u(\cdot))$ if $u\in\MR(0,T)$.

In \cite[Th.\ 3.3]{AV22variational} it is shown that under Assumption \ref{ass: crit var set}, there exists a unique local solution to \eqref{eq: original stoch ev eq}. In \cite[Th.\ 3.5]{AV22variational}, this is extended to a global well-posedness result under a  coercivity condition on $(A,B)$.
The next result follows from \cite[Th.\ 3.5]{AV22variational}.
\begin{theorem}\label{th: original stoch ev eq global well-posedness}
Let $(A,B)$ satisfy Assumption \ref{ass: crit var set} and suppose that $(A,B)$ is coercive in the following sense: for all $T>0$, there exist $\theta, M>0$ and $\phi\in L^2(0,T)$ such that for all $v\in V$ and $t\in[0,T]$,
\begin{equation}\label{eq: coercivity condition (A,B)}
\<A(t,v),v\?-\frac{1}{2}\nn B(t,v)\nn_H^2 \geq \theta\|v\|_V^2-M\|v\|_H^2-|\phi(t)|^2.
\end{equation}
Then, for any $x\in H$ and $T>0$, there exists a unique strong solution $u$ to \eqref{eq: original stoch ev eq} on $[0,T]$.
\end{theorem}
Energy estimates can also be found in \cite[Th.\ 3.5]{AV22variational}, but these will not be used. More general theory in an $L^p$-setting was developed in \cite{AV22nonlinear1,AV22nonlinear2}.

\subsection{Statement of the main result}

\begin{definition}\label{def: LDP}
Let $\mathcal{E}$ be a Polish space, let $(\Om,\F,\P)$ be a probability space and let $(Y^\eps)_{\eps>0}$ be a collection of $\mathcal{E}$-valued random variables on $(\Om,\F,\P)$. Let $I\colon \mathcal{E}\to[0,\infty]$ be a function. Then $(Y^{\eps})$ \emph{satisfies the large deviation principle (LDP) on $\mathcal{E}$} with rate function $I\colon S\to[0,\infty]$ if
\begin{enumerate}[label=\emph{(\roman*)},ref=\textup{(\roman*)}]
    \item $I$ has compact sublevel sets,
    \item for all open $E\subset \mathcal{E}$: $\liminf_{\eps\downarrow0}\eps\log \P(Y^{\eps}\in E)\geq -\inf_{z\in E}I(z)$,
    \item for all closed $E\subset \mathcal{E}$: $\limsup_{\eps\downarrow0}\eps\log \P(Y^{\eps}\in E)\leq -\inf_{z\in E}I(z)$.
\end{enumerate}
\end{definition}

Before we formulate our LDP result, we define  the skeleton equation, which appears in the rate function of our LDP.

\begin{definition}\label{def: skeleton}
Let $x\in H$ be fixed.
For $\psi\in L^2(0,T;U)$, the \emph{skeleton equation} associated to the stochastic evolution equation \eqref{eq: original stoch ev eq} is given by
\begin{equation}\label{eq: skeleton eq}
  \begin{cases}
  &(u^\psi)'(t)=-A(t,u^\psi(t))+B(t,u^\psi(t))\psi(t), \quad t\in[0,T],\\
  &u^\psi(0)=x.
\end{cases}
\end{equation}
\end{definition}

The main theorem of this paper is as follows.

\begin{theorem}\label{th: main LDP theorem}
  Suppose that $(A,B)$ satisfies Assumption \ref{ass: crit var set} and coercivity \eqref{eq: coercivity condition (A,B)}. Let $x\in H$. For $\eps\in(0,1]$, let $Y^\eps$ be the strong solution to
  \begin{equation*}
  \begin{cases}
  &\dd Y^\eps(t)=-A(t,Y^\eps(t))\dd t
  +\sqrt{\eps}B(t,Y^\eps(t))\dd W(t), \quad t\in[0,T], \\
  &u(0)=x,
\end{cases}
\end{equation*}
Then $(Y^\eps)$ satisfies the LDP on $\MR(0,T)$ with rate function $I\colon \MR(0,T)\to [0,+\infty]$ given by
\begin{equation}\label{eq: def rate I}
I(z)=\frac{1}{2}\inf\Big\{\int_0^T\|\psi(s)\|_U^2\dd s : \psi\in L^2(0,T;U), z=u^{\psi}\Big\},
\end{equation}
where $\inf\varnothing\coloneqq +\infty$ and $u^\psi$ is the strong solution to \eqref{eq: skeleton eq}.
\end{theorem}

We have taken $\eps\in (0,1]$ to ensure that $(A,\sqrt{\eps}B)$ satisfies coercivity \eqref{eq: coercivity condition (A,B)}, so that the equation for $Y^\eps$ is well-posed by Theorem \ref{th: original stoch ev eq global well-posedness}.

To have $u^\psi$ appearing in \eqref{eq: def rate I} well-defined, \eqref{eq: skeleton eq} needs to be (globally) well-posed. In Section \ref{sec: skeleton}, we prove that this is the case.
Finally, we recall that the LDP is equivalent to the \emph{Laplace principle} \cite[Def. 1.2.2, Th.\ 1.2.1, Th.\ 1.2.3]{dupuis97}. The weak convergence approach from \cite{budhidupuis01}  offers sufficient conditions for the latter, hence for the LDP. The approach is stated in Subsection \ref{sec: weak convergence approach}, after which we apply it to prove Theorem \ref{th: main LDP theorem} in the remainder of Section \ref{sec: proof LDP}.

\section{Well-posedness of the skeleton equation}\label{sec: skeleton}

Before we turn to large deviations, we prove global well-posedness of the skeleton equation \eqref{eq: skeleton eq} under Assumption \ref{ass: crit var set} and coercivity \eqref{eq: coercivity condition (A,B)}. This is needed, since the solution to \eqref{eq: skeleton eq} appears in the rate function \eqref{eq: def rate I} of the LDP.

Unfortunately, well-posedness cannot be proved at once.
Instead, we first achieve well-posedness of an appropriate linearized version of the skeleton equation in Corollary \ref{cor: skeleton MR linearized crit var setting}, together with a maximal regularity estimate.
Then, we can borrow the strategies from \cite[Chap.\ 18]{HNVWvolume3}, \cite{pruss17addendum}, \cite{pruss18JDE}, \cite{pruss16}.
That is, we use the maximal regularity estimate of Corollary \ref{cor: skeleton MR linearized crit var setting} for the linearized equation in a fixed point argument, yielding existence of a local solution to the skeleton equation in Theorem \ref{th: local well posedness skeleton}. Finally, we extend to a global solution in Theorem \ref{th: global well posedness skeleton}, making use of a blow-up criterion. Uniqueness will be obtained along the way.

\subsection{Linearized skeleton equation}

We consider the following linearization of \eqref{eq: skeleton eq}. We discard the non-linearities $F$ and $G$ and for fixed $w\in L^\infty(0,T;H)$, we consider
\begin{equation}\label{eq: linear skeleton}
\begin{cases}
    &u'(t)+A_0(t,w(t))u(t)-B_0(t,w(t))u(t)\psi(t)=\bar{f}(t)+\bar{g}(t)\psi(t), \\
    &u(0)=x,
\end{cases}
\end{equation}
where $A_0$ and $B_0$ are as in Assumption \ref{ass: crit var set} and $\bar{f}\in L^2(0,T;V^*)$, $\bar{g}\in L^2(0,T;\UH)$.
In this subsection we prove well-posedness of \eqref{eq: linear skeleton} using the method of continuity \cite[Lem.\ 16.2.2]{HNVWvolume3}, together with a suitable maximal regularity estimate. We prove it for more general equations as this does not require any more effort and makes the exposition more transparent. Let us introduce spaces $S$ and $E$ that will be used in the method of continuity.

\begin{definition}\label{def: spaces for method of continuity}
For $T>0$, we let
\[
S\coloneqq L^2(0,T;V^*)+L^1(0,T;H)
\]
be the sum space of the interpolation couple $(L^2(0,T;V^*),L^1(0,T;H))$, where we note that both components embed continuously into the Hausdorff topological vector space $L^1(0,T;V^*)$. The norm on $S$ is given by
\[
\|h\|_S\coloneqq \inf\big\{\|f\|_{L^2(0,T;V^*)}+\|g\|_{L^1(0,T;H)}: h=f+g, f\in L^2(0,T;V^*), g\in L^1(0,T;H \}\big\}.
\]
Note that $S$ is a Banach space \cite[Prop.\ C.1.3]{HNVWvolume1} and $S\into L^1(0,T;V^*)$.
Moreover, we define
\begin{align*}
&E\coloneqq \{u\in \MR(0,T):u \text{ is weakly differentiable, } u'\in S\},\quad \|u\|_E\coloneqq \|u\|_{\MR(0,T)}+\|u'\|_{S}.
\end{align*}
Note that trivially, $E\into \MR(0,T)$.
\end{definition}
Dealing with the sum space $S$ is not standard in part of the literature. However, it is covered excellently in Pardoux' thesis \cite{pardoux}.

The following proposition is a direct consequence of \cite[Th.\ 2.1]{pardoux}.

\begin{proposition}\label{prop: surjectivity L_0}
  Let $\bar{A}\colon [0,T]\to\mathcal{L}(V,V^*)$ be such that
   $[0,T]\to V^*\colon t\mapsto \bar{A}(t)v$ is strongly Borel measurable for all $v\in V$
  and suppose that $a_T\coloneqq \sup_{t\in[0,T]}\|\bar{A}(t)\|_{\mathcal{L}(V,V^*)}<\infty$.
  Suppose that there exists $\theta>0$ such that for all $t\in[0,T]$ and $v\in V$:
  \begin{equation*}
   \<\bar{A}(t)v,v\?\geq \theta \|v\|_V^2.
  \end{equation*}
  Then, for any $h\in S$ and $x\in H$, there exists a unique $u\in E$ satisfying
  \begin{equation}\label{eq: eq bar{A}}
    \begin{cases}
    u'(t)+\bar{A}(t)u(t)=h(t), \quad t\in[0,T],\\
    u(0)=x.
  \end{cases}
  \end{equation}
\end{proposition}

We will need an extension of Proposition \ref{prop: surjectivity L_0} with the coercivity condition replaced by the weaker condition \eqref{eq: coercivity L_1} below. As a preparation, we first prove a maximal regularity estimate.

\begin{lemma}\label{lem: a priori}
Let $\bar{A}\colon [0,T]\to\mathcal{L}(V,V^*)$ be such that
$\bar{A}(\cdot)u(\cdot)\in S$ for any $u\in\MR(0,T)$.
Suppose that there exist $\theta>0$ and $M\in L^1(0,T)$, $M\geq 0$ such that for all $t\in[0,T]$ and $v\in V$:
\begin{equation}\label{eq: coercivity L_1}
\<\bar{A}(t)v,v\?\geq \theta \|v\|_V^2-M(t)\|v\|_H^2.
\end{equation}
Let $h\in S$ and $x\in H$ and suppose that $u\in \MR(0,T)$ is a strong solution to \eqref{eq: eq bar{A}}.
Then
\begin{equation}\label{eq: a priori}
\|u\|_{\MR(0,T)}\leq C_{\theta}\exp(2\|M\|_{L^1(0,T)})\big(\|h\|_S+\|x\|_H\big),
\end{equation}
for a constant $C_{\theta}>0$ depending only on $\theta$.
\end{lemma}
\begin{proof}
Write $h=f+g$ with $f\in L^2(0,T;V^*)$, $g\in L^1(0,T;H)$.
We apply \cite[Th.\ 2.2]{pardoux}. Since $u\in\MR(0,T)$ is a strong solution, we have $u(t)=x+\int_0^t v(s)\dd s$ with $v\coloneqq h(\cdot)-\bar{A}(\cdot)u(\cdot)\in S=L^2(0,T;V^*)+L^1(0,T;H)$.
Hence, the chain rule \eqref{eq: Ito pardoux deterministic} and \eqref{eq: coercivity L_1} yield for all $t\in[0,T]$:
\begin{align}
\|u(t)\|_H^2 &=\|x\|_H^2+2\int_0^t \<h(s),u(s)\?\dd s-2\int_{0}^{t} \<\bar{A}(s)u(s),u(s)\?\dd s\notag\\
&\leq \|x\|_H^2+2\int_0^t \<h(s),u(s)\?\dd s-2\theta\|u\|_{L^2(0,t;V)}^2+2\int_0^t M(s)\|u(s)\|_H^2\dd s.\label{eq: gronwall prep bar{A}}
\end{align}
Note that by Young's inequality, we have for all $s\in[0,t]$:
\begin{align*}
\<h(s),u(s)\?=\<f(s),u(s)\?+\<g(s),u(s)\?
&\leq \|u(s)\|_{V}\|f(s)\|_{V^*}+\|u(s)\|_{H}\|g(s)\|_{H}\\
&\leq \frac{\theta}{2}\|u(s)\|_V^2+\frac{1}{2\theta}\|f(s)\|_{V^*}^2+\sup_{r\in[0,t]}\|u(r)\|_H\|g(s)\|_H.
\end{align*}
Entering this into \eqref{eq: gronwall prep bar{A}} we obtain for all $0\leq t\leq t_1\leq T$:
\begin{align*}
\|u(t)\|_H^2+\theta\|u\|_{L^2(0,t;V)}^2&\leq \|x\|_H^2+\frac{1}{\theta}\|f\|_{L^2(0,T;V^*)}^2 +2\sup_{r\in[0,t_1]}\|u(r)\|_H\|g\|_{L^1(0,T;H)}\\
&\qquad+\int_0^{t_1} 2M(s)\|u(s)\|_H^2\dd s\\
\leq \|x\|_H^2+\frac{1}{\theta}\|f\|_{L^2(0,T;V^*)}^2 &+\frac{1}{2}\sup_{r\in[0,t_1]}\|u(r)\|_H^2+2\|g\|_{L^1(0,T;H)}^2 +\int_0^t 2M(s) \|u(s)\|_H^2  \dd s.
\end{align*}
Hence, taking $\sup_{t\in[0,t_1]}$ in the above and writing $F(t_1)\coloneqq \frac{1}{2}\sup_{t\in[0,t_1]}\big(\|u(t)\|_H^2+\theta\|u\|_{L^2(0,t;V)}\big)$ gives for all $0\leq t_1\leq T$:
\begin{align*}
F(t_1)&\leq 2F(t_1)-\frac{1}{2}\sup_{r\in[0,t_1]}\|u(r)\|_H^2\\
&\leq \|x\|_H^2+\frac{1}{\theta}\|f\|_{L^2(0,T;V^*)}^2 +2\|g\|_{L^1(0,T;H)}^2 +\int_0^{t_1} 2M(s) \|u(s)\|_H^2  \dd s\\
&\leq \|x\|_H^2+\frac{1}{\theta}\|f\|_{L^2(0,T;V^*)}^2 +2\|g\|_{L^1(0,T;H)}^2 +\int_0^{t_1} 4M(s) F(s)  \dd s,
\end{align*}
so by Gronwall's inequality, we obtain
\[
\|u\|_{C([0,T];H)}^2+\theta\|u\|_{L^2(0,T;V)}^2\leq 4F(T)\leq 4\big(\|x\|_H^2+\frac{1}{\theta}\|f\|_{L^2(0,T;V^*)}^2 +2\|g\|_{L^1(0,T;H)}^2\big)\exp(4\|M\|_{L^1(0,T)}).
\]
Thus
\begin{align*}
  \|u\|_{\MR(0,T)}^2 &\leq (1\vee\theta^{-1})4\big(\|x\|_H^2+(\theta^{-1}\vee2)(\|f\|_{L^2(0,T;V^*)} + \|g\|_{L^1(0,T;H)})^2\big)\exp(4\|M\|_{L^1(0,T)}).
\end{align*}
Since $f$ and $g$ with $h=f+g$ were arbitrary, taking the infimum over $\{(f,g)\in L^2(0,T;V^*)\times L^1(0,T;H):h=f+g\}$ gives
\begin{align*}
  \|u\|_{\MR(0,T)}^2 &\leq C_{\theta}^2  \big(\|x\|_H^2+\|h\|_{S}^2\big)\exp(4\|M\|_{L^1(0,T)}).
\end{align*}
where $C_\theta\coloneqq \big(4(1\vee\theta^{-1})(\theta^{-1}\vee2)\big)^{\frac{1}{2}}$. Taking square roots on both sides yields \eqref{eq: a priori}.
\end{proof}

We now prove Proposition \ref{prop: surjectivity L_0} under the weaker coercivity \eqref{eq: coercivity L_1}.

\begin{theorem}\label{th: well-posed with M in L^1}
Let $\bar{A}\colon [0,T]\to\mathcal{L}(V,V^*)$ and suppose that
for all $u\in\MR(0,T)$:
  \begin{align}\label{eq: bar{A} bdd MR(0,T) to S}
  \bar{A}(\cdot)u(\cdot)\in S, \qquad \|\bar{A}(\cdot)u(\cdot)\|_S\leq \alpha \|u\|_{\MR(0,T)},
  \end{align}
  for some constant $\alpha>0$ independent of $u$.
  Suppose that coercivity \eqref{eq: coercivity L_1} is satisfied for some $\theta>0$ and $M\in L^1(0,T)$. Then for any $h\in S$, there exists a unique strong solution  $u\in \MR(0,T)$ to \eqref{eq: eq bar{A}}. Moreover, the estimate  \eqref{eq: a priori} holds.
\end{theorem}
\begin{proof}
We use the method of continuity \cite[Lem.\ 16.2.2]{HNVWvolume3}.
Define $A_0\in\mathcal{L}(V,V^*)$ by $A_0v\coloneqq \theta(\cdot,v)_V$.
For $\lambda\in[0,1]$, put
\begin{align*}
&A_\lambda\colon [0,T]\to\mathcal{L}(V,V^*)\colon t\mapsto(1-\lambda){A}_0+\lambda\bar{A}(t), \\
& L_\lambda\colon E\to S\times H\colon \big(L_\lambda u\big)\coloneqq  \big(u'(\cdot)+A_\lambda(\cdot)(u(\cdot)),u(0)\big).
\end{align*}
Clearly, $L_\lambda$ is linear. We show that $L_\lambda\in \mathcal{L}(E,S\times H)$ and that
$[0,1]\to \mathcal{L}(E,S\times H)\colon \lambda\mapsto L_\lambda$ is continuous.

Let $u\in E$ be arbitrary. For all $t\in[0,T]$ we have $A_0u(t)=\theta\<\cdot,u(t)\?\in V^*$, so by the Riesz isomorphism, $\|A_0u(t)\|_{V^*}=\theta\|u(t)\|_V$. Since $u\in L^2(0,T;V)$, it follows that $A_0u(\cdot)\in L^2(0,T;V^*)\subset S$ and
\begin{align*}
\|{A}_0 u(\cdot)\|_S&\leq \|{A}_0 u(\cdot)\|_{L^2(0,T;V^*)}
=\theta\| u\|_{L^2(0,T;V)}\leq \theta\| u\|_{\MR(0,T)}
\end{align*}
Combining with \eqref{eq: bar{A} bdd MR(0,T) to S} gives ${A}_\lambda(\cdot)u(\cdot)\in S$ and
\begin{equation}\label{eq: est A_lambda}
\|{A}_\lambda(\cdot)u(\cdot)\|_S\leq (1-\lambda)\|A_0u(\cdot)\|_S+\lambda\|\bar{A}(\cdot)u(\cdot)\|_S\leq ( \theta+ \alpha)\|u\|_{\MR(0,T)}.
\end{equation}
Note that $u'\in S$ and $\|u'\|_S\leq\|u\|_E$ by definition of $E$. Moreover, $E\into \MR(0,T)$, thus
\begin{align*}
\|L_\lambda u\|_{S\times H}&\leq \|u'\|_S+\|{A}_\lambda(\cdot)u(\cdot)\|_S+\|u(0)\|_H\\
&\leq \|u\|_E+( \theta+ \alpha)\|u\|_{\MR(0,T)}+\|u\|_{C([0,T];H)}\\
&\leq (2+\theta+\alpha)\|u\|_E,
\end{align*}
proving $L_\lambda\in \mathcal{L}(E,S\times H)$. Moreover, we have for any $\lambda,\mu\in [0,1]$ and $u\in E$:
\begin{align*}
\|(L_\lambda-L_\mu) u\|_{S\times H}&=\|\big((\mu-\lambda)A_0u(\cdot)+(\lambda-\mu)\bar{A}(\cdot)u(\cdot),0\big)\|_{S\times H}\\
&\leq |\mu-\lambda|\|A_0u(\cdot)\|_S+|\lambda-\mu|\|\bar{A}(\cdot)u(\cdot)\|_S\\
&\leq |\mu-\lambda|(\theta+\alpha)\|u\|_{\MR(0,T)}\\
&\leq |\mu-\lambda|(\theta+\alpha)\|u\|_E,
\end{align*}
i.e. $\|L_\lambda-L_\mu\|_{\mathcal{L}(E,S\times H)}\leq |\mu-\lambda|(\theta+\alpha)$. Thus $\lambda\mapsto L_\lambda$ is (Lipschitz) continuous.

Next, we verify that $\|u\|_E\leq K\|L_\lambda u\|_{S\times H}$ for some $K>0$ independent of $\lambda$.
Note that $A_\lambda$ satisfies all conditions of Lemma \ref{lem: a priori}. Coercivity \eqref{eq: coercivity L_1} holds since
\begin{align*}
\<v,A_\lambda(t)v\?=(1-\lambda)\<v,{A}_0v\?+\lambda\<v,\bar{A}(t)v\?&\geq (1-\lambda)\theta\|v\|_V^2+\lambda (\theta \|v\|_V^2-|M(t)|\|v\|_H^2)\\
&\geq\theta\|v\|_V^2-|M(t)|\|v\|_H^2.
\end{align*}
Thus, by \eqref{eq: a priori} applied to $h\coloneqq u'+A_\lambda u\in S$ and $x\coloneqq u(0)$:
\[
\|u\|_{\MR(0,T)}\leq C_\theta\exp(2\|M\|_{L^1(0,T)})\big(\|h\|_S+\|u(0)\|_H\big)=C\|L_\lambda u\|_{S\times H},
\]
with $C\coloneqq C_\theta\exp(2\|M\|_{L^1(0,T)})$. Together with \eqref{eq: est A_lambda} this gives for all $u\in E$:
\begin{align*}
  \|u\|_E =\|u\|_{\MR(0,T)}+\|u'\|_S &=\|u\|_{\MR(0,T)}+\|h(\cdot)-A_\lambda(\cdot)u(\cdot)\|_S\\
  &\leq \|u\|_{\MR(0,T)}+\|L_\lambda u\|_{S\times H}+\|A_\lambda(\cdot)u(\cdot)\|_S\\
   &\leq (1+\theta+\alpha)\|u\|_{\MR(0,T)}+\|L_\lambda u\|_{S\times H}\\
   &\leq (1+C(1+\theta+\alpha))\|L_\lambda u\|_{S\times H}.
\end{align*}

Finally, note that $L_0\colon E\to S\times H: \big(L_0u\big)=\big(u'(\cdot)+A_0u(\cdot),u(0)\big)$ is surjective. This follows from Proposition \ref{prop: surjectivity L_0} applied to $\bar{A}_0\colon [0,T]\to \mathcal{L}(V,V^*)$ given by $\bar{A}_0(t)v\coloneqq A_0v= \theta(\cdot,v)_{V}$.

All requirements for the method of continuity are fulfilled and we conclude that $L_1$ is surjective, giving existence of strong solutions. The a priori estimate \eqref{eq: a priori} now follows from Lemma \ref{lem: a priori} and proves uniqueness of strong solutions at once, since $\bar{A}(t)$ is linear.
\end{proof}

As promised, a mere application of Theorem \ref{th: well-posed with M in L^1} now gives us the desired well-posedness and maximal regularity estimate for \eqref{eq: linear skeleton}.

\begin{corollary}\label{cor: skeleton MR linearized crit var setting}
Let $A_0$ and $B_0$ satisfy the conditions concerning $A_0,B_0$ in Assumption \ref{ass: crit var set} and let $\psi\in L^2(0,T;U)$.
Let $T>0$ and $w\in L^\infty(0,T;H)$.
Define $\bar{A}\colon [0,T]\to\mathcal{L}(V,V^*)$ by $\bar{A}(t)v\coloneqq  A_0(t,w(t))v-B_0(t,w(t))v\psi(t)$.
Then $\bar{A}$ satisfies all conditions of Theorem \ref{th: well-posed with M in L^1}. Consequently, for any $\bar{f}\in L^2(0,T;V^*)$ and $\bar{g}\in L^2(0,T;\UH)$, there exists a unique strong solution  $u\in \MR(0,T)$ to
\begin{equation*}
\begin{cases}
    &u'(t)+A_0(t,w(t))u(t)-B_0(t,w(t))u(t)\psi(t)=\bar{f}(t)+\bar{g}(t)\psi(t), \\
    &u(0)=x,
\end{cases}
\end{equation*}
Moreover, for any $\tilde{T}\in[0,T]$ there exists a constant $K_{\tilde{T}}>0$ such that
 \begin{equation}\label{def: maxreg}
 \|u\|_{\MR(0,\tilde{T})}\leq K_{\tilde{T}}\left(\|x\|_H+\|\bar{f}\|_{L^2(0,\tilde{T};V^*)} +\|\bar{g}\|_{L^2(0,\tilde{T};\UH)}\right),
 \end{equation}
and $K_{\tilde{T}}$ is non-decreasing in $\tilde{T}$ and depends further only on $T$, $\|w\|_{L^\infty(0,T;H)}$ and $\|\psi\|_{L^2(0,\tilde{T};U)}$.
\end{corollary}
\begin{proof}
Put $n\coloneqq \|w\|_{L^\infty(0,T;H)}$.
Since strong solutions only depend on $\bar{A}$ through an integral, we can fix a strongly measurable, pointwise defined measurable version of $w$ which satisfies $\|w(t)\|_H\leq n$ for all $t\in[0,T]$.
Strong measurability of $\bar{A}(\cdot)u(\cdot)$ is then satisfied if $u\in\MR(0,T)$, see Remark \ref{rem: mble A_0 B_0 F G}. Moreover, Assumption \ref{it:ass3} gives for all $u\in \MR(0,T)$:
\[
\|A_0(\cdot,w(\cdot))u(\cdot)\|_{L^2(0,T;V^*)} \leq C_{n,T}\|u\|_{L^2(0,T;V)}<\infty
\]
and by the Cauchy--Schwarz inequality,
\begin{align*}
\|B_0(\cdot,w(\cdot))u(\cdot)\psi(\cdot)\|_{L^1(0,T;H)}
&\leq  \|B_0(\cdot,w(\cdot))u(\cdot)\|_{L^2(0,T;\UH)}\|\psi\|_{L^2(0,T;U)}  \\
& \leq   C_{n,T}\|u \|_{L^2(0,T;V)} \|\psi\|_{L^2(0,T;U)}<\infty.
\end{align*}
Hence, $\bar{A}(\cdot)u(\cdot)\in S$ and we have
\begin{align*}
 \|\bar{A}(\cdot)u(\cdot)\|_S&\leq  \|A_0(\cdot,w(\cdot))u(\cdot)\|_{L^2(0,T;V^*)}+\|B_0(\cdot,w(\cdot))u(\cdot)\psi(\cdot)\|_{L^1(0,T;H)}\\
 &\leq C_{n,T}(1+\|\psi\|_{L^2(0,T;U))})\|u\|_{L^2(0,T;V)}\\
  &\leq \alpha \|u\|_{\MR(0,T)},
\end{align*}
where $\alpha\coloneqq C_{n,T}(1+\|\psi\|_{L^2(0,T;U))})$.

Furthermore, by Assumption \ref{it:ass2}, we have for all $v\in V$ and $t\in[0,T]$:
    \begin{align*}
    \<\bar{A}(t)v,v\?&=\<A_0(t,w(t))v,v\?-\<B_0(t,w(t))\psi(t)v,v\?\\
    &\geq \<A_0(t,w(t))v,v\?- \frac{1}{2}\nn B_0(t,w(t))v \nn_H^2-\frac{1}{2}\|\psi(t)\|_U^2\|v\|_H^2\\
    &\geq \theta_{n,T}\|v\|_V^2-({M}_{n,T}+\frac{1}{2}\|\psi(t)\|_U^2)\|v\|_H^2,
  \end{align*}
  so coercivity \eqref{eq: coercivity L_1} is satisfied with ${\theta}\coloneqq \theta_{n,T}$ and  ${M}(\cdot)\coloneqq {M}_{n,T}+\frac{1}{2}\|\psi(\cdot)\|_U^2\in L^1(0,T)$.

As before by the Cauchy--Schwarz inequality, $h\coloneqq \bar{f}+\bar{g}\psi \in S$.
Now Theorem \ref{th: well-posed with M in L^1} yields existence of a unique strong solution $u\in\MR(0,T)$ to \eqref{eq: linear skeleton}.
Finally, let $\tilde{T}\in(0,T]$ be arbitrary and
put $\tilde{S}\coloneqq L^2(0,\tilde{T};V^*)+L^1(0,\tilde{T};H)$. We have
\begin{align}\label{eq: est F-norm f}
\|h\|_{\tilde{S}}\leq \|\bar{f}\|_{L^2(0,\tilde{T};V^*)}+\|\bar{g}\psi\|_{L^1(0,\tilde{T};H)}
\leq \|\bar{f}\|_{L^2(0,\tilde{T};V^*)}+\|\bar{g}\|_{L^2(0,\tilde{T};\UH)}\|\psi\|_{L^2(0,\tilde{T};U)}.
\end{align}
As $u|_{[0,\tilde{T}]}$ is a strong solution to \eqref{eq: linear skeleton} on $[0,\tilde{T}]$, \eqref{eq: a priori} and \eqref{eq: est F-norm f} yield
\begin{align*}
 \|u\|_{\MR(0,\tilde{T})}&\leq  C_{\theta}\exp(2\|M\|_{L^1(0,\tilde{T})})\big(\|h\|_{\tilde{S}}+\|x\|_H\big)\\
 &\leq  K_{\tilde{T}}\left(\|x\|_H+\|\bar{f}\|_{L^2(0,\tilde{T};V^*)} +\|\bar{g}\|_{L^2(0,\tilde{T};\UH)}\right),
\end{align*}
where $K_{\tilde{T}}\coloneqq C_{\theta}\exp(2\|M\|_{L^1(0,\tilde{T})})(1\vee\|\psi\|_{L^2(0,\tilde{T};U)})$ is non-decreasing in $\tilde{T}$. Note that apart from $\tilde{T}$, $K_{\tilde{T}}$ only depends on $T$, $n$ and $\|\psi\|_{L^2(0,\tilde{T};U)}$, since these determine $\theta$ and $\|M\|_{L^1(0,\tilde{T})}$.
\end{proof}

\begin{remark}\label{rem: mble A_0 B_0 F G}
In Corollary \ref{cor: skeleton MR linearized crit var setting}, the map $t\mapsto \bar{A}(t)v$ is strongly Borel measurable for any $v\in V$, and even more is true. Assumption \ref{ass: crit var set} assures that we have strong Borel measurability of $A_0(\cdot,w(\cdot))u(\cdot), F(\cdot,u(\cdot))\colon [0,T]\to V^*$ and $B_0(\cdot,w(\cdot))u(\cdot), G(\cdot,u(\cdot))\colon [0,T]\to \UH$, for any $u\in L^0(0,T;V)$ and $w\in L^0(0,T;H)$. This follows from strong measurability of $u\colon [0,T]\to V$ and $w:[0,T]\to H$ and the fact that by Assumption \ref{it:ass3}, $F(t,\cdot), G(t,\cdot)$ are continuous on $V$ and $A_0(t,\cdot)\cdot, B_0(t,\cdot)\cdot$ are continuous on $H\times V$. Moreover, one uses the measurability of Assumption \ref{it:ass1}, separability of $V$, $H$, $V^*$, $\UH$ and continuity of $V\into H\into V^*$.
\end{remark}

\subsection{Local well-posedness}\label{sec: local skeleton}

From now on, we let $\psi\in L^2(0,T;U)$ be arbitrary but fixed.
Using Corollary \ref{cor: skeleton MR linearized crit var setting}, we will prove local well-posedness of the actual skeleton equation \eqref{eq: skeleton eq}. Local well-posedness is established in Theorem \ref{th: local well posedness skeleton}. Its proof and preparatory lemma's are analogous to \cite[\S 18.2]{HNVWvolume3}, which was inspired by \cite{pruss17addendum} and \cite{pruss18JDE}.

The skeleton equation does not fit in the setting of \cite{HNVWvolume3}, \cite{pruss17addendum} or \cite{pruss18JDE}, for the reason that we only have $L^1$-( instead of $L^2$-)integrability of the term $B(\cdot,u^\psi(\cdot))\psi(\cdot)$ in \eqref{eq: skeleton eq}. Besides that, our maximal regularity space $\MR(0,T)=C([0,T];H)\cap L^2(0,T;V)$ is different.

When no confusion can arise, we omit the time input in our notations for brevity. For example, for $u,v\in\MR(0,T)$ we denote by $A_0(u)v$ and $B_0(u)v\psi$ the maps $t\mapsto A_0(t,u(t))v(t)$ and $t\mapsto B_0(t,u_0)v(t)\psi(t)$ respectively and similarly for $F(u)$ and $G(u)\psi$. We define the following $V^*$-valued mappings:
\begin{equation*}
  \tilde{A}(u)v\coloneqq A_0(u)v-B_0(u)v\psi,\qquad \tilde{F}(u)=F(u)+f+(G(u)+g)\psi.
\end{equation*}

\begin{theorem}[Local well-posedness of the skeleton equation]\label{th: local well posedness skeleton}
Suppose that $(A,B)$ satisfies Assumption \ref{ass: crit var set}.
Let $u_0\in H$ be fixed. Then there exist $\tilde{T},\eps>0$ such that for each $v_0\in B_H(u_0,\eps)$,
there exists a unique strong solution $u_{v_0}\in \MR(0,\tilde{T})$ to
\begin{equation}\label{eq: local well-posed}
  \begin{cases}
    &u'+\tilde{A}(u)u=\tilde{F}(u)\qquad \text{ on } [0,\tilde{T}] \\
    &u(0)=v_0.
  \end{cases}
\end{equation}
Moreover, there exists a constant $C>0$ such that for all $v_0,w_0\in B_H(u_0,\eps)$:
\begin{equation}\label{eq: ct dependence local well-posed}
  \|u_{v_0}-u_{w_0}\|_{\MR(0,\tilde{T})}\leq C\|v_0-w_0\|_H.
\end{equation}
\end{theorem}

Theorem \ref{th: local well posedness skeleton} will be proved using the Banach fixed point theorem, applied to the map $\Psi_{v_0}\colon \MR(0,\tilde{T})\to \MR(0,\tilde{T})$ defined by $\Psi_{v_0}(v)\coloneqq u$, where $u$ is the unique strong solution to
\begin{equation}\label{eq:contraction map def}
  \begin{cases}
    &u'+\tilde{A}(u_0)u=(\tilde{A}(u_0)-\tilde{A}(v))v+\tilde{F}(v)\quad \text{on } [0,\tilde{T}],\\
    &u(0)=v_0.
  \end{cases}
\end{equation}
Note that $u\in \MR(0,\tilde{T})$ is a strong solution to \eqref{eq: local well-posed} if and only if $\Psi_{v_0}(u)=u$.

Our first task is to prove that $\Psi_{v_0}$ is well-defined, i.e. \eqref{eq:contraction map def} is well-posed. By  Corollary \ref{cor: skeleton MR linearized crit var setting} ($w(t)\coloneqq u_0$) it suffices to show that
\begin{equation}\label{eq:def tilde f tilde g}
\tilde{f}\coloneqq (A_0(u_0)-A_0(v))v+F(v)+f,\quad \tilde{g}\coloneqq (B_0(u_0)-B_0(v))v+G(v)+g
\end{equation}
satisfy $\tilde{f}\in L^2(0,\tilde{T};V^*)$ and $\tilde{g}\in L^2(0,\tilde{T};\UH)$. The latter will be ascertained by the following lemma, which will also be used later on in Section \ref{sec: proof LDP}.

\begin{lemma}\label{lem: embeddings rho_j}
  Let $\rho_j\geq 0, \beta_j\in(\frac{1}{2},1)$ be such that $(2\beta_j-1)(\rho_j+1)\leq 1$. Let $V_{\beta_j}=[V^*,V]_{\beta_j}$ be the complex interpolation space with norm $\|\cdot\|_{\beta_j}\coloneqq \|\cdot\|_{V_{\beta_j}}$. Then, for any $T>0$:
  \begin{enumerate}[start, label=\emph{(\roman*)},ref=\textup{(\roman*)}]
    \item \label{it:emb1} $\iota_{\scriptscriptstyle j,T}:{\MR(0,T)}\into {L^{2(\rho_j+1)}(0,T;V_{\beta_j})}$. The embedding satisfies $\|\iota_{\scriptscriptstyle  j,T}\|\leq M_T^j$ with $M_T^j\in\R_+$ non-decreasing in $T$.
  \end{enumerate}
  Suppose that $(A,B)$ satisfies Assumption \ref{ass: crit var set}.
  Let $n\in\R_+$ and $T>0$. For $C_{n,T}$ the constant from Assumption \ref{it:ass3} (non-decreasing in $n$ and $T$), it holds that
  \begin{enumerate}[resume,label=\emph{(\roman*)},ref=\textup{(\roman*)}]
  \item \label{it:emb2} for all $u\in C([0,T];H)$, $w\in L^2(0,T;V)$ with $\|u\|_{C([0,T];H)}\leq n$:
  \begin{equation*}\label{eq: estimate A_0 B_0}
  \|A_0(u)w\|_{L^2(0,T;V^*)}\vee \|B_0(u)w\|_{L^2(0,T;\UH)}\leq C_{n,T}\|w\|_{L^2(0,T;V)},
  \end{equation*}
  \item \label{it:emb3} for all $u,v\in C([0,T];H)$, $w\in L^2(0,T;V)$ with $\|u\|_{C([0,T];H)},\|v\|_{C([0,T];H)}\leq n$:
  \begin{align*}\label{eq: estimate A_0 B_0 difference}
  &\|(A_0(u)-A_0(v))w\|_{L^2(0,T;V^*)}\vee \|(B_0(u)-B_0(v))w\|_{L^2(0,T;\UH)}\\
  &\leq C_{n,T}\Big(\int_0^T\|u(s)-v(s)\|_H^2\|w(s)\|_{V}^{2}\dd s\Big)^{\frac{1}{2}}\leq C_{n,T}\|u-v\|_{C([0,T];H)}\|w\|_{L^2(0,T;V)}.
  \end{align*}
  \end{enumerate}
  Moreover, there exists a constant $\tilde{C}_{n,T}$ non-decreasing in $T$ such that
  \begin{enumerate}[resume,label=\emph{(\roman*)},ref=\textup{(\roman*)}]
    \item \label{it:emb4} for all $u\in\MR(0,T)$ with $\|u\|_{C([0,T];H)}\leq n$:
  \begin{equation*}\label{eq: estimate F G}
  \|F(u)\|_{L^2(0,T;V^*)}\vee \|G(u)\|_{L^2(0,T;\UH)}\leq \tilde{C}_{n,T}(1+\|u\|_{L^2(0,T;V)}).
  \end{equation*}
  \end{enumerate}
  Lastly, for each $\sigma>0$ there exists a constant $C_{n,T,\sigma}$ non-decreasing in $T$ such that
  \begin{enumerate}[resume,label=\emph{(\roman*)},ref=\textup{(\roman*)}]
  \item \label{it:emb5} for all $u,v\in\MR(0,T)$ with $\|u\|_{C([0,T];H)},\|v\|_{C([0,T];H)}\leq n$:
  \begin{align*}\label{eq: estimate F G difference}
  \|&F(u)-F(v)\|_{L^2(0,T;V^*)}^2\vee \|G(u)-G(v)\|_{L^2(0,T;\UH)}^2\\
  &\leq C_{n,T,\sigma}\int_0^t(1+\|u(s)\|_V^2+\|v(s)\|_V^2)\|u(s)-v(s)\|_H^2\dd s+\sigma C_{n,T}^2\|u-v\|_{L^2(0,T;V)}^2.
  \end{align*}
  \end{enumerate}
\end{lemma}
\begin{proof}
\ref{it:emb1}: By the interpolation estimate \eqref{eq:interpolation estimate}, we have for any $u\in \MR(0,T)$:
\begin{align*}
\int_0^{T} \|u(t)\|_{\beta_j}^{2(\rho_j+1)}\dd t&\leq K \int_0^{T} \|u(t)\|_{H}^{2(\rho_j+1)(2-2\beta_j)}\|u(t)\|_{V}^{2(\rho_j+1)(2\beta_j-1)}\dd t\\
&\leq K \|u\|_{C([0,T];H)}^{2(\rho_j+1)(2-2\beta_j)}\int_0^{T}\|u(t)\|_{V}^{2(\rho_j+1)(2\beta_j-1)}\dd t\\
&\leq K \|u\|_{C([0,T];H)}^{2(\rho_j+1)(2-2\beta_j)}\|1\|_{L^{p_j'}(0,T)}\|\|u\|_V^{2(\rho_j+1)(2\beta_j-1)}\|_{L^{p_j}(0,T)}\\
&\leq K \|u\|_{C([0,T];H)}^{2(\rho_j+1)(2-2\beta_j)}(1\vee T)^{\frac{p_j-1}{p_j}}\|u\|_{L^2(0,T;V)}^{2(\rho_j+1)(2\beta_j-1)},
\end{align*}
where we applied H\"older's inequality for each $j$ with $p_j\coloneqq \frac{1}{(\rho_j+1)(2\beta_j-1)}\in[1,\infty)$, $p_j'\coloneqq \frac{p_j}{p_j-1}\in[1,\infty]$ and included the maximum with 1 to cover the case $p_j'=\infty$.
We conclude that
\begin{align*}
  \|u\|_{L^{2(\rho_j+1)}(0,T;V_{\beta_j})}&\leq M_{T}^j\|u\|_{C([0,T];H)}^{(2-2\beta_j)}\|u\|_{L^2(0,T;V)}^{(2\beta_j-1)}\\
  &\leq M_{T}^j \left((2-2\beta_j)\|u\|_{C([0,T];H)}+(2\beta_j-1)\|u\|_{L^2(0,T;V)}\right)\\
  &\leq M_{T}^j \|u\|_{\MR(0,T)},
\end{align*}
where $M_{T}^j\in\R_+$ is non-decreasing in $T$. We used  Young's inequality and the fact that $\beta_j\in(\frac{1}{2},1)$.

In \ref{it:emb2}-\ref{it:emb5}, note that strong measurability of $A_0(\cdot,u(\cdot))w(\cdot)$, $F(\cdot,u(\cdot))$, $B_0(\cdot,u(\cdot))w(\cdot)$ and $G(\cdot,u(\cdot))$ holds, as was mentioned in Remark \ref{rem: mble A_0 B_0 F G}.
Moreover, by symmetry in Assumption \ref{ass: crit var set}, $B_0$ and $G$ can be estimated in the same way as $A_0$ and $F$. We provide the estimates for the latter. 

Assumption \ref{it:ass3}  immediately yields \ref{it:emb2} and \ref{it:emb3}.

For
\ref{it:emb4}: by Assumption \ref{it:ass3} and \eqref{eq:interpolation estimate}, we have pointwise in $t\in[0,T]$:
\begin{align*}
\|F(u)\|_{V^*}\leq C_{n,T}\sum_{j=1}^{m_F}(1+\|u\|_{\beta_j}^{\rho_j+1})&\leq C_{n,T}\sum_{j=1}^{m_F}(1+(Kn^{2-2\beta_j})^{\rho_j+1}\|u\|_{V}^{(2\beta_j-1)(\rho_j+1)})\\
&\leq {C}_{n,T}\sum_{j=1}^{m_F}(1+C_n(1+\|u\|_{V})) \\
&\leq \bar{C}_{n,T}(1+\|u\|_{V}),
\end{align*}
where we used that $(2\beta_j-1)(\rho_j+1)\leq 1$ and put $C_n\coloneqq  \max_{j=1,\ldots,m_F} (Kn^{2-2\beta_j})^{\rho_j+1}<\infty$ and $\bar{C}_{n,T}\coloneqq m_F C_{n,T}(1+C_n)$.
Thus
\[
\|F(u)\|_{L^2(0,T;V^*)}\leq \bar{C}_{n,T}(T^{\frac{1}{2}}+\|u\|_{L^2(0,T;V)}) \leq \tilde{C}_{n,T}(1+\|u\|_{L^2(0,T;V)}),
\]
with $\tilde{C}_{n,T}=\bar{C}_{n,T}(T^{\frac{1}{2}}\vee 1)$.
Since ${C}_{n,T}$ is non-decreasing in $T$, the same holds for $\bar{C}_{n,T}$ and $\tilde{C}_{n,T}$.

For \ref{it:emb5}: the following estimates can be found in the proof of \cite[Prop.\ 4.5]{AV22variational}.
By Assumption \ref{it:ass3} we have pointwise in $t\in[0,T]$:
\begin{align}\label{eq: gronwall prep w_n I_2^n f_n}
\|F(u)-F(v) \|_{V^*}\leq C_{n,T}\sum_{j=1}^{m_F}\left(1+\|u\|_{\beta_j}^{\rho_j}+\|v\|_{\beta_j}^{\rho_j}\right)\|u-v\|_{\beta_j}.
\end{align}
By the interpolation estimate \eqref{eq:interpolation estimate} and Young's inequality (with powers $\frac{1}{2-2\beta}$ and $\frac{1}{2\beta-1}$), we have for all $y,z\in V$, $\beta\in(\frac{1}{2},1)$, $\rho\geq0$ with $(2\beta-1)(\rho+1)\leq 1$ and for all $\sigma>0$:
\begin{align}\label{eq: gronwall prep w_n I_2^n interpol}
  \|y\|_{\beta}^{\rho}\|z\|_\beta &\leq \left(K^{\rho+1}\|y\|_H^{(2-2\beta)\rho}\|y\|_V^{(2\beta-1)\rho}\|z\|_H^{2-2\beta}\right)\|z\|_V^{2\beta-1}\notag \\
 &\leq \sigma^{-\frac{2\beta-1}{2-2\beta}}(2-2\beta) K^{\frac{\rho+1}{2-2\beta}}\|y\|_H^{\rho}\|y\|_V^{\frac{(2\beta-1)\rho}{2-2\beta}}\|z\|_H+\sigma(2\beta-1)\|z\|_V \notag\\
 &\leq \sigma^{-\frac{2\beta-1}{2-2\beta}} K^{\frac{\rho+1}{2-2\beta}}\|y\|_H^{\rho}(1+\|y\|_V)\|z\|_H+\sigma\|z\|_V \notag\\
 &\leq M_{\sigma,\beta,\rho} \|y\|_H^{\rho}(1+\|y\|_V)\|z\|_H+\sigma\|z\|_V,
\end{align}
where $M_{\sigma,\beta,\rho}> 0$ is a constant depending only on $\sigma$, $\beta$ and $\rho$ and we let $0^0=1$. In the above we used that $a\coloneqq \frac{(2\beta-1)\rho}{2-2\beta} \in[0,1]$, hence $x^a\leq 1+x$ for $x\geq 0$.
For $j\in\{1,\ldots,m_F\}$, application of \eqref{eq: gronwall prep w_n I_2^n interpol} gives pointwise in $t\in[0,T]$:
\begin{align}\label{eq: gronwall prep w_n I_2^n f_n second ter,eq: gronwall prep w_n I_2^n f_n second term}
\Big(1+\|&u\|_{\beta_j}^{\rho_j}+\|v\|_{\beta_j}^{\rho_j}\Big)\|u-v\|_{\beta_j}\notag\\
\leq &\Big(M_{\sigma,\beta_j,0}+M_{\sigma,\beta_j,\rho_j}\|u\|_{H}^{\rho_j}(1+\|u\|_V) +M_{\sigma,\beta_j,\rho_j}\|v\|_{H}^{\rho_j}(1+\|v\|_V)\Big)\|u-v\|_{H}+3\sigma\|u-v\|_V\notag\\
&\quad\leq M_{\sigma}(1+\|u\|_V+\|v\|_V)\|u-v\|_{H}+3\sigma\|u-v\|_V,
\end{align}
with $M_\sigma\coloneqq \max_{j=1,\ldots, m_F}(M_{\sigma,\beta_j,0}+2M_{\sigma,\beta_j,\rho_j}N^{\rho_j})<\infty$.
Now, \eqref{eq: gronwall prep w_n I_2^n f_n} and \eqref{eq: gronwall prep w_n I_2^n f_n second ter,eq: gronwall prep w_n I_2^n f_n second term} imply
\begin{align*}
\|F(u)-F(v)\|_{V^*}&\leq C_{n,T}m_F M_{\sigma}(1+\|u\|_V+\|v\|_V)\|u-v\|_{H}+3\sigma C_{n,T}m_F \|u-v\|_V
\end{align*}
and hence, applying $(x_1+\ldots+x_d)^2\leq d(x_1^2+\ldots+x_d^2)$ with $d=2,3$:
\begin{align*}
  \|F(u)-F(v)\|_{L^2(0,T;V^*)}^2&\leq \bar{C}_{n,T,\sigma}\int_0^T(1+\|u(t)\|_V^2+\|v(t)\|_V^2)\|u(t)-v(t)\|_{H}^2\dd t\\
  &\qquad\qquad+2(3\sigma C_{n,T}m_F)^2 \int_0^T\|u(t)-v(t)\|_V^2\dd t,
\end{align*}
with $\bar{C}_{N,T,\sigma}=6(C_{n,T}m_F M_{\sigma})^2$. Since $C_{n,T}$ is non-decreasing in $T$, the same holds for $\bar{C}_{n,T,\sigma}$.
Substituting $\sigma=18\bar{\sigma}^2 m_F^2$, ${C}_{n,T,\sigma}\coloneqq \bar{C}_{n,T,\bar{\sigma}}$ now yields \ref{it:emb5}.
\end{proof}

\begin{remark}
Lemma \ref{lem: embeddings rho_j} yields $A(\cdot,u(\cdot))\in L^2(0,T;V^*)$ and $B(\cdot,u(\cdot))\in L^2(0,T;\UH)$  a.s.\ if $u\in \MR(0,T)$ a.s. Hence, under Assumption \ref{ass: crit var set}, this condition is redundant in the definition of a strong solution (Definition \ref{def:strong sol}).
\end{remark}

From Lemma \ref{lem: embeddings rho_j}, we see  that $\tilde{f}$ and $\tilde{g}$ defined by \eqref{eq:def tilde f tilde g} lie in $L^2(0,\tilde{T};V^*)$ and $L^2(0,\tilde{T};\UH)$ respectively,
for any $\tilde{T}>0$ and $v\in\MR(0,\tilde{T})$ (put $n=\|u_0\|_{H}\vee\|v\|_{C([0,\tilde{T}];H)}$ and apply \ref{it:emb3} and \ref{it:emb5}). Thus Corollary \ref{cor: skeleton MR linearized crit var setting} gives that \eqref{eq:contraction map def} is well-posed, i.e. $\Psi_{v_0}$ is well-defined.

Our next concern is to prove that $\Psi_{v_0}$ is contractive on a suitable smaller subspace of $\MR(0,T)$. To define this subspace, let us introduce some notations. For what follows, we fix an arbitrary $u_0\in H$ and $T>0$.
For $v_0\in H$, we let $z_{v_0}\in \MR(0,T)$ be the \emph{reference solution}, defined as the unique strong solution to the linear problem
\begin{equation}\label{eq:def z_v_0}
  \begin{cases}
    &z'+\tilde{A}(u_0)z=0\quad \text{ on } [0,T], \\
    &z(0)=v_0.
  \end{cases}
\end{equation}
Well-posedness holds by Corollary \ref{cor: skeleton MR linearized crit var setting}.
Note that $z_{u_0}(0)=u_0$ and $z_{u_0}\in \MR(0,T)$, so there exists a $T_1\in(0,T]$ such that
\begin{equation}\label{eq: def T_1}
  \|z_{u_0}-u_0\|_{C([0,T_1];H)}\leq \frac{1}{3}.
\end{equation}
We fix such a  $T_1$.
Finally, for $v_0\in H$, $r>0$ and $\tilde{T}\in[0,T]$, we define
\begin{equation}\label{eq:def Z_rT}
Z_{r,\tilde{T}}(v_0)\coloneqq \{v\in \MR(0,\tilde{T}):v(0)=v_0, \|v-z_{u_0}\|_{\MR(0,\tilde{T})}\leq r\}.
\end{equation}
Note that $Z_{r,\tilde{T}}(v_0)$ is closed in $\MR(0,\tilde{T})$, hence complete.
Eventually, we will find that $\Psi_{v_0}$ is contractive on some $Z_{r,\tilde{T}}(v_0)$.
Several crucial estimates will be gathered in the next lemma's.

\begin{lemma}\label{lem:eps T_1}
There exist $\eps_1,r_1>0$ such that for all $\eps\in(0,\eps_1]$, $r\in(0,r_1]$, $\tilde{T}\in(0,T_1]$, $v_0\in B_H(u_0,\eps)$ and $v\in Z_{r,\tilde{T}}(v_0)$ it holds that  $\|v-u_0\|_{C([0,\tilde{T}];H)}\leq 1$.
\end{lemma}
\begin{proof}
Let $\eps,r>0$, $\tilde{T}\in(0,T_1]$ and let $v\in Z_{r,\tilde{T}}(v_0)$.
We have
  \begin{align*}
     \|v-z_{v_0}\|_{\MR(0,\tilde{T})} &\leq \|v-z_{u_0}\|_{\MR(0,\tilde{T})} + \|z_{u_0}-z_{v_0}\|_{\MR(0,T_1)}\\
    &= \|v-z_{u_0}\|_{\MR(0,\tilde{T})} + \|z_{u_0-v_0}\|_{\MR(0,T_1)}\\
    &\leq r + K_{T_1}\|{u_0-v_0}\|_{H},
  \end{align*}
  where the last inequality follows from the definition of $Z_{r,\tilde{T}}(v_0)$ and \eqref{def: maxreg}.
  Therefore,
  \begin{align*}
    \|v-u_0\|_{C([0,\tilde{T}];H)} &\leq \|v-z_{v_0}\|_{\MR(0,\tilde{T})}+\|z_{v_0}-z_{u_0}\|_{\MR(0,T_1)}+\|z_{u_0}-u_0\|_{C([0,T_1];H)} \\
    &\leq \left(r + K_{T_1}\|{u_0-v_0}\|_{H}\right) +  K_{T_1}\|{u_0-v_0}\|_{H}+\frac{1}{3}\\
    &\leq r+2 K_{T_1}\eps+\frac{1}{3},
  \end{align*}
  whenever $v_0\in B_H(u_0,\eps)$. Taking $r_1= \frac{1}{3}$ and $\eps_1=(6 K_{T_1})^{-1}$, the claim is proved.
\end{proof}

The next lemma is analogous to \cite[Lem.\ 18.2.10]{HNVWvolume3}.

\begin{lemma}\label{lem:tilde f tilde g}
Let $u_0\in H$ and suppose that $(A,B)$ satisfies Assumption \ref{ass: crit var set}.
Let $\tilde{f}\in L^2(0,\tilde{T};V^*)$ and $\tilde{g}\in L^2(0,\tilde{T};\UH)$ be defined by \eqref{eq:def tilde f tilde g}.
For $\eps_1$ and $r_1$ from Lemma \ref{lem:eps T_1},  the following estimates hold for any  $\tilde{T}\in(0,T_1]$, $\eps\in(0,\eps_1]$, $r\in(0,r_1]$, $v_0\in B_H(u_0,\eps)$, $v\in Z_{r,\tilde{T}}(v_0)$  and $\sigma>0$:
\begin{align}
&\| \tilde{f}\|_{L^2(0,\tilde{T};V^*)}\vee \|\tilde{g}\|_{L^2(0,\tilde{T};\UH)}\leq \alpha_{T_1}(\tilde{T})+\beta_{T_1,\sigma}(\tilde{T},r)r+\sigma r, 
\label{eq: estimate tilde f sigma}
\end{align}
  where $\alpha_{T_1}(\tilde{T}),\beta_{T_1,\sigma}(\tilde{T},r)\downarrow 0$ as $\tilde{T},r\downarrow0$ and $\alpha_{T_1}(\tilde{T})$ and $\beta_{T_1,\sigma}(\tilde{T},r)$ are independent of $v_0$ and $v$.
\end{lemma}
\begin{proof}
Let  $v_0\in B_H(u_0,\eps)$, $\tilde{T}\in(0,T_1]$ and $v\in Z_{r,\tilde{T}}(v_0)$ be arbitrary. We estimate each term appearing in the definition of $\tilde{f}$. By Lemma \ref{lem:eps T_1},
\begin{equation}\label{eq: v estimate}
\|v\|_{C([0,\tilde{T}];H)}\leq \|v-u_0\|_{C([0,T_1];H)}+\|u_0\|_H\leq \|u_0\|_H+1.
\end{equation}
Putting $C_{T_1}\coloneqq C_{\|u_0\|_H+1,T_1}$, Lemma \ref{lem: embeddings rho_j}\ref{it:emb3} gives
\begin{align}
\|A_0(u_0)v&-A_0(v)v\|_{L^2(0,\tilde{T};V^*)}\leq C_{T_1}\|u_0-v\|_{C([0,\tilde{T}];H)}\|v\|_{L^2(0,\tilde{T};V)}\notag\\
&\leq C_{T_1}\left(\|u_0-z_{u_0}\|_{C([0,\tilde{T}];H)}+\|z_{u_0}-v\|_{C([0,\tilde{T}];H)}\right)\left(\|v-z_{u_0}\|_{L^2(0,\tilde{T};V)}+\|z_{u_0}\|_{L^2(0,\tilde{T};V)}\right)\notag\\
&\leq C_{T_1}(\alpha(\tilde{T})+r)^2  \leq 2C_{T_1}(\alpha(\tilde{T})^2+r^2)\label{eq: estimate A_0}
\end{align}
with
\[
\alpha(\tilde{T})\coloneqq \|u_0-z_{u_0}\|_{C([0,\tilde{T}];H)}\vee\|z_{u_0}\|_{L^2(0,\tilde{T};V)}.
\]
Note that $\alpha(\tilde{T})\downarrow 0$ as $\tilde{T}\downarrow 0$, since $z_{u_0}\in C([0,T_1];H)\cap L^2(0,T_1;V)$ and $z_{u_0}(0)=u_0$.

We turn to the term $F(v)$ appearing in $\tilde{f}$.
By \eqref{eq: def T_1},
\begin{align}\label{eq: z_u_0 estimate}
& \|z_{u_0}\|_{C([0,\tilde{T}];H)}\leq \|z_{u_0}-u_0\|_{C([0,T_1];H)}+\|u_0\|_H<1+\|u_0\|_H.
\end{align}
Now we apply  Lemma \ref{lem: embeddings rho_j}\ref{it:emb5} with  $\tilde{\sigma}\coloneqq \sigma^2{C}_{\|u_0\|_H+1,T_1}^{-2}$ and let  
 $\tilde{C}_{T_1,\sigma}\coloneqq C_{\|u_0\|_H+1,T_1,\tilde{\sigma}}$ denote the constant of Lemma \ref{lem: embeddings rho_j}\ref{it:emb5} corresponding to $\tilde{\sigma}$. 
Recalling \eqref{eq: z_u_0 estimate} and \eqref{eq: v estimate}, we obtain
\begin{align*}
  \|F(v)&\|_{L^2(0,\tilde{T};V^*)}\leq \|F(v)-F(z_{u_0})\|_{L^2(0,\tilde{T},V^*)}+\|F(z_{u_0})\|_{L^2(0,\tilde{T},V^*)}\notag\\
  &\leq \Big(\tilde{C}_{T_1,\sigma}\int_0^{\tilde{T}}(1+\|v\|_V^2+\|z_{u_0}\|_V^2)\|v-z_{u_0}\|_H^2\dd s\Big)^{\frac{1}{2}} +\sigma 
  \|v-z_{u_0}\|_{L^2(0,\tilde{T};V)}\notag\\
  &\qquad\qquad+\|F(z_{u_0})\|_{L^2(0,\tilde{T},V^*)}\notag\\
  &\leq \Big(\tilde{C}_{T_1,\sigma}\int_0^{\tilde{T}}(1+\|v\|_V^2+\|z_{u_0}\|_V^2)r^2\dd s\Big)^{\frac{1}{2}}
  +\sigma r 
  +\|F(z_{u_0})\|_{L^2(0,\tilde{T},V^*)}\notag\\
&\leq r\tilde{C}_{T_1,\sigma}^{\frac{1}{2}}(\tilde{T}^{\frac{1}{2}}+\|v\|_{L^2(0,\tilde{T};V)}+\|z_{u_0}\|_{L^2(0,\tilde{T};V)})+\sigma r 
+\|F(z_{u_0})\|_{L^2(0,\tilde{T},V^*)} \notag\\
&\leq r\tilde{C}_{T_1,\sigma}^{\frac{1}{2}}(\tilde{T}^{\frac{1}{2}}+\|v-z_{u_0}\|_{L^2(0,\tilde{T};V)}
+2\|z_{u_0}\|_{L^2(0,\tilde{T};V)})+\sigma r 
+\|F(z_{u_0})\|_{L^2(0,\tilde{T};V^*)} \notag\\
&\leq r\tilde{C}_{T_1,\sigma}^{\frac{1}{2}}(\tilde{T}^{\frac{1}{2}}+r+2\|z_{u_0}\|_{L^2(0,\tilde{T};V)})+\sigma r 
+\|F(z_{u_0})\|_{L^2(0,\tilde{T};V^*)}.
\end{align*} 
It follows that
\begin{align}\label{eq: estimate F(v) direct}
\|F(v)\|_{L^2(0,\tilde{T};V^*)}+ \|f\|_{L^2(0,\tilde{T};V^*)}\leq \tilde{\beta}_{T_1,\sigma}(\tilde{T},r)r+\sigma r 
+\gamma(\tilde{T}),
\end{align}
with
\begin{align*}
&\tilde{\beta}_{T_1,\sigma}(\tilde{T},r)\coloneqq \tilde{C}_{T_1,\sigma}^{\frac{1}{2}}(\tilde{T}^{\frac{1}{2}}+r+2\|z_{u_0}\|_{L^2(0,\tilde{T};V)}),\\ 
&\gamma(\tilde{T})\coloneqq (\|f\|_{L^2(0,\tilde{T};V^*)}\vee\|g\|_{L^2(0,\tilde{T};\UH)})+\|F(z_{u_0})\|_{L^2(0,\tilde{T};V^*)}.
\end{align*}
Recall that $z_{u_0}\in \MR(0,T_1)\subset L^2(0,T_1;V)$ and by Lemma \ref{lem: embeddings rho_j}\ref{it:emb4}, $F(z_{u_0})\in L^2(0,T_1;V^*)$. So  $\tilde{\beta}_{T_1,\sigma}(\tilde{T},r)\downarrow 0$ as $\tilde{T},r\downarrow0$ and $\gamma(\tilde{T})\downarrow 0$ as $\tilde{T}\downarrow0$ by the Dominated Convergence Theorem.
Combining \eqref{eq: estimate F(v) direct} and \eqref{eq: estimate A_0} and putting
\begin{align*}
&\beta_{T_1,\sigma}(\tilde{T},r)\coloneqq  \tilde{\beta}_{T_1,\sigma}(\tilde{T},r)+2C_{T_1}r,\\
&\alpha_{T_1}(\tilde{T})\coloneqq 2C_{T_1}\alpha(\tilde{T})^2+\gamma(\tilde{T}),
\end{align*}
proves  \eqref{eq: estimate tilde f sigma} for $\tilde{f}$. By symmetry in Lemma \ref{lem: embeddings rho_j}, the estimate for $\tilde{g}$ follows similarly.
\end{proof}

Before we prove Theorem \ref{th: local well posedness skeleton}, we need one more lemma, a modification of \cite[Lemma 18.2.12]{HNVWvolume3}.

\begin{lemma}\label{lem: lipschitz estimates}
Let $u_0\in H$ and suppose that $(A,B)$ satisfies Assumption \ref{ass: crit var set}.
For $\eps_1$ and $r_1$ from Lemma \ref{lem:eps T_1},  the following estimates hold for any $\tilde{T}\in(0,T_1]$, $\eps\in(0,\eps_1]$, $r\in(0,r_1]$, $v_0, w_0\in B_H(u_0,\eps)$, $v\in Z_{r,\tilde{T}}(v_0)$, $w\in Z_{r,\tilde{T}}(w_0)$, $u\in\MR(0,\tilde{T})$ and $\sigma>0$:
\begin{align*}
&\|(A_0(v)-A_0(w))v\|_{L^2(0,\tilde{T};V^*)}\vee\|(B_0(v)-B_0(w))v\|_{L^2(0,\tilde{T};\UH)}\\
& \hspace{9.5cm}\leq c_{T_1}(r+\alpha(\tilde{T}))\|v-w\|_{\MR(0,\tilde{T})},\\
&\|(A_0(u_0)-A_0(w))u\|_{L^2(0,\tilde{T};V^*)}\vee\|(B_0(u_0)-B_0(w))u\|_{L^2(0,\tilde{T};\UH)}\\
& \hspace{9.5cm}\leq c_{T_1}(r+\beta(\tilde{T}))\|u\|_{\MR(0,\tilde{T})},\\
&\|F(v)-F(w)\|_{L^2(0,\tilde{T};V^*)}\vee\|G(v)-G(w)\|_{L^2(0,\tilde{T};\UH)}\leq (\gamma_{T_1,\sigma}(\tilde{T},r)+\sigma)\|v-w\|_{\MR(0,\tilde{T})}, 
\end{align*}
where $c_{T_1}$ is a constant and $\alpha(\tilde{T}),\beta(\tilde{T}),\gamma_{T_1,\sigma}(\tilde{T},r)\downarrow 0$ as $\tilde{T},r\downarrow0$. Moreover,  $c_{T_1},\alpha(\tilde{T}),\beta(\tilde{T})$ and $\gamma_{T_1,\sigma}(\tilde{T},r)$ are independent of $v_0,w_0,v$ and $w$.  
\end{lemma}

\begin{proof}
Fix $n\coloneqq 2\|u_0\|_H+2$, $c_{T_1}\coloneqq C_{n,T_1}$ and note that $\|v\|_{C([0,\tilde{T}];H)}+\|w\|_{C([0,\tilde{T}];H)}\leq n$ by Lemma \ref{lem:eps T_1}. By Lemma \ref{lem: embeddings rho_j}\ref{it:emb3}, we have
\begin{align*}
\|(A_0(v)-A_0(w))v\|_{L^2(0,\tilde{T};V^*)}&\leq c_{T_1}\|v-w\|_{C([0,\tilde{T}];H)}\|v\|_{L^2(0,\tilde{T};V)}\\
&\leq c_{T_1}\|v-w\|_{\MR(0,\tilde{T})}(\|v-z_{u_0}\|_{\MR(0,\tilde{T})}+\|z_{u_0}\|_{L^2(0,\tilde{T};V)})\\
&\leq c_{T_1}\|v-w\|_{\MR(0,\tilde{T})}(r+\alpha(\tilde{T})),
\end{align*}
where $\alpha(\tilde{T})\coloneqq \|z_{u_0}\|_{L^2(0,\tilde{T};V)}\downarrow0$ as $\tilde{T}\downarrow0$.
Similarly,
\begin{align*}
\|(A_0(u_0)-A_0(w))u\|_{L^2(0,\tilde{T};V^*)}&\leq c_{T_1}\|u_0-w\|_{C([0,\tilde{T}];H)}\|u\|_{L^2(0,\tilde{T};V)}\\
& \leq c_{T_1}(\|u_0-z_{u_0}\|_{C([0,\tilde{T}];H)}+\|z_{u_0}-w\|_{\MR(0,\tilde{T})})\|u\|_{\MR(0,\tilde{T})}\\
& \leq c_{T_1}(\beta(\tilde{T})+r)\|u\|_{\MR(0,\tilde{T})},
\end{align*}
where $\beta(\tilde{T})\coloneqq \|u_0-z_{u_0}\|_{C([0,\tilde{T}];H)}\downarrow 0$ as $\tilde{T}\downarrow0$ since $z_{u_0}(0)=u_0$.

Now we turn to $F$. By Lemma \ref{lem: embeddings rho_j}\ref{it:emb5}, we have for any $\tilde{\sigma}>0$:
\begin{align*}
 \|F(v)-F(w)\|_{L^2(0,\tilde{T},V^*)}^2&\leq  
 C_{n,T_1,\tilde{\sigma}}\|v-w\|_{\MR(0,\tilde{T})}^2\int_0^{\tilde{T}} 1+\|v\|_V^2+\|w\|_V^2\dd t\\
 &\qquad +\tilde{\sigma} C_{n,T_1}^2\|v-w\|_{\MR(0,\tilde{T})}^2.
\end{align*}
Moreover,
\[
\|v\|_{L^2(0,\tilde{T};V^*)}\leq \|v-z_{u_0}\|_{L^2(0,\tilde{T};V^*)}+\|z_{u_0}\|_{L^2(0,\tilde{T};V^*)}\leq r +\|z_{u_0}\|_{L^2(0,\tilde{T};V^*)}
\]
and similarly for $w$. Applying the above with $\tilde{\sigma}\coloneqq \sigma^2 C_{n,T_1}^{-2}$, putting $\tilde{C}_{T_1,\sigma}\coloneqq C_{n,T_1,\tilde{\sigma}}$ and taking square roots, we find 
\begin{align*}
 \|F(v)-F(w)\|_{L^2(0,\tilde{T},V^*)}&\leq  
 \tilde{C}_{T_1,{\sigma}}^{\frac{1}{2}}\|v-w\|_{\MR(0,\tilde{T})}(\tilde{T}^{\frac{1}{2}}+2r+2\|z_{u_0}\|_{L^2(0,\tilde{T};V^*)})
 +{\sigma} \|v-w\|_{\MR(0,\tilde{T})}.
\end{align*}
The desired estimate thus holds with $\gamma_{T_1,\sigma}(\tilde{T},r)\coloneqq \tilde{C}_{T_1,{\sigma}}^{\frac{1}{2}}(\tilde{T}^{\frac{1}{2}}+2r+2\|z_{u_0}\|_{L^2(0,\tilde{T};V^*)})$. 

By symmetry in Assumption \ref{it:ass3}, $B_0$ and $G$ can be estimated similarly. 
\end{proof}

We are now ready to prove Theorem \ref{th: local well posedness skeleton}. The proof is adapted from \cite[Th.\ 18.2.6]{HNVWvolume3}.

\begin{proof}[Proof of Theorem \ref{th: local well posedness skeleton}]
Let  $\eps_1,r_1>0$ be as in Lemma \ref{lem:eps T_1} and let $\tilde{T}\in(0,T_1]$, $\eps\in(0,\eps_1]$, $r\in(0,r_1]$. As above, define $\Psi_{v_0}\colon \MR(0,\tilde{T})\to \MR(0,\tilde{T})$ by $\Psi_{v_0}(v)\coloneqq u$, where $u$ is the unique strong solution to \eqref{eq:contraction map def}. Recall that $u$ solves \eqref{eq: local well-posed} if and only if $\Psi_{v_0}(u)=u$ and recall that $Z_{r,\tilde{T}}(v_0)$ defined by \eqref{eq:def Z_rT} is closed in $\MR(0,\tilde{T})$, hence complete.
We show that for $\tilde{T},\eps,r$ small enough, the mapping $\Psi_{v_0}$ maps $Z_{r,\tilde{T}}(v_0)$ to itself and is contractive. The Banach fixed point theorem then gives existence of a unique fixed point in $Z_{r,\tilde{T}}(v_0)$, hence existence of a solution to \eqref{eq: local well-posed}. We will extend the uniqueness within $Z_{r,\tilde{T}}(v_0)$ to uniqueness in $\MR(0,\tilde{T})$.

Let $v\in Z_{r,\tilde{T}}(v_0)$ and let $u\coloneqq \Psi_{v_0}(v)$. Let $z_{u_0}$ be defined as in \eqref{eq:def z_v_0} and define $\tilde{f},\tilde{g}$ by \eqref{eq:def tilde f tilde g}.
Note that $u-z_{u_0}=\Psi_{v_0-u_0}(v)$ , so by \eqref{def: maxreg} and \eqref{eq: estimate tilde f sigma}, we have for any $\sigma>0$:
 \begin{align*}
  \|u-z_{u_0}\|_{\MR(0,\tilde{T})}\leq \|u-z_{u_0}\|_{\MR(0,T_1)}
  &\leq K_{T_1}\left(\|v_0-u_0\|_H+\|\tilde{f}\|_{L^2(0,T_1;V^*)}+\|\tilde{g}\|_{L^2(0,T_1;\UH)}\right)\\
  &\leq K_{T_1}\left(\eps+2\alpha_{T_1}(\tilde{T})+2\beta_{T_1,\sigma}(\tilde{T},r)r+2\sigma r\right), 
 \end{align*}
 with $\alpha_{T_1}(\tilde{T}),\beta_{T_1,\sigma}(\tilde{T},r)\downarrow 0$ as $\tilde{T},r\downarrow 0$. Recall that $K_{T_1}$ from \eqref{def: maxreg} only depends on $T_1$, $T$, $\|u_0\|_H$ and $\psi$, not on $v_0$ or $v$. 
Fixing first $\sigma\coloneqq (4K_{T_1})^{-1}$, we find 
 \[
 \|u-z_{u_0}\|_{\MR(0,\tilde{T})}\leq \frac{r}{2}+K_{T_1}\left(\eps+2\alpha_{T_1}(\tilde{T})+2\beta_{T_1,\sigma}(\tilde{T},r)r\right).
 \]
For all small enough $r$ and all small enough $\tilde{T},\eps$ (dependent on $r$), one thus  has $\|u-z_{u_0}\|_{\MR(0,\tilde{T})}\leq r$, i.e. $\Psi_{v_0}(v)=u\in Z_{r,\tilde{T}}(v_0)$.
In particular, for all such $r,\tilde{T},\eps$ and for all $v_0\in B_H(u_0,\eps)$,  $\Psi_{v_0}$ maps $Z_{r,\tilde{T}}(v_0)$ to itself.

Now we show that for some (even smaller) $r,\tilde{T},\eps>0$, the map $\Psi_{v_0}\colon Z_{r,\tilde{T}}(v_0)\to Z_{r,\tilde{T}}(v_0)$ is contractive for all $v_0\in B_H(u_0,\eps)$ and we prove continuous dependence on the initial value $v_0$.
Let $v_0,w_0\in B_H(u_0,\eps)$, $v\in Z_{r,\tilde{T}}(v_0)$, $w\in Z_{r,\tilde{T}}(w_0)$ and note that $u\coloneqq \Psi_{v_0}(v)-\Psi_{w_0}(w)$ is a strong solution to
\begin{equation*}
  \begin{cases}
    &u'+\tilde{A}(u_0)u=(\tilde{A}(u_0)-\tilde{A}(v))v-(\tilde{A}(u_0)-\tilde{A}(w))w+\tilde{F}(v)-\tilde{F}(w)\quad \text{on } [0,\tilde{T}],\\
    &u(0)=v_0-w_0.
  \end{cases}
\end{equation*}
Hence, by \eqref{def: maxreg}:
\[
\|u\|_{\MR(0,\tilde{T})}\leq  K_{T_1}(\|v_0-w_0\|_H+\|\bar{f}\|_{L^2(0,\tilde{T};V^*)}+\|\bar{g}\|_{L^2(0,\tilde{T};\UH)}),
\]
with $\bar{f}\coloneqq (A_0(u_0)-A_0(v))v-(A_0(u_0)-A_0(w))w + F(v)-F(w)$ and $\bar{g}\coloneqq (B_0(u_0)-B_0(v))v-(B_0(u_0)-B_0(w))w + G(v)-G(w)$. We have by Lemma \ref{lem: lipschitz estimates}, for any $\sigma>0$: 
\begin{align*}
\|\bar{f}\|_{L^2(0,\tilde{T};V^*)}&\leq \|(A_0(v)-A_0(w))v\|_{L^2(0,\tilde{T};V^*)} +\|(A_0(u_0)-A_0(w))(v-w)\|_{L^2(0,\tilde{T};V^*)}\\
&\qquad\qquad+\|F(v)-F(w)\|_{L^2(0,\tilde{T};V^*)}\\
&\leq \Big(c_{T_1}(2r+\alpha(\tilde{T})+\beta(\tilde{T}))+\gamma_{T_1,\sigma}(\tilde{T},r)+\sigma\Big)\|v-w\|_{\MR(0,\tilde{T})},
\end{align*}
with $\alpha(\tilde{T}), \beta(\tilde{T}),\gamma_{T_1,\sigma}(\tilde{T},r)\downarrow 0$ as $\tilde{T},r\downarrow 0$. The same estimate applies to $\|\bar{g}\|_{L^2(0,\tilde{T};\UH)}$ by symmetry. Fixing $\sigma\coloneqq (8K_{T_1})^{-1}$ and putting $C(\tilde{T},r)\coloneqq 2\big(c_{T_1}(2r+\alpha(\tilde{T})+\beta(\tilde{T}))+\gamma_{T_1,\sigma}(\tilde{T},r)\big)$, we conclude that
\[
\|u\|_{\MR(0,\tilde{T})}\leq K_{T_1}\|v_0-w_0\|_H+\big(K_{T_1}C(\tilde{T},r)+\frac{1}{4}\big)\|v-w\|_{\MR(0,\tilde{T})},
\]
with $C(\tilde{T},r)\downarrow0$ as $\tilde{T},r\downarrow0$.
For all small enough $r,\tilde{T},\eps$ we thus have $K_{T_1}C(\tilde{T},r)\leq \frac{1}{4}$ and
\begin{equation}\label{eq: Psi contraction}
\|\Psi_{v_0}(v)-\Psi_{w_0}(w)\|_{\MR(0,\tilde{T})}=\|u\|_{\MR(0,\tilde{T})}\leq K_{T_1}\|v_0-w_0\|_H+\frac{1}{2}\|v-w\|_{\MR(0,\tilde{T})}.
\end{equation} 
Application to $w_0=v_0$ shows that $\Psi_{v_0}\colon Z_{r,\tilde{T}}(v_0)\to Z_{r,\tilde{T}}(v_0)$ is a strict contraction. A unique fixed point is thus guaranteed by the Banach fixed point theorem.
Now let $u_{v_0}\in Z_{r,\tilde{T}}(v_0)$ and $u_{w_0}\in Z_{r,\tilde{T}}(w_0)$ be fixed points of $\Psi_{v_0}$ and $\Psi_{w_0}$, respectively. Then \eqref{eq: Psi contraction} yields
\[
\|u_{v_0}-u_{w_0}\|_{\MR(0,\tilde{T})}=\|\Psi_{v_0}(u_{v_0})-\Psi_{w_0}(u_{w_0})\|_{\MR(0,\tilde{T})}\leq K_{T_1}\|v_0-w_0\|_H+\frac{1}{2}\|u_{v_0}-u_{w_0}\|_{\MR(0,\tilde{T})}.
\]
Consequently, \eqref{eq: ct dependence local well-posed} holds with $C\coloneqq 2K_{T_1}>0$.

It remains to show that uniqueness not only holds within $Z_{r,\tilde{T}}(v_0)$ but also within the larger space $\MR(0,\tilde{T})$.
Let $v,\tilde{v}\in \MR(0,\tilde{T})$ be strong solutions to \eqref{eq: local well-posed} and suppose that $v\neq \tilde{v}$. Then we have $s\coloneqq \inf\{t\in[0,\tilde{T}]:v(t)\neq \tilde{v} \text{ in } H\} \in[0,\tilde{T})$ since $\MR(0,\tilde{T})=C([0,\tilde{T}];H)\cap L^2(0,\tilde{T};V)$ and $V\into H$ is injective.  Moreover, $v(s)=\tilde{v}(s)\eqqcolon w_0$ as $v,\tilde{v}\in C([0,\tilde{T}];H)$ and $v(\cdot+s)$ and $\tilde{v}(\cdot+s)$ are strong solutions to
\begin{equation}\label{eq: shifted local}
\begin{cases}
    &u'+\tilde{A}(u)u=\tilde{F}(u) \quad\text{on } [0,\tilde{T}-s],\\
    &u(0)=w_0.
\end{cases}
\end{equation}
Now, by the first part of the proof, there exist $r_0,T_0>0$ such that \eqref{eq: shifted local} has a unique solution in $Z_{r,\delta}(w_0)$ for all $r\in(0,r_0]$ and $\delta\in(0,T_0]$ (take $u_0=v_0=w_0$).
Fix
\[
\delta\coloneqq \sup\{t\in[0,\min\{T_0,\tilde{T}-s\}):\|v(\cdot+s)-z_{w_0}\|_{\MR(0,t)}\vee\|\tilde{v}(\cdot+s)-z_{w_0}\|_{\MR(0,t)}<r_0\}
\]
and note that $\delta\in (0,\min\{T_0,\tilde{T}-s\}]$ since $v(0+s)=\tilde{v}(0+s)=w_0=z_{w_0}(0)$.
In particular, $\delta\in(0,T_0]$ and $v(\cdot+s),\tilde{v}(\cdot+s)\in Z_{r_0,T}(w_0)$ by definition of $\delta$.
Uniqueness of solutions in $Z_{r_0,\delta}(w_0)$ implies that $v(\cdot+s)=\tilde{v}(\cdot+s)$ on $[0,\delta]$.
Therefore $v=\tilde{v}$ on $[0,s+\delta]$, contradicting the definition of $s$.
We conclude that $v=\tilde{v}$.
\end{proof}

\begin{remark}
Observe that the local well-posedness could also have been proved under mere coercivity of $A_0$ instead of coercivity of $(A_0,B_0)$ (Assumption \ref{it:ass2}). Indeed, in the current section, we have only used Corollary \ref{cor: skeleton MR linearized crit var setting} and the estimates from Assumption \ref{it:ass3}. Now, the proof of Corollary \ref{cor: skeleton MR linearized crit var setting} continues when we only assume $\<A_0(t,u)v,v\?\geq \theta_{n,T}\|v\|_V^2-M_{n,T}\|v\|_H^2$,
  since then, combined with Assumption \ref{it:ass3} and Young's inequality:
      \begin{align*}
    \<\bar{A}(t)v,v\?&\geq \<A_0(t,w(t))v,v\?-\sigma\nn B_0(t,w(t))v \nn_H^2-C_\sigma\|\psi(t)\|_U^2\|v\|_H^2\\
    &\geq \theta_{n,T}\|v\|_V^2-({M}_{n,T}+C_\sigma\|\psi(t)\|_U^2)\|v\|_H^2-\sigma C_{n,T}^2\|v\|_V^2.
  \end{align*}
  Putting $\sigma\coloneqq \theta_{n,T}(2C_{n,T}^2)^{-1}$, the required coercivity \eqref{eq: coercivity L_1} for $\bar{A}$ follows.
\end{remark}

\subsection{Global well-posedness}

Similar to \cite[Chap.\ 5]{pruss16} and \cite[\S 18.2]{HNVWvolume3}, we will extend Theorem \ref{th: local well posedness skeleton} to a global well-posedness result by means of maximal solutions and a blow-up criterion.

\begin{definition}\label{def: maximal solution}
For $T\in(0,\infty]$, we define
\[
\MR_{\mathrm{loc}}(0,T)\coloneqq \{u\colon [0,T)\to H: u|_{[0,\tilde{T}]}\in\MR(0,\tilde{T}) \text{ for all } \tilde{T}\in[0,T)\}.
\]
A \emph{maximal solution} to \eqref{eq: skeleton eq} is a pair $(u_*,T_*)\in \MR_{\mathrm{loc}}(0,T_*) \times (0,\infty]$  such that
\begin{enumerate}[label=(\roman*)]
  \item for all $T\in(0,T_*)$, $u_*|_{[0,T]}$ is a strong solution to \eqref{eq: skeleton eq},
  \item for any $T>0$ and for any strong solution $u\in\MR(0,T)$ to \eqref{eq: skeleton eq} it holds that $T\leq T_*$ and $u=u_*$ on $[0,T]$.
\end{enumerate}
\end{definition}

Note that maximal solutions are unique by definition.
The proof of the next proposition is adapted from \cite[Th.\ 18.2.14, Th.\ 18.2.15]{HNVWvolume3} and \cite{pruss16}.

\begin{proposition}[Blow-up criterion]\label{prop: blow up criterion}
  Let $x\in H$ and $\psi\in L^2_{\mathrm{loc}}(\R_+;U)$. Suppose that $(A,B)$ satisfies Assumption \ref{ass: crit var set}. Then  equation \eqref{eq: skeleton eq} has a maximal solution $(u_*,T_*)$. Moreover, if $T_*<\infty$ and $\sup_{T\in[0,T_*)}\|u_*\|_{L^2(0,T;V)}<\infty$, then $\lim_{t\uparrow T_*}u_*(t)$ does not exist in $H$.
\end{proposition}
\begin{proof}
The proof of Theorem \ref{th: local well posedness skeleton} (with $u_0=v_0=x$) shows that there exists a local solution and that any strong solution on any finite time interval is unique. Hence, there exists a maximal solution $(u_*,T_*)$ for some $T_*\in(0,\infty]$ and $u_*\in \MR_{\mathrm{loc}}(0,T_*)$.

Suppose that $T_*<\infty$, $\sup_{T\in[0,T_*)}\|u_*\|_{L^2(0,T;V)}<\infty$ and $u^*\coloneqq \lim_{t\uparrow T_*}u_*(t)$ does exist in $H$. We will derive a contradiction. Note that the second assumption implies $u_*\in L^2(0,T_*;V)$.

By Theorem \ref{th: local well posedness skeleton}, there exists $\delta>0$ and a strong solution $u\in\MR(T_*,T_*+\delta)$ to
\begin{equation}\label{eq: loc sol in cond x}
  \begin{cases}
    &u'+\tilde{A}(u)u=\tilde{F}(u)\quad \text{on } [T_*,T_*+\delta],\\
    &u(T_*)=u^*,
  \end{cases}
\end{equation}
where we use that the translated pair $(A(T_*+\cdot,\cdot), B(T_*+\cdot,\cdot))$
also satisfies Assumption \ref{ass: crit var set}.
Then
\[
\bar{u}(t)\coloneqq   \begin{cases}
    u_*(t), \qquad &t\in [0,T_*),  \\
    u(t), &t\in [T_*,T_*+{\delta}]
  \end{cases}
\]
satisfies $\bar{u}\in \MR(0,T_*+{\delta})$ and $\bar{u}$ is a strong solution to \eqref{eq: local well-posed} on $[0,T_*+{\delta}]$, contradicting maximality of $(u_*,T_*)$.
\end{proof}

Using the blow-up criterion, we finally prove global well-posedness for the skeleton equation. Besides Assumption \ref{ass: crit var set}, we now also assume the coercivity condition \eqref{eq: coercivity condition (A,B)} for the pair $(A,B)$. This condition has not been used so far, but it is also needed for the global well-posedness result for the stochastic evolution equation \cite[Th.\ 3.5]{AV22variational}, see Theorem \ref{th: original stoch ev eq global well-posedness}.

\begin{theorem}[Global well-posedness skeleton equation]\label{th: global well posedness skeleton}
Suppose that $(A,B)$ satisfies Assumption \ref{ass: crit var set} and coercivity \eqref{eq: coercivity condition (A,B)}.
Then for any $\psi \in L_{\mathrm{loc}}^2(\R_+;U)$, $x\in H$ and $T>0$, there exists a unique strong solution $u\in\MR(0,T)$ to \eqref{eq: skeleton eq}. Moreover,
\begin{equation}\label{eq: MR estimate global sol}
  \|u\|_{\MR(0,T)}\leq (2+\frac{1}{\theta})^{\frac{1}{2}}\left(\|x\|_H+\sqrt{2}\|\phi\|_{L^2(0,T)} \right)\exp[ MT+\frac{1}{2}\|\psi\|_{L^2(0,T;U)}^2],
\end{equation}
where $\theta, M>0$ and $\phi\in L^2(0,T)$ are such that \eqref{eq: coercivity condition (A,B)} holds for $t\in[0,T]$.
\end{theorem}
\begin{proof}
By Proposition \ref{prop: blow up criterion} we have a maximal solution $(u_*,T_*)$ to  \eqref{eq: skeleton eq}. If $T_*=\infty$, then well-posedness for every $T>0$ follows.
Suppose that $T_*<\infty$. We will derive a contradiction.
Let $\theta, M>0$ and $\phi\in L^2(0,T_*)$ be such that the coercivity condition \eqref{eq: coercivity condition (A,B)} holds with $T=T^*$. By definition of the maximal solution, $u_*|_{[0,T]}$ is a strong solution to  \eqref{eq: skeleton eq} on $[0,T]$ for all $T\in[0,T_*)$.
The chain rule \eqref{eq: Ito pardoux deterministic} thus gives for all $t\in [0,T_*)$:
 \begin{align*}
   \|u_*(t)\|_H^2&=\|x\|_H^2+2\int_0^t \<-A(s,u_*(s)),u_*(s)\?+\<B(s,u_*(s))\psi(s),u_*(s)\? \dd s \\
    &\leq \|x\|_H^2+2\int_0^t -\frac{1}{2}\nn B(s,u_*(s))\nn_H^2 -\theta\|u_*(s)\|_V^2+M\|u_*(s)\|_H^2+|\phi(s)|^2 \\
    &\qquad\qquad\qquad\qquad+\nn B(s,u_*(s))\nn_H\|\psi(s)\|_U\|u_*(s)\|_H \dd s \\
    &\leq \|x\|_H^2+2\int_0^t -\frac{1}{2}\nn B(s,u_*(s))\nn_H^2 -\theta\|u_*(s)\|_V^2+M\|u_*(s)\|_H^2+|\phi(s)|^2  \\
    &\qquad\qquad\qquad\qquad+\frac{1}{2}\nn B(s,u_*(s))\nn_H^2+\frac{1}{2}\|\psi(s)\|_U^2\|u_*(s)\|_H^2 \dd s  \\
    &=-2\theta\|u_*\|_{L^2(0,t;V)}^2+\|x\|_H^2+2\|\phi\|_{L^2(0,t)}^2+\int_0^t (2M+\|\psi(s)\|_U^2)\|u_*(s)\|_H^2\dd s.
  \end{align*}
  By Lemma \ref{lem: gronwall consequence}(Gronwall), we obtain for all $T\in(0,T_*)$:
  \[
  \|u_*\|_{C([0,T];H)}^2+\|u_*\|_{L^2(0,T;V)}^2\leq (1+\frac{1}{2\theta})\left(\|x\|_H^2+2\|\phi\|_{L^2(0,T)}^2\right)\exp[2MT+\|\psi\|_{L^2(0,T;U)}^2],
  \]
  hence
  \begin{align}\label{eq: MR estimate global sol prep}
  \|u_*\|_{\MR(0,T)} &\leq (2+\frac{1}{\theta})^{\frac{1}{2}}\left(\|x\|_H+\sqrt{2}\|\phi\|_{L^2(0,T)} \right)\exp[ MT+\frac{1}{2}\|\psi\|_{L^2(0,T;U)}^2]\eqqcolon K(T),
  \end{align}
  where $K\colon [0,T_*]\to \R_+$ is increasing.
Applying Lemma \ref{lem: embeddings rho_j}\ref{it:emb4}
  with $n=K(T_*)<\infty$ we find that $F(u_*)\in L^2(0,t;V^*)$ for all $t\in (0,T_*)$ and $L\coloneqq \sup_{t\in[0,T_*)}\|F(u_*)\|_{L^2(0,t;V^*)}<\infty$.
  Thus, by the Monotone Convergence Theorem, $\|F(u_*)\|_{L^2(0,T_*;V^*)}\leq L<\infty$.
  Similarly, $G(u_*)\in L^2(0,T_*;\UH)$.
 Now we apply Corollary \ref{cor: skeleton MR linearized crit var setting} with $T\coloneqq T^*$, $w\coloneqq u_*\in C([0,T_*);H)\subset L^\infty(0,T;H)$ (extend by $w(T)\coloneqq u_*(0)$ on the Lebesgue null set $\{T\}$), $n\coloneqq K(T_*)$
 and $\bar{f}\coloneqq F(u_*)+f\in L^2(0,T_*;V^*)$, $\bar{g}\coloneqq G(u_*)+g\in L^2(0,T_*;\UH)$.
Corollary \ref{cor: skeleton MR linearized crit var setting} gives existence of a strong solution $\bar{u}\in \MR(0,T_*)$ to \eqref{eq: skeleton eq} on $[0,T_*]$.
By uniqueness of the maximal solution, it follows that $u_*|_{[0,T]}=\bar{u}|_{[0,T]}$ for all $T\in[0,T_*)$. Hence $\lim_{t\uparrow T_*}u_*(t)=\lim_{t\uparrow T_*}\bar{u}(t)=\bar{u}(T_*)\in H$, contradicting Proposition \ref{prop: blow up criterion}.

We conclude that the assumption $T_*<\infty$ was false, i.e. $T_*=\infty$ and for any $T>0$, $u\coloneqq u_*|_{[0,T]}\in\MR(0,T)$ is the desired strong solution on $[0,T]$. Finally, note that the estimates leading to \eqref{eq: MR estimate global sol prep} can be repeated with $\theta$, $M$, $\phi$ of the coercivity condition belonging to $T$ instead of $T_*$,  proving \eqref{eq: MR estimate global sol}.
\end{proof}

\section{Proof of the large deviation principle}\label{sec: proof LDP}

\subsection{Weak convergence approach}\label{sec: weak convergence approach}

We return to our original setting of Section \ref{sec: Main result} and start with the proof of the LDP of Theorem \ref{th: main LDP theorem}. From now on, assume that $U$ is a real separable Hilbert space and $(\Om,\F,\P,(\F_t)_{t\geq 0})$ is a filtered probability space. For $\eps>0$, we let $Y^\eps$ be the unique strong solution to
  \begin{equation}\label{eq: SPDE Y^eps}
  \begin{cases}
    \dd Y^\eps(t)=-A(t,Y^\eps(t))+\sqrt{\eps}B(t,Y^\eps(t))\dd W(t), \quad t\in[0,T],\\
    Y^\eps(0)=x.
  \end{cases}
  \end{equation}
Here, $W$ is a \emph{$U$-cylindrical Brownian motion}, which is defined as follows.

\begin{definition}\label{def: cylindrical BM}
Let $W\in\mathcal{L}(L^2(\R_+;U),L^2(\Om))$. Then $W$ is called a \emph{$U$-cylindrical Brownian motion with respect to $(\Om,\F,\P,(\F_t)_{t\geq 0})$} if for all $f,g\in L^2(\R_+;U)$ and $t\in\R_+$:
  \begin{enumerate}[label=\emph{(\roman*)},ref=\textup{(\roman*)}]
    \item \label{it:bm1} $Wf$ is normally distributed with mean zero and $\E[Wf Wg]=\<f,g\?_{L^2(\R_+;U)}$,
    \item \label{it:bm2} if $\supp(f)\subset [0,t]$, then $Wf$ is $\F_t$-measurable,
    \item \label{it:bm3} if $\supp(f)\subset [t,\infty)$, then $Wf$ is independent of $\F_t$.
  \end{enumerate}
\end{definition}

 There exist several different definitions of a cylindrical Brownian motion or cylindrical Wiener process in the literature. Some references in our proof of the LDP use (an equivalent of) an \emph{$\R^\infty$-Brownian motion}, defined below.

\begin{definition}\label{def: sequence independent BM}
An \emph{$\R^\infty$-Brownian motion} (in $U$) is a pair $\tilde{W}\coloneqq ((\beta_k)_{k\in\N},(e_k)_{k\in\N})$, with $(\beta_k)_{k\in\N}$ a sequence of independent standard real-valued $(\F_t)$-Brownian motions and $(e_k)_{k\in\N}$ an orthonormal basis for $U$.
\end{definition}

In Proposition \ref{prop: equiv cyl BM} of Appendix \ref{appendix}, the connection between the $U$-cylindrical Brownian motion and the $\R^\infty$-Brownian motion is summarized, as well as their equivalent, but differently constructed stochastic integrals.
The $\R^\infty$-Brownian motion of Definition \ref{def: sequence independent BM} is e.g.\ used in \cite{liurockner15}, where it is called a cylindrical $Q$-Wiener process (with $Q\coloneqq I\in\mathcal{L}(U;U)$ the identity operator). Often, the notation $\tilde{W}(t)= \sum_{k\in\N}\beta_k(t)e_k$ is also used, which is only formal as the series does not converge in $L^2(\Om;U)$. However, we will write $\tilde{W}=((\beta_k)_{k\in\N},(e_k)_{k\in\N})$.

\begin{remark}
For the proof of the LDP for $(Y^\eps)$, without loss of generality, we can assume that the filtration $(\F_t)_{t\geq 0}$ is right-continuous and complete. Indeed, one can fix any orthonormal basis $(e_k)_{k\in \N}$ of $U$ and put  $\mathcal{H}^k_t\coloneqq \sigma(W(\one_{(0,s]}\otimes e_k):s\in[0,t])$ for $k\in\N$ and
\[
\F^0_t\coloneqq \sigma(\bigcup_{k\in\N}\mathcal{H}^k_t), \quad \mathcal{H}^0_t\coloneqq \sigma(\bigcup_{k\in\N}\mathcal{H}^k_t\cup \mathcal{N}), \quad \mathcal{H}_t\coloneqq \mathcal{H}^0_{t^+}\coloneqq \bigcap_{h>0}\mathcal{H}^0_{t+h},
\]
where $\mathcal{N}$ is the collection of all $(\Om,\F,\P)$-null sets. Then $(\mathcal{H}_t)_{t\geq 0}$ is a complete, right-continuous filtration on $(\Om,\bar{\F},\bar{\P})$. Moreover, one can show that $W$ is a $U$-cylindrical Brownian motion with respect to $(\Om,\bar{\F},\bar{\P}, (\mathcal{H}_t)_{t\geq 0})$ and with respect to $(\Om,{\F},{\P}, (\F^0_t)_{t\geq 0})$. Let $Y^\eps_0$ and $\bar{Y}^\eps$  be the unique strong solution to \eqref{eq: SPDE Y^eps} on $(\Om,{\F},{\P}, (\F^0_t)_{t\geq 0})$ and $(\Om,\bar{\F},\bar{\P},(\mathcal{H}_t)_{t\geq 0})$, respectively. Since $\F^0_t\subset \F_t \cap\mathcal{H}_t$, $Y^\eps_0$  is also a strong solution to \eqref{eq: SPDE Y^eps} on $(\Om,{\F},{\P},(\F_t)_{t\geq 0})$ and on $(\Om,\bar{\F},\bar{\P},(\mathcal{H}_t)_{t\geq 0})$. Pathwise uniqueness gives  ${Y}^\eps=Y_0^\eps=\bar{Y}^\eps$ $\P$-a.s. Now trivially from Definition \ref{def: LDP}, if we prove the LDP for $(\bar{Y}^\eps)$, then the LDP carries over to $(Y^\eps)$.
\end{remark}

In view of the above remark, we assume that the filtration \emph{$(\F_t)_{t\geq 0}$ is right-continuous and complete} from now on, and we assume that  $W$ is a $U$-cylindrical Brownian motion with respect to $(\F_t)_{t\geq 0}$.
Moreover, we fix any orthonormal basis $(e_k)_{k\in \N}$ for $U$.
We let $\tilde{W}=(\beta_k)_{k\in\N},(e_k)_{k\in\N})$ denote the  unique $\R^\infty$-Brownian motion associated to $W$ from Proposition \ref{prop: equiv cyl BM}, i.e.\ satisfying \eqref{eq: assump beta_k}.
In the upcoming proofs $\tilde{W}$ will be useful, since we will be applying the Yamada-Watanabe theorem and Girsanov's theorem for $\R^\infty$-Brownian motions.
Finally, from now on we fix a separable Hilbert space $U_1$ and a Hilbert-Schmidt \emph{inclusion}
$J\colon U\into U_1$. This is always possible: let $\<u,v\?_1\coloneqq \sum_{k=1}^\infty \frac{1}{k}\<u,e_k\?_U \<e_k,v\?_U$ for $u,v\in U$ and let $U_1\coloneqq \mathrm{completion}(U, \<\cdot,\cdot\?_1)$.
We associate to $\tilde{W}$ the following $U_1$-valued process:
\begin{equation}\label{eq: Q-Wiener}
\tilde{W}_1(t)\coloneqq \sum_{k=1}^{\infty}\beta_k(t)J e_k, \qquad t\in[0,T].
\end{equation}
By \cite[Prop.\ 2.5.2]{liurockner15}, $\tilde{W}_1$ is a $Q_1$-Wiener process on $U_1$, with $Q_1\coloneqq JJ^*$.
In what follows,
$\tilde{W}_1$ denotes this $Q_1$-Wiener process defined by \eqref{eq: Q-Wiener}. We note that the paths of $\tilde{W}_1$ are in $C([0,T];U_1)$.

\begin{definition}\label{def: S_K A_K}
We define
\[
\mathcal{A}\coloneqq \{\Psi\colon [0,T]\times \Om\to U : \Psi \text{ is an } (\F_t)\text{-predictable process}, \|\Psi\|_{L^2(0,T;U)}<\infty \; \P\text{-a.s.}\}
\]
and for $K>0$,
\[
S_K\coloneqq \{\psi\in L^2(0,T;U) : \|\psi\|_{L^2(0,T;U)}\leq K\},\qquad \mathcal{A}_K\coloneqq \{\Psi\in\mathcal{A}: \Psi\in S_K \; \P\text{-a.s.}\}.
\]
We write $(S_K,\mathrm{weak})$ for the topological space consisting of $S_K$, equipped with the weak topology inherited from $L^2(0,T;U)$.
\end{definition}

The next theorem gives sufficient conditions for the LDP and is known as the \emph{weak convergence approach}, which originates from \cite[Th.\ 4.4]{budhidupuis01}. In \cite{matoussi21}, a useful adaptation was proved. The following version is immediately derived from \cite[Th.\ 3.2]{matoussi21}. We will use it to prove Theorem \ref{th: main LDP theorem}.

\begin{theorem}\label{th: sufficient LDP criterion}
Let $\mathcal{E}$ be a Polish space and let $(Y^\eps)_{\eps>0}$ be a collection of $\mathcal{E}$-valued random variables on $(\Omega,\F,\P)$.
Let $\tilde{W}=((\beta_k)_{k\in\N},(e_k)_{k\in\N})$ be an $\R^\infty$-Brownian motion. Let $\tilde{W}_1\colon \Om\to C([0,T];U_1)$ be the associated $Q_1$-Wiener process on $U_1$ defined by \eqref{eq: Q-Wiener}.
Suppose that for $\eps\geq 0$, there exist measurable maps $\G^\eps\colon C([0,T];U_1)\to\mathcal{E}$ such that
\begin{enumerate}[label=\emph{(\roman*)},ref=\textup{(\roman*)}]
\item \label{it:suf1} $Y^\eps=\G^\eps(\tilde{W}_1(\cdot))$ a.s.\ for all $\eps>0$,
  \item \label{it:suf2} for any $K<\infty$,  $(\psi_n)\subset S_K$ and $\psi\in S_K$ with $\psi_n\to \psi$ weakly in $L^2(0,T;U)$, it holds that
  \[
  \G^0\left(\int_0^{\cdot} \psi_n(s)\dd s\right)\to\G^0\left(\int_0^\cdot \psi(s)\dd s\right) \text{ in } \mathcal{E},
  \]
  \item \label{it:suf3} for any $K<\infty$ and $(\Psi^\eps)\subset \mathcal{A}_K$, it holds that
  \[
  \G^\eps\left(\tilde{W}_1(\cdot)+\frac{1}{\sqrt{\eps}}\int_0^{\cdot} \Psi^\eps(s)\dd s\right)-\G^0\left(\int_0^\cdot \Psi^\eps(s)\dd s\right)\to 0 \text{ in probability}
  \]
  as $\mathcal{E}$-valued random variables.
\end{enumerate}
Then $(Y^\eps)_{\eps>0}$ satisfies the LDP on $\mathcal{E}$ with good rate function
\begin{align}\label{eq: def rate I budhi}
I(z)\coloneqq \frac{1}{2}\inf\Big\{\int_0^T\|\psi(s)\|_U^2\dd s : \psi\in L^2(0,T;U),\, z=\G^0(\int_0^{\cdot} \psi(s)\dd s)\Big\}.
\end{align}
\end{theorem}
Conditions \ref{it:suf2} and \ref{it:suf3} imply the conditions of the original weak convergence approach of \cite{budhidupuis01}. For the latter, instead of  \ref{it:suf2} and \ref{it:suf3}, one would require
\begin{enumerate}[label={(\Roman*)},ref=\textup{(\Roman*)}, start=2]
    \item\label{it:bd1}  for any $K<\infty$, $\{\G^0(\int_0^\cdot \psi(s)\dd s):\psi\in S_K\}$ is a compact subset of $\mathcal{E}$,
  \item \label{it:bd2} for any $K<\infty$, if $(\Psi^\eps)\subset \mathcal{A}_K$ with $\Psi^\eps\to \Psi$ in distribution with respect to the weak topology on $L^2(0,T;U)$, then
  $\G^\eps\left(\tilde{W}_1(\cdot)+\frac{1}{\sqrt{\eps}}\int_0^{\cdot} \Psi^\eps(s)\dd s\right)\to  \G^0\left(\int_0^\cdot \Psi(s)\dd s\right) \text{ in distribution}$.
\end{enumerate}
Here, \cite[Th.\ 4.4]{budhidupuis01} is applied with $Q_1$-Wiener process  $\tilde{W}_1$,
$H\coloneqq U_1$, $H_0\coloneqq Q_1^{\frac{1}{2}}(U_1)$ and one uses that $Q_1^{\frac{1}{2}}(U_1)=J(U)=U$ as a subspace of $U_1$, see \cite[Prop.\ 2.5.2]{liurockner15} (with $Q\coloneqq I$, $U_0\coloneqq I^{\frac{1}{2}}(U)=U$).

Note that \ref{it:bd1} means that the sublevel sets of the rate function $I$ defined by \eqref{eq: def rate I budhi} are compact, as is also required in Definition \ref{def: LDP}. On the other hand, \ref{it:suf2} means that the map $\tau\colon (S_K,\mathrm{weak})\to \mathcal{E}\colon \psi\mapsto \mathcal{G}^0(\int_0^\cdot \psi\dd s)=u^{\psi}$
is continuous ($S_K$ is weakly metrizable as opposed to $L^2(0,T;U)$, thus sequential continuity suffices). In particular, this implies \ref{it:bd1}. Indeed,  $S_K\subset L^2(0,T;U)$ is weakly compact by the Banach-Alaoglu theorem and reflexivity of $L^2(0,T;U)$, so $\{\G^0(\int_0^\cdot \psi(s)\dd s):\psi\in S_K\}=\tau(S_K)$ is the continuous image of a compact set, hence it is compact.

We will apply Theorem \ref{th: sufficient LDP criterion} with the map $\mathcal{G}^0\colon C([0,T];U_1)\to\MR(0,T)$ given by
  \begin{equation}\label{eq: def G^0}
  \mathcal{G}^0(\gamma)\coloneqq \begin{cases}
                    u^\psi, \quad &\text{if } \gamma=\int_0^\cdot \psi(s)\dd s, \, \psi\in L^2(0,T;U),\\
                    0, &\text{otherwise},
                    \end{cases}
  \end{equation}
where $u^\psi$ is the strong solution to \eqref{eq: skeleton eq}. Note that the rate function $I$ defined by \eqref{eq: def rate I} is then precisely equal to the rate function given by \eqref{eq: def rate I budhi}.

We will verify that all conditions in Theorem \ref{th: sufficient LDP criterion} are satisfied for $Y^\eps$ defined as the strong solution to \eqref{eq: SPDE Y^eps}. Condition \ref{it:suf1} follows from the Yamada-Watanabe theorem in \cite{rock08}. The details are given in Lemma \ref{lem: yamada-watanabe}, as well as a preparation for the proof of condition \ref{it:suf3}.

\begin{lemma}\label{lem: yamada-watanabe}
Suppose that Assumption \ref{ass: crit var set} holds and suppose that $(A,B)$ satisfies \eqref{eq: coercivity condition (A,B)}. Let $x\in H$. Then for each $\eps>0$, there exists a measurable map $\mathcal{G}^\eps\colon C([0,T];U_1)\to\MR(0,T)$  such that the unique strong solution $Y^\eps$ to
  \eqref{eq: SPDE Y^eps} satisfies $Y^\eps=\mathcal{G}^{\eps}(\tilde{W}_1)$ a.s., where $\tilde{W}_1$ is given by \eqref{eq: Q-Wiener}.
  Moreover, for any $\Psi^\eps\in \mathcal{A}_K$,
  $X^\eps\coloneqq \mathcal{G}^{\eps}({\tilde{W}_1}(\cdot)+\frac{1}{\sqrt{\eps}}\int_0^{\cdot}{\Psi}^{\eps}(s)\dd s)$ is a strong solution to
  \begin{equation}\label{eq: SPDE X^eps}
  \begin{cases}
    \dd X^\eps(t)=-A(t,X^\eps(t))+B(t,X^\eps(t))\Psi^\eps(t)+\sqrt{\eps}B(t,X^\eps(t))\dd W(t), \quad t\in[0,T],\\
    X^\eps(0)=x.
  \end{cases}
  \end{equation}
\end{lemma}
\begin{proof}
To prove the first statement, we use the Yamada-Watanabe theorem from \cite[Th.\ 2.1]{rock08}  on $[0,T]$
with $L^1(0,T;V)$ replaced by $L^2(0,T;V)$.
Let $\eps>0$.  For any $Y^\eps$ with $Y^\eps\in\MR(0,T)$ a.s.\ and for any $\xi\in L^0((\Om,\F_0);H)$, we have that $(Y^\eps, \tilde{W})$ is a weak solution in the sense of \cite[Def. 1.4]{rock08} to
\begin{equation*}
  \begin{cases}
    \dd \tilde{Y}^\eps(t)=-A(t,\tilde{Y}^\eps(t))+\sqrt{\eps}B(t,\tilde{Y}^\eps(t))\dd \tilde{W}(t),\\
    \tilde{Y}^\eps(0)=\xi,
  \end{cases}
\end{equation*}
if and only if $Y^\eps$ is a strong solution in the sense of \cite[Def. 3.2]{AV22variational} to \eqref{eq: SPDE Y^eps} with $x$ replaced by $\xi$. This is a mere consequence of \eqref{eq: identity stoch integrals defs}
and the fact that $B(\cdot,Y(\cdot))\in  L^2(([0,T]\times\Om,\PP,\lambda\times \P);\UH)\subset\mathcal{N}(0,T)$ for any $Y\in \MR(0,T)$. By \cite[Th.\ 3.5]{AV22variational}, \eqref{eq: SPDE Y^eps} has a unique strong solution $Y^\eps$, also when $x$ is replaced by random initial data $\xi$. Thus we have pathwise uniqueness in the sense of \cite[Def. 1.7]{rock08} and we have existence of a.s.\ $\MR(0,T)$-valued weak solutions.
Now fix $x\in H$ and  $\eps>0$ and let $Y^\eps$ be the unique strong solution to \eqref{eq: SPDE Y^eps}. By \cite[Th.\ 2.1, Def. 1.9(2), Def. 1.8]{rock08}
there exists a measurable map $\mathcal{G}^\eps\colon C([0,T];U_1)\to \MR(0,T)$ such that a.s.  $Y^\eps=\mathcal{G}^\eps(\tilde{W}_1(\cdot))$.

Next, let $X^\eps\coloneqq \mathcal{G}^{\eps}({\tilde{W}_1}(\cdot)+\frac{1}{\sqrt{\eps}}\int_0^{\cdot}{\Psi}^{\eps}(s)\dd s)$.
We prove that $X^\eps$ solves \eqref{eq: SPDE X^eps}. Define
\[
\hat{W}\coloneqq \tilde{W}+\frac{1}{\sqrt{\eps}}\int_0^{\cdot} \Psi^{\eps}(s)\dd s\coloneqq ((\hat{\beta}_k)_{k\in\N},(e_k)_{k\in\N}), \qquad \hat{\beta}_k\coloneqq \beta_k+\frac{1}{\sqrt{\eps}}\int_0^{\cdot} \<\Psi^{\eps}(s),e_k\?_U\dd s.
\]
We have
$\E[\exp(\frac{1}{2}\|-\frac{1}{\sqrt{\eps}}\Psi^\eps\|_{L^2(0,T;U)}^2)]\leq
\exp(\frac{K^2}{2\eps})<\infty$, so by Novikov's condition \cite[Prop.\ 5.12]{karatzas98},
\begin{equation*}
  \E\left[\exp\left(\int_0^T\<-\frac{1}{\sqrt{\eps}}\Psi^\eps(s),\dd \tilde{W}(s)\?_{U}-\frac{1}{2}\|\frac{1}{\sqrt{\eps}}\Psi^\eps\|_{L^2(0,T;U)}^2\right)\right]=1.
\end{equation*}
Now Girsanov's theorem \cite[Proposition I.0.6]{liurockner15}, \cite[Th.\ 2.3]{ferrario13} yields that $\hat{W}$ is an $\R^\infty$-Brownian motion on $(\Om,\F,\hat{\P},(\F_t)_{t\geq 0})$, where
\[
\hat{\P}\coloneqq  \exp\left(-\frac{1}{\sqrt{\eps}}\int_0^T\<\Psi^\eps(s),\dd \tilde{W}(s)\?_{U}-\frac{1}{2\eps}\|\Psi^\eps\|_{L^2(0,T;U)}^2\right) \dd \P.
\]
Moreover, $\hat{W}$ induces a $U_1$-valued $Q_1$-Wiener process $\hat{W}_1$ on $(\Om,\F,\hat{\P},(\F_t)_{t\geq 0})$ using the same Hilbert-Schmidt inclusion $J\colon U\into U_1$ as we used for $\tilde{W}_1$ in \eqref{eq: Q-Wiener}, resulting in:
\begin{align*}
\hat{W}_1(t)\coloneqq \sum_{k\in\N}\hat{\beta}_k(t)Je_k&=\sum_{k\in\N}{\beta}_k(t)Je_k+\frac{1}{\sqrt{\eps}}\sum_{k\in\N}\Big(\int_0^{t} \<\Psi^{\eps}(s),e_k\?_U\dd s\Big)Je_k\\
&=\tilde{W}_1(t)+\frac{1}{\sqrt{\eps}}\int_0^{t}\sum_{k\in\N} \<\Psi^{\eps}(s),e_k\?_Ue_k\dd s\\
&=\tilde{W}_1(t)+\frac{1}{\sqrt{\eps}}\int_0^{t} \Psi^{\eps}(s)\dd s
\end{align*}
$\P$-a.s.\ in $U_1$, where we used that $\Psi^{\eps}\in\mathcal{A}_K$ to apply Fubini's theorem in the second line. Thus, recalling the definition of ${X}^\eps$ and noting that  $\hat{\P}\ll \P\ll \hat{\P}$, we have $\hat{\P}$-a.s.\ $X^\eps =\mathcal{G}^{\eps}(\hat{W}_1(\cdot))$.
By the Yamada-Watanabe theorem \cite[Th.\ 2.1, Def. 1.9]{rock08}, for ${X}^\eps=\mathcal{G}^{\eps}(\hat{W}_1(\cdot))$ we have that $({X}^\eps, \hat{W})$ is a weak solution to \eqref{eq: SPDE Y^eps}. That is, ${X}^\eps$ satisfies $\hat{\P}$-a.s.\ in $V^*$:
\begin{align}\label{eq: X eps}
    {X}^\eps(t)&=x+\int_0^t -A(s,{X}^\eps(s))\dd s+\int_0^t\sqrt{\eps}B(s,{X}^\eps(s))\dd \hat{W}(s).
\end{align}
By Proposition \ref{prop: equiv cyl BM}, there exists a unique $U$-cylindrical Brownian motion $\hat{\mathcal{W}}\in \mathcal{L}(L^2(\R_+;U);L^2(\Om))$ with respect to $(\Om,\F,\hat{\P},(\F_t)_{t\geq 0})$, satisfying for all $u\in U$ and $t\in[0,T]$:
\begin{align}
  \hat{\mathcal{W}}(\one_{(0,t]}\otimes u)=\sum_{k=1}^\infty \hat{\beta}_k(t)\<u,e_k\?_U
  &= W(\one_{(0,t]}\otimes u)+\frac{1}{\sqrt{\eps}}\int_0^t \<\Psi^\eps(s),u\?_U\dd s, \label{eq: hat W identical on elementary}
  \end{align}
where the last equality follows from the definition of $\hat{\beta}_k$ and  \eqref{eq: assump beta_k}.
Let $\hat{\mathcal{N}}(0,T)$ denote the stochastically integrable processes with respect to $\hat{\mathcal{W}}$ and $\hat{W}$ on $(\Om,\F,\hat{\P},(\F_t)_{t\geq 0})$, i.e. \eqref{eq: def class integrable processes} with $\P$ replaced by $\hat{\P}$. Note that $\hat{\mathcal{N}}(0,T)=\mathcal{N}(0,T)$, since $\P\ll\hat{\P}\ll\P$.
Thus, Proposition \ref{prop: equiv cyl BM} gives $\int_0^t \Phi(s)\dd \hat{\mathcal{W}}(s)=\int_0^t \Phi(s)\dd \hat{W}(s)$
$\hat{\P}$-a.s.\ for all $\Phi\in \mathcal{N}(0,T)$
and $t\in[0,T]$.  Therefore, combined with \eqref{eq: X eps},  ${X}^\eps$ satisfies $\hat{\P}$-a.s.\ (hence $\P$-a.s.) in $V^*$:
\begin{align*}
    {X}^\eps(t)&=x+\int_0^t -A(s,{X}^\eps(s))\dd s+\int_0^t\sqrt{\eps}B(s,{X}^\eps(s))\dd \hat{\mathcal{W}}(s)\\
    &=x+\int_0^t -A(s,{X}^\eps(s))\dd s+\int_0^t\sqrt{\eps}B(s,{X}^\eps(s))\dd W(s)+\int_0^tB(s,{X}^\eps(s))\Psi^{\eps}(s)\dd s.
\end{align*}
In the last line we used that $\int_0^t\Phi(s)\dd \hat{\mathcal{W}}(s)  =\int_0^t\Phi(s)\dd {W}(s) +\frac{1}{\sqrt{\eps}}\int_0^t\Phi(s)\Psi^\eps(s)\dd s$ for $\Phi\in \mathcal{N}(0,T)$ and $t\in[0,T]$.
For $\Phi= \one_{A\times (t_1,t_2]}\otimes(u\otimes x)$ with $0\leq t_1<t_2\leq T$, $A\in\F_{t_1}$, $u\in U$, $x\in H$,
the identity follows from \eqref{eq: hat W identical on elementary} and the definition of the stochastic integral for elementary processes \cite[p.\ 305]{NVW15}.  By linearity and continuity of the integrals and by a density argument  and localization, the identity extends for $\Phi\in \mathcal{N}(0,T)$.
This finishes the proof of the last claim of the lemma.
\end{proof}
\begin{remark}
The above proof also yields existence and uniqueness of strong solutions to \eqref{eq: SPDE X^eps}, since it was actually shown that $X^\eps$ is a strong solution to \eqref{eq: SPDE X^eps} if and only if it is a strong solution to \eqref{eq: SPDE Y^eps} with $W$ replaced by the $U$-cylindrical Brownian motion $\hat{\mathcal{W}}$. The latter was already considered in Theorem \ref{th: original stoch ev eq global well-posedness}.
\end{remark}

\subsection{Weakly continuous dependence in the skeleton equation}\label{ss:weakskeleton}

In this subsection we prove that condition \ref{it:suf2} of Theorem \ref{th: sufficient LDP criterion} is satisfied.
This will be achieved in the upcoming Proposition \ref{prop:cpt sublevel sets}.
Its proof
was inspired by \cite[Th.\ 3.2]{hongliliu21}.
Using an additional approximation by Bochner-simple functions, we can omit the time(-H\"older) regularity assumptions on $B$ of  \cite[(H5)]{hongliliu21}.

\begin{lemma}\label{lem: I_3 to zero}
Let $(w_n)\subset C([0,T];H)$, $(\alpha_n) \subset L^1(0,T;V^*)$ and $(\psi_n)\subset L^2(0,T;U)$ be such that
\[
w_n(t) = \int_0^t \alpha_n(s) \dd s
\] and such that $C_\alpha\coloneqq\sup_{n\in\N}\|\alpha_n\|_{L^1(0,T;V^*)}<\infty$,  $C_w\coloneqq\sup_{n\in\N}\|w_n\|_{C([0,T];H)}<\infty$ and $\psi_n\to \psi$ weakly in $L^2(0,T;U)$. Let ${b} \in L^2(0,T;\UH)$.
Then,
\begin{equation}\label{eq: weak conv lemma}
\lim_{n\to\infty}\sup_{t\in[0,T]}\Big|\int_0^t \<{b}(s)(\psi_n(s)-\psi(s)),w_n(s)\?\dd s\Big|=0.
\end{equation}
\end{lemma}
\begin{proof}
Without loss of generality, we can assume $\psi=0$, i.e. $\psi_n\to 0$ weakly in $L^2(0,T;U)$ (apply to $\psi_n-\psi$).
Since $(\psi_n)$ is weakly convergent, it is bounded. Throughout the proof we let
\[
C_\psi\coloneqq  \sup_{n\in\N}\|\psi_n\|_{L^2(0,T;U)}<\infty.
\]
First, let us observe that it suffices to prove \eqref{eq: weak conv lemma} for all $b$ in the collection \[
\mathcal{S}\coloneqq\{\one_D\otimes u\otimes v: D\in \BB([0,T]), u\in U, v\in V\}\subset L^2(0,T;\UH),
\]
where $\big(\one_D\otimes u\otimes v\big)(t)x\coloneqq\one_D(t)(u,x)_Uv\in H$ for $t\in [0,T]$ and  $x\in U$. Note that $\mathrm{span}(\mathcal{S})$ is dense in $L^2(0,T;\UH)$, using consecutively density of Bochner-simple functions, density of finite rank operators in $\mathcal{L}_2(U,H)$ and density of $V$ in $H$.
Define for $n\in\N$:
\[
I_n\colon L^2(0,T;\UH)\to C([0,T];\R), \; I_n(b) \coloneqq \int_0^\cdot \<{b}(s)\psi_n(s),w_n(s)\?\dd s.
\]
Each $I_n$ is linear and continuous with $\|I_n\|\leq C_\psi C_w$, independent of $n$:
\begin{align}\label{eq: I_n unif bnd}
\|I_n(b)\|_{C([0,T];\R)}&\leq  \|b\|_{L^2(0,T;\UH)}\|\psi_n\|_{L^2(0,T;U)}\|w_n\|_{C(0,T;H)}\leq \|b\|_{L^2(0,T;\UH)}C_\psi C_w.
\end{align}
If \eqref{eq: weak conv lemma} holds for all $b\in \mathcal{S}$, i.e.\ $\lim_{n\to\infty}\|I_n(b)\|_{C([0,T];\R)}=0$, then it also holds for all $b\in \mathrm{span}(\mathcal{S})$, by the triangle inequality in $C([0,T];\R)$. Moreover, for $b\in L^2(0,T;\UH)$, we find $(b_k)\subset \mathrm{span}(\mathcal{S})$ with $b_k \to b$ in $L^2(0,T;\UH)$ by density. Now \eqref{eq: I_n unif bnd} and a standard $2\eps$-argument
yield \eqref{eq: weak conv lemma} for $b$.

It remains to prove \eqref{eq: weak conv lemma} for $b=\one_D\otimes u\otimes v$ with $D\in \BB([0,T])$, $u\in U$ and $v\in V$. Note that in this case,
\begin{equation}\label{eq: reduction to R}
I_n(b)=\int_0^\cdot \one_D(s)\big(u,\psi_n(s)\big)_U\<v,w_n(s)\?\dd s
\end{equation}
and we have
\[
(v, w_n(s))_H=\<v,w_n(s)\?=\int_0^t \<\alpha_n(s),v\?\dd s.
\]
Since $u$ and $v$ are fixed, we have $(v, w_n(\cdot))_H\in C([0,T];\R)$ and $\<\alpha_n(\cdot),v\?\in L^1(0,T)$ with norms uniformly bounded in $n$. Moreover, $(u,\psi_n(\cdot))_U\in L^2(0,T)$ and $\psi_n\to 0$ weakly in $L^2(0,T;U)$ implies $(u,\psi_n(\cdot))_U\to 0$ weakly in $L^2(0,T)$.
Combined with \eqref{eq: reduction to R}, we conclude that it suffices to prove the lemma for $U=V=H=V^*=\R$ and $b=\one_D\in L^2(0,T)=L^2(0,T;\mathcal{L}_2(\R;\R))$.

Let $(w_n)$, $(\alpha_n)$, $(\psi_n)$  be as in the statement, now real-valued, and with $\psi=0$. Define $I^n(t)\coloneqq \int_0^t \one_D(s)\psi_n(s)w_n(s)\dd s$. We have to show that $\lim_{n\to\infty}\sup_{t\in[0,T]}|I^n(t)|=0$.

We use  an equidistant time discretization to approximate $w_n$.
For $\delta>0$ and $t\in[0,T]$, put $t_\delta\coloneqq \lfloor \frac{t}{\delta}\rfloor\delta$.
We have for all $n\in\N$ and $\delta>0$:
\begin{align}\label{eq: est I_3^n}
 |I^n(t)| &\leq \Big|\int_0^t \one_D(s)\psi_n(s)(w_n(s)-w_n(s_\delta))\dd s\Big|+\Big|\int_0^t \one_D(s)\psi_n(s)(w_n(s_\delta))\dd s\Big|\notag\\
&\leq \Big|\int_0^t \one_D(s)\psi_n(s)(w_n(s)-w_n(s_\delta))\dd s\Big|\notag\\
  &\quad+\sum_{l=0}^{\lfloor \frac{t}{\delta}\rfloor-1}\Big|\int_{l\delta}^{(l+1)\delta} \one_D(s)\psi_n(s)w_n(s_\delta)\dd s\Big|+\Big|\int_{t_\delta}^t \one_D(s)\psi_n(s)w_n(s_\delta)\dd s\Big|\notag\\
  &\eqqcolon  J_1^{n,\delta}(t)+\sum_{l=0}^{\lfloor \frac{T}{\delta}\rfloor-1}J_2^{n,\delta,l}+J_3^{n,\delta}(t).
\end{align}

We estimate each term.
Since $\psi_n\to 0$ weakly in $L^2(0,T)$, we have for all $\delta>0$ and $l\in\N$:
\begin{align}\label{eq: est J_3}
 J_2^{n,\delta,l}&=|w_n(l\delta)|\Big|\int_{l\delta}^{(l+1)\delta} \one_{D}(s)\psi_n(s)\dd s\Big|\leq C_w\Big|\int_{l\delta}^{(l+1)\delta} \one_{D}(s)\psi_n(s)\dd s\Big|\to 0 \quad\text{as } n\to\infty.
\end{align}

Furthermore, we have for all $n\in\N$:
\begin{align}\label{eq: est J_4}
  \sup_{t\in[0,T]}J_3^{n,\delta}(t)
  &\leq C_w \sup_{t\in[0,T]}\int_{t_\delta}^t|\psi_n(s)|\dd s
  \leq C_w C_\psi  \delta^{\frac{1}{2}} \to 0 \quad\text{as } \delta\downarrow 0,
\end{align}
where we used that $|t-t_\delta|<\delta$ for all $t\in[0,T]$. Note that the convergence is uniform in $n$.

Finally, we estimate $J_1^{n,\delta}(t)$ uniformly in $n$ and $t$. By the Cauchy--Schwarz inequality, we have for all $n \in\N$ and $\delta>0$:
\begin{align}\label{eq: est J_1}
  \sup_{t\in[0,T]}J_1^{n,\delta}(t)
  \leq \int_0^T |\psi_n(s)|\,|w_n(s)-w_n(s_\delta)|\dd s
  \leq C_\psi\|w_n(\cdot)-w_n(\cdot_\delta)\|_{L^2(0,T)}
\end{align}
To estimate further, we use an argument inspired by \cite[Lem.\ 3.3]{hongliliu21}. Note that $w_n(0)=0$ and
\begin{align}\label{eq: est w_n difference H-norm}
 \int_0^T |w_n(t)-w_n(t_\delta)|^2 \dd t&= \int_0^\delta |w_n(t)|^2 \dd t+\int_\delta^T |w_n(t)-w_n(t_\delta)|^2 \dd t \notag\\
 &\leq \delta C_w^2+\int_\delta^T |w_n(t)-w_n(t_\delta)|^2 \dd t.
\end{align}
For any $t\in[\delta,T]$, we can apply the chain rule \eqref{eq: Ito pardoux deterministic}
to $v_n^{t,\delta}(\cdot)\coloneqq w_n(\cdot)-w_n(t_\delta)=\int_{t_\delta}^\cdot \alpha_n(s)\dd s$ on $[t_\delta,T]$ and obtain for all $\tilde{t}\in[t_\delta,T]$:
\[
|w_n(\tilde{t})-w_n(t_\delta)|^2=2\int_{t_\delta}^{\tilde{t}}\alpha_n(s)(w_n(s)-w_n(t_\delta))\dd s.
\]
Applying the above expression with $\tilde{t}=t$ we estimate the second term from \eqref{eq: est w_n difference H-norm}:
\begin{align}\label{eq: est w_n difference H-norm term3}
 \int_\delta^T |w_n(t)-w_n(t_\delta)|^2  \dd t &=2\int_\delta^T \int_{t_\delta}^{t}\alpha_n(s)(w_n(s)-w_n(t_\delta))\dd s\dd t\notag\\
  &\leq 4C_w \int_\delta^T \int_{t-\delta}^{t}|\alpha_n(s)|\dd s\dd t\notag\\
  &\leq 4C_w\int_0^{T}\int_\delta^T  \one_{[s,(s+\delta)\wedge T]}(t)\dd t \, |\alpha_n(s)|\dd s\notag\\
  &\leq 4\delta C_wC_\alpha,
\end{align}
where we used that $\one_{[t-\delta,t]}(s)\leq \one_{[s,(s+\delta)\wedge T]}(t)$ for all $(s,t)\in[0,T]\times[\delta,T]$.
Combining \eqref{eq: est J_1}, \eqref{eq: est w_n difference H-norm} and \eqref{eq: est w_n difference H-norm term3} we conclude that for all $\delta>0$:
\begin{align}\label{eq: est J_1 final}
  \sup_{n\in\N}\sup_{t\in[0,T]}J_1^{n,\delta}(t)&\leq C_\psi\big(\delta C_w^2+4\delta C_w C_\alpha\big)^{\frac{1}{2}}.
\end{align}

Now let $\eps>0$.
According to \eqref{eq: est J_1 final} and \eqref{eq: est J_4}, fix $\delta>0$ sufficiently small such that we have
$\sup_{n\in\N}\sup_{t\in[0,T]}J_1^{n,\delta}(t)<\frac{\eps}{3}$ and $\sup_{n\in\N}\sup_{t\in[0,T]}J_3^{n,\delta}(t)<\frac{\eps}{3}.
$
Then, according to \eqref{eq: est J_3}, pick $N\in\N$ such that for all $n\geq N$: $J_2^{n,\delta,l}<\frac{\eps}{3\lfloor \frac{T}{\delta}\rfloor}$.
By \eqref{eq: est I_3^n}, we obtain $\sup_{t\in[0,T]}|I^n(t)|<\eps$ for all $n\geq N$. Thus  $\lim_{n\to\infty}\sup_{t\in[0,T]}|I^n(t)|=0$.
\end{proof}

Equipped with the lemma above, we now prove that condition \ref{it:suf2} of Theorem \ref{th: sufficient LDP criterion} is satisfied.
Note that the growth bounds on $B$ in Assumption \ref{ass: crit var set} contain $V$-norms (instead of merely $H$-norms), making it more difficult to apply Gronwall inequalities. To deal with this, the estimates from Lemma \ref{lem: embeddings rho_j} will be used.

\begin{proposition}\label{prop:cpt sublevel sets}
Suppose that Assumption \ref{ass: crit var set} holds and suppose that $(A,B)$ satisfies \eqref{eq: coercivity condition (A,B)}.
For $\psi\in L^2(0,T;U)$ let $u^\psi$ be the unique strong solution to \eqref{eq: skeleton eq}.
Then for any $K\geq 0$, the map $(S_K,\mathrm{weak})\to \MR(0,T)\colon  \psi\mapsto u^{\psi}$
is continuous.
\end{proposition}

\begin{proof}
Note that $S_K$ is weakly metrizable (as opposed to $L^2(0,T;U)$),
so we may verify sequential continuity.
Suppose that $\psi_n \to \psi$ weakly in $L^2(0,T;U)$ and write $w_n\coloneqq u^{\psi_n}-u^\psi$. We show that $w_n\to 0$ in $\MR(0,T)$.
For each $n\in\N$, $w_n$ is a strong solution to
\begin{align*}
\begin{cases}
  &w_n'+\bar{A_0} w_n=f_n+\big(\bar{B_0} w_n+g_n\big)\psi_n+{b}(\psi_n-\psi), \\
  &w_n(0)=0,
\end{cases}
\end{align*}
where $\bar{A_0}\coloneqq A_0 (u^\psi)$, $\bar{B_0}\coloneqq B_0(u^\psi)$ and
\begin{align*}
f_n&\coloneqq (A_0(u^\psi)-A_0(u^{\psi_n}))u^{\psi_n}+F(u^{\psi_n})-F(u^{\psi})\in L^2(0,T;V^*),\\
g_n&\coloneqq -(B_0(u^\psi)-B_0(u^{\psi_n}))u^{\psi_n}+G(u^{\psi_n})-G(u^\psi)\in L^2(0,T;\UH),\\
{b}&\coloneqq B(u^\psi)=B_0(u^\psi) u^\psi+G(u^\psi)+g\in L^2(0,T;\UH).
\end{align*}
By the chain rule \eqref{eq: Ito pardoux deterministic}, we have for all $t\in[0,T]$:
\begin{align}
 \frac{1}{2} \|w_n(t)\|_H^2 &= \int_0^t -\<\bar{A_0}w_n(s),w_n(s)\?+\<\bar{B_0}w_n(s)\psi_n(s),w_n(s)\?\dd s\notag\\
  &\qquad+\int_0^t\<f_n(s),w_n(s)\?+\<g_n(s)\psi_n(s),w_n(s)\?\dd s\notag\\
  &\qquad\qquad+\int_0^t \<{b}(s)(\psi_n(s)-\psi(s)),w_n(s)\?\dd s\notag\\
  &\eqqcolon I_1^n(t)+I_2^n(t)+I_3^n(t). \label{eq: gronwall prep w_n}
\end{align}
The strategy is now to use Lemma  \ref{lem: gronwall consequence} (Gronwall) for  deriving an estimate of the form
\[
\|w_n\|_{\MR(0,T)}^2\leq C\sup_{t\in[0,T]}|I_3^n(t)|,
\]
after which we will apply Lemma \ref{lem: I_3 to zero} to $I_3^n$ and obtain $w_n\to 0$ in $\MR(0,T)$.
Using the maximal regularity estimate \eqref{eq: MR estimate global sol} and boundedness of $(\psi_n)$ in $L^2(0,T;U)$, we put
\begin{equation}\label{eq: def N}
N\coloneqq  \|u^\psi\|_{\MR(0,T)}+\sup_{n\in\N}\|u^{\psi_n}\|_{\MR(0,T)} <\infty.
\end{equation}
Let  $\theta_{N,T}$, $M_{N,T}$ and  $C_{N,T}$ be as in Assumption \ref{ass: crit var set}. We estimate $I_1^n$ and $I_2^n$ appearing in \eqref{eq: gronwall prep w_n}. The coercivity of $(A_0,B_0)$ in Assumption \ref{it:ass2} gives
\begin{align}\label{eq: gronwall prep w_n I_1^n}
  I_1^n(t)&\leq \int_0^t-\<\bar{A_0}w_n(s),w_n(s)\?+\nn\bar{B_0}w_n(s)\nn_H\|\psi_n(s)\|_U\|w_n(s)\|_H\dd s \notag\\
  &\leq \int_0^t-\<\bar{A_0}w_n(s),w_n(s)\?+\frac{1}{2}\nn\bar{B_0}w_n(s)\nn_H^2+\frac{1}{2}\|\psi_n(s)\|_U^2\|w_n(s)\|_H^2\dd s \notag\\
  &\leq \int_0^t-\theta_{N,T}\|w_n(s)\|_V^2+(M_{N,T}+\frac{1}{2}\|\psi_n(s)\|_U^2)\|w_n(s)\|_H^2\dd s.
\end{align}
Moreover,
\begin{align}\label{eq: gronwall prep w_n I_2^n}
  I_2^n(t)&\leq \int_0^t\|f_n(s)\|_{V^*}\|w_n(s)\|_V+\nn g_n(s)\nn_H\|\psi_n(s)\|_U\|w_n(s)\|_H\dd s\notag\\
  &\leq \int_0^t\frac{1}{\theta_{N,T}}\|f_n(s)\|_{V^*}^2+\frac{\theta_{N,T}}{4}\|w_n(s)\|_V^2+\frac{1}{2}\nn g_n(s)\nn_H^2+\frac{1}{2}\|\psi_n(s)\|_U^2\|w_n(s)\|_H^2\dd s.
\end{align}
For $f_n$, Lemma \ref{lem: embeddings rho_j}\ref{it:emb3}\ref{it:emb5} gives for any $\sigma>0$:
\begin{align}\label{eq: gronwall prep w_n I_2^n f_n appl}
\|f_n\|_{L^2(0,t;V^*)}^2&\leq 2\|(A_0 (u^\psi)-A_0 (u^{\psi_n}))u^{\psi_n}\|_{L^2(0,t;V^*)}^2+2\|F(u^{\psi_n})-F(u^{\psi}) \|_{L^2(0,t;V^*)}^2\notag\\
&\leq 2C_{N,T}^2\int_0^t\|u^{\psi_n}\|_V^2\|w_n\|_H^2\dd s\\
&\qquad+2C_{N,T,\sigma}\int_0^t\left(1+\|u^\psi\|_{V}^2+\|u^{\psi_n}\|_{V}^2\right)\|w_n\|_{H}^2\dd s+2\sigma C_{N,T}^2\|w_n\|_{L^2(0,t;V)}^2.\notag
\end{align}
Similarly, $\|g_n\|_{L^2(0,t;\UH)}^2$ is bounded by the right-hand side of  \eqref{eq: gronwall prep w_n I_2^n f_n appl}, by Lemma \ref{lem: embeddings rho_j}\ref{it:emb3}\ref{it:emb5}.
Fix $\bar{\sigma}\coloneqq  \theta_{N,T}^2(4(2+\theta_{N,T})C_{N,T}^2)^{-1}>0$.
Combining \eqref{eq: gronwall prep w_n I_2^n} and \eqref{eq: gronwall prep w_n I_2^n f_n appl} yields
\begin{align}\label{eq: gronwall prep w_n I_2^n final}
  I_2^n(t) \leq &(\frac{2}{\theta_{N,T}}+1)\left(C_{N,T}^2\int_0^t\|u^{\psi_n}\|_V^2\|w_n\|_H^2\dd s
+C_{N,T,\bar{\sigma}}\int_0^t\left(1+\|u^\psi\|_{V}^2+\|u^{\psi_n}\|_{V}^2\right)\|w_n\|_{H}^2\dd s\right)\notag\\
&\qquad\qquad+(\frac{2}{\theta_{N,T}}+1)\bar{\sigma} C_{N,T}^2\|w_n\|_{L^2(0,t;V)}^2+\frac{\theta_{N,T}}{4}\|w_n\|_{L^2(0,t;V)}^2+\int_0^t\frac{1}{2}\|\psi_n\|_U^2\|w_n\|_H^2\dd s\notag\\
&  = \int_0^t h_n(s)\|w_n(s)\|_H^2\dd s+\frac{\theta_{N,T}}{2}\|w_n\|_{L^2(0,t;V)}^2,
\end{align}
where
\[
h_n(s)\coloneqq (\frac{2}{\theta_{N,T}}+1)\left(C_{N,T}^2\|u^{\psi_n}(s)\|_V^2+ C_{N,T,\bar{\sigma}}(1+\|u^\psi(s)\|_{V}^2+\|u^{\psi_n}(s)\|_{V}^2)  \right)+\frac{1}{2}\|\psi_n(s)\|_U^2.
\]
Note that $\sup_{n\in\N}\|h_n\|_{L^1(0,T)}<\infty$, by \eqref{eq: def N} and since $(\psi_n)\subset S_K$.
Now \eqref{eq: gronwall prep w_n I_1^n} and \eqref{eq: gronwall prep w_n I_2^n final} give
\[
I_1^n(t)+I_2^n(t)\leq -\frac{\theta_{N,T}}{2} \|w_n\|_{L^2(0,t;V)}^2+\int_0^t \Big(h_n(s)+M_{N,T}+\frac{1}{2}\|\psi_n(s)\|_U^2\Big)\|w_n(s)\|_{H}^2\dd s.
\]
Hence, combined with \eqref{eq: gronwall prep w_n}:
\begin{equation*}
\|w_n(t)\|_{H}^2\leq -\theta_{N,T}\|w_n\|_{L^2(0,t;V)}^2 +2\int_0^t \Big(h_n(s)+M_{N,T}+\frac{1}{2}\|\psi_n(s)\|_U^2\Big)\|w_n(s)\|_{H}^2\dd s+2\sup_{s\in[0,t]}|I_3^n(s)|.
\end{equation*}
Lemma \ref{lem: gronwall consequence} (Gronwall) gives for all $n\in\N$:
\begin{equation}\label{eq: gronwall w_n}
 \frac{1}{2}\|w_n\|_{\MR(0,T)}^2\leq \sup_{t\in[0,T]}\|w_n(t)\|_{H}^2+\|w_n\|_{L^2(0,T;V)}^2\leq 2(1+\frac{1}{\theta_{N,T}})\sup_{s\in[0,t]}|I_3^n(s)|\exp(2\kappa),
\end{equation}
with constant $\kappa\coloneqq \sup_{n\in\N}\left(\|h_n\|_{L^1(0,T)}+\frac{1}{2}\|\psi_n\|_{L^2(0,T;U)}^2\right)+M_{N,T}<\infty$.

By \eqref{eq: gronwall w_n}, it remains to show that $\lim_{n\to\infty}\sup_{t\in[0,T]}|I_3^n(t)|=0$. We use Lemma \ref{lem: I_3 to zero}. Note that $\sup_{n\in\N}\|w_n\|_{\MR(0,T)}\leq N$ by \eqref{eq: def N}, so we only have to verify boundedness of $(\alpha_n)\subset L^1(0,T;V^*)$, where $\alpha_n\coloneqq -\bar{A_0}w_n+f_n+(\bar{B_0}w_n+g_n)\psi_n+{b}(\psi_n-\psi)\in L^2(0,T;V^*)+L^1(0,T;H)\subset L^1(0,T;V^*)$. The last inclusion is continuous, so it suffices to prove boundedness of $(-\bar{A_0}w_n+f_n)\subset L^2(0,T;V^*)$ and  $(\beta_n)\coloneqq((\bar{B_0}w_n+g_n)\psi_n+{b}(\psi_n-\psi))\subset L^1(0,T;H)$.
Note that $\|(\bar{B_0}w_n+g_n)\psi_n\|_H\leq \nn \bar{B_0}w_n+g_n\nn_H\|\psi_n\|_U$ with $(\psi_n)$ bounded in $L^2(0,T;U)$ and similar for ${b}(\psi_n-\psi)$. Thus by the Cauchy--Schwarz inequality, if we show that $(\bar{B_0}w_n)$ and $(g_n)$ are bounded in $L^2(0,T;\UH)$, then boundedness of $(\beta_n)\subset L^1(0,T;H)$ follows (${b}\in L^2(0,T;\UH)$ does not depend on $n$).
By symmetry in Assumption \ref{it:ass3}, $\bar{B_0}w_n$ and $g_n$ can be estimated in the same way as $\bar{A_0}w_n\coloneqq A_0(u^\psi)w_n$ and $f_n\coloneqq (A_0(u^\psi)-A_0(u^{\psi_n}))u^{\psi_n}+F(u^{\psi_n})-F(u^{\psi})$, respectively. We provide the estimates for the latter here.
By Lemma \ref{lem: embeddings rho_j}\ref{it:emb2}\ref{it:emb3}, $\|A_0(u^\psi)w_n\|_{L^2(0,T;V^*)}\leq C_{N,T}N<\infty$ and  $\|(A_0(u^\psi)-A_0(u^{\psi_n}))u^{\psi_n}\|_{L^2(0,T;V^*)}\leq C_{N,T}N^2<\infty$.
Furthermore, Lemma \ref{lem: embeddings rho_j}\ref{it:emb4} gives
$\|F(u^{\psi_n})\|_{L^2(0,T;V^*)}\leq \tilde{C}_{N,T} (1+ N)<\infty$.
Finally, $F(u^{\psi})\in L^2(0,T;V^*)$ does not depend on $n$. We conclude that $(-\bar{A_0}w_n+f_n)$ is bounded in $L^2(0,T;V^*)$ and by the considerations above, $(\beta_n)$ is bounded in $L^1(0,T;H)$.
Lemma \ref{lem: I_3 to zero} thus yields $\lim_{n\to\infty}\sup_{t\in[0,T]}|I_3^n(t)|=0$ and \eqref{eq: gronwall w_n} gives $w_n\to 0$ in $\MR(0,T)$.
\end{proof}

\begin{remark}
Proposition \ref{prop:cpt sublevel sets} also ensures measurability of the map $\mathcal{G}^0\colon C([0,T];U_1)\to \MR(0,T)$ defined by \eqref{eq: def G^0}, as required in  Theorem \ref{th: sufficient LDP criterion}.
Note that $\{\int_0^\cdot \psi(s)\dd s:\psi \in L^2(0,T;U)\}=\{v\in W^{1,2}(0,T;U): v(0)=0\}\eqqcolon  W^{1,2}_0$.
By Sobolev embedding \cite[Corollary L.4.6]{HNVWvolume3}, $W^{1,2}_0$ embeds continuously into $C([0,T];U)$ ($W^{1,2}_0$ is a closed subspace of $W^{1,2}(0,T;U)$).
Hence Kuratowski's theorem \cite[Th.\ 15.1]{kechris95} gives $\mathcal{B}(W^{1,2}_0)\subset \mathcal{B}(C([0,T];U))$.
Moreover, $\gamma\colon W^{1,2}_0\to \MR(0,T)\colon \int_0^\cdot \psi(s)\dd s\mapsto u^\psi$ is continuous, since $\int_0^\cdot \psi_n(s)\dd s\to \int_0^\cdot \psi(s)\dd s$ in $W^{1,2}_0$ implies $\psi_n\to\psi$ in $L^2(0,T;U)$, and $L^2(0,T;U)\to\MR(0,T)\colon \psi\mapsto u^\psi$ is norm-continuous since it is weakly sequentially continuous by Proposition \ref{prop:cpt sublevel sets}. It follows that for $E\in \mathcal{B}(\MR(0,T))$, we have
\[
\left(\mathcal{G}^0\right)^{-1}(E)
=\begin{cases}
\gamma^{-1}(E)\in \mathcal{B}(W^{1,2}_0)\subset \mathcal{B}(C([0,T];U)) , &   0\notin E,\\
\gamma^{-1}(E)\cup (C([0,T];U)\setminus W^{1,2}(0))\in \mathcal{B}(C([0,T];U)), & 0\in E.
\end{cases}
\]
Since $U\into U_1$, Kuratowski's theorem yields $\mathcal{B}(C([0,T];U))\subset \mathcal{B}(C([0,T];U_1))$. Thus  $\mathcal{G}^0$ is measurable.
\end{remark}

\subsection{Stochastic continuity criterion}\label{ss:stochcont}

It remains to verify the stochastic continuity criterion \ref{it:suf3} of Theorem \ref{th: sufficient LDP criterion}.
Before we prove that \ref{it:suf3} is satisfied, we first derive some stochastic bounds which we will later apply to  $X^\eps\coloneqq \G^\eps\big(\tilde{W}_1(\cdot)+\frac{1}{\sqrt{\eps}}\int_0^{\cdot} \Psi^\eps(s)\dd s\big)$. In the next lemma we use a stochastic Gronwall lemma as in \cite{AV22variational} to avoid further growth bound assumptions on $B$.

\begin{lemma}\label{lem: X^eps bdd}
Suppose that Assumption \ref{ass: crit var set} holds and suppose that $(A,B)$ satisfies \eqref{eq: coercivity condition (A,B)}.
Let $K>0$, $(\Psi^\eps)_{0< \eps<\frac{1}{2}}\subset \mathcal{A}_K$ and let $x\in H$.
For $\eps\in (0,\frac{1}{2})$, let $X^\eps$ be a strong solution to
  \begin{align}
  &\begin{cases}
    \dd X^\eps(t)=\left(-A(t,X^\eps(t)) +B(t,X^\eps(t))\Psi^\eps(t)\right)\dd t+\sqrt{\eps}B(t,X^\eps(t))\dd W(t), \quad t\in[0,T],\\
    X^\eps(0)=x,
  \end{cases}\label{eq: X^eps SPDE}
  \end{align}
Then there exists $C>0$ such that for all $\gamma>0$,
 \[
 \begin{cases}
 \sup_{\eps\in(0,\frac{1}{2})}\P(\|X^\eps\|_{\MR(0,T)}>\gamma)\leq\frac{C}{\gamma^2},\\
 \sup_{\eps\in(0,\frac{1}{2})}\P(\|B(\cdot,X^\eps(\cdot))\|_{L^2(0,T;\UH)}>\gamma)\leq\frac{C}{\gamma^2}.
 \end{cases}
 \]
 The constant $C$ depends only on $x,K,T$ and $\phi,M,\theta$ from \eqref{eq: coercivity condition (A,B)}.
\end{lemma}
\begin{proof}
By the It\^o formula \eqref{eq: Ito pardoux},
 by \eqref{eq: coercivity condition (A,B)} and since $\eps<\frac{1}{2}$, we have a.s.\ for all $t\in[0,T]$:
  \begin{align}
    \|X^\eps(t)\|_H^2-\|&x\|_H^2=2\int_0^t\<-A(s,X^\eps(s)),X^\eps(s)\?+\<B(s,X^\eps(s))\Psi^\eps(s),X^\eps(s)\?\dd s\notag\\
    &\qquad\qquad+2\sqrt{\eps}\int_0^t \<X^\eps(s),B(s,X^\eps(s))\dd W(s)\?+\eps\int_0^t \nn B(s,X^\eps(s))\nn_H^2\dd s\notag\\
    &\leq 2\int_0^t-\frac{1}{2}\nn B(s,X^\eps(s))\nn_H^2-\theta\|X^\eps(s)\|_V^2+M\|X^\eps(s)\|_H^2+|\phi(s)|^2\dd s \notag\\
    &\qquad +2\int_0^t\<B(s,X^\eps(s))\Psi^\eps(s),X^\eps(s)\?\dd s\notag\\
    &\qquad\qquad +\eps\int_0^t \nn B(s,X^\eps(s))\nn_H^2\dd s+2\sqrt{\eps}\int_0^r \<X^\eps(s),B(s,X^\eps(s)(\cdot)\?\dd W(s)\notag\\
    &\leq -\int_0^t \nn B(s,X^\eps(s))\nn_H^2\dd s-2\theta\|X^\eps\|_{L^2(0,t;V)}^2+\int_0^t 2M\|X^\eps(s)\|_H^2\dd s+2\|\phi\|_{L^2(0,t)}^2\notag\\
    &\qquad+2\int_0^t \frac{1}{8}\nn B(s,X^\eps(s))\nn_H^2+2\|\Psi^\eps(s)\|_U^2\|X^\eps(s)\|_H^2\dd s\notag\\
    &\qquad\qquad+\frac{1}{2}\int_0^t \nn B(s,X^\eps(s))\nn_H^2\dd s+2\sqrt{\eps}\int_0^t \<X^\eps(s),B(s,X^\eps(s)(\cdot)\?\dd W(s)\notag\\
    &= -\frac{1}{4}\|B(\cdot,X^\eps(\cdot))\|_{L^2(0,t;\UH)}^2
    -2\theta\|X^\eps\|_{L^2(0,t;V)}^2+2\|\phi\|_{L^2(0,t)}^2\notag\\
    &\qquad+\int_0^t 2(M+2\|\Psi^\eps(s)\|_U^2)\|X^\eps(s)\|_H^2\dd s +2\sqrt{\eps}\int_0^t \<X^\eps(s),B(s,X^\eps(s)(\cdot)\?\dd W(s). \label{eq:prep gronwall X^eps}
  \end{align}
We conclude that $
y_\eps(t)\leq h(t)+\int_0^ty_\eps(s)a_\eps(s)\dd s+2\sqrt{\eps}\int_0^t \<X^\eps(s),B(s,X^\eps(s)(\cdot)\?\dd W(s)$, where
\begin{align*}
y_\eps(t)\coloneqq\|X^\eps(t)\|_H^2+2\theta\|X^\eps\|_{L^2(0,t;V)}^2+\frac{1}{4}\| B(\cdot,X^\eps(\cdot))\|_{L^2(0,t;\UH)}^2, \\
h(t)\coloneqq \|x\|_H^2+2\|\phi\|_{L^2(0,T)}^2, \qquad a_\eps(t)\coloneqq  2(M+2\|\Psi^\eps(t)\|_U^2).                                                    
\end{align*}
Now  the stochastic Gronwall inequality \cite[Cor.\ 5.4b), (50)]{geiss24} (with $R\coloneqq2MT+4K^2$) 
gives 
\begin{align*}
\P\Big(\sup_{t\in[0,T]}y_\eps(t)>\gamma\Big)\leq \frac{\exp(2MT+4K^2)}{\gamma}\E[h(T)] \leq\frac{\exp(2MT+4K^2)}{\gamma}\big(\|x\|_H^2+2\|\phi\|_{L^2(0,T)}^2\big)
\end{align*}
for all $\gamma>0$, where we used that $\|\Psi^\eps\|_{L^2(0,T;U)}\leq K$ a.s.\ since $(\Psi^\eps)\subset \mathcal{A}_K$. Using
\begin{align*}
&\Big\{\sup_{t\in[0,T]}y_\eps(t)>\gamma\Big\}\\
&\supset \Big\{\|X^\eps\|_{C([0,T];H)}^2+2\theta\|X^\eps\|_{L^2(0,T;V)}^2>{2\gamma}\Big\}\cup \Big\{\| B(\cdot,X^\eps(\cdot))\|_{L^2(0,T;\UH)}^2>4\gamma\Big\}\\
&\supset \Big\{\|X^\eps\|_{\MR(0,T)}^2>\frac{4\gamma}{1\wedge 2\theta}\Big\}\cup \Big\{\| B(\cdot,X^\eps(\cdot))\|_{L^2(0,T;\UH)}^2>4\gamma\Big\}
\end{align*}
and putting $C\coloneqq \frac{4}{1\wedge 2\theta}\exp(2MT+4K^2)(\|x\|_H^2+2\|\phi\|_{L^2(0,T)}^2)$, yields for all $\eps\in(0,\frac{1}{2})$:
\begin{align*}
&\P(\|X^\eps\|_{\MR(0,T)}^2>\gamma)\leq\frac{C}{\gamma},\quad \P(\|B(\cdot,X^\eps(\cdot))\|_{L^2(0,T;\UH)}^2>\gamma)\leq\frac{C}{\gamma}.
\end{align*}
Consequently, we have  $\P(\|X^\eps\|_{\MR(0,T)}>\gamma)=\P(\|X^\eps\|_{\MR(0,T)}^2>\gamma^2)\leq \frac{C}{\gamma^2}$
and in the same way, $\P(\|B(\cdot,X^\eps(\cdot))\|_{L^2(0,T;\UH)}>\gamma)\leq\frac{C}{\gamma^2}$, uniformly in $\eps\in(0,\frac{1}{2})$.
\end{proof}

We now prove that condition \ref{it:suf3} of Theorem \ref{th: sufficient LDP criterion} is satisfied.

\begin{proposition}\label{prop: X^eps weak conv}
Suppose that Assumption \ref{ass: crit var set} holds and suppose that $(A,B)$ satisfies \eqref{eq: coercivity condition (A,B)}.
Let $(\Psi^\eps)_{0<\eps<\frac{1}{2}}\subset \mathcal{A}_K$  for some $K>0$ and let $x\in H$.
For $\eps\in(0,\frac{1}{2})$, let $X^\eps$ and $u^{\eps}$ be defined by
\[
X^{\eps}\coloneqq \mathcal{G}^{\eps}\Big({\tilde{W}_1}(\cdot)+\frac{1}{\sqrt{\eps}}\int_0^{\cdot}{\Psi}^{\eps}(s)\dd s\Big), \quad u^{\eps}\coloneqq \mathcal{G}^0\Big(\int_0^\cdot \Psi^\eps(s)\dd s\Big),
\]
where $\mathcal{G}^{\eps}:C([0,T];U_1)\to\MR(0,T)$ is the measurable map from Lemma \ref{lem: yamada-watanabe} for $\eps>0$, $\G^0$ is defined by \eqref{eq: def G^0} and $\tilde{W}_1$ by \eqref{eq: Q-Wiener}.
Then $X^\eps- u^{\eps}\to 0$ in probability in $\MR(0,T)$ as $\eps\downarrow0$.
\end{proposition}

\begin{proof}
We will apply It\^o's formula and Assumption \ref{ass: crit var set}. However, because the estimates in Assumption \ref{ass: crit var set} are $n$-dependent, below we use a cut-off argument to reduce to processes that are bounded by $n$ in $H$-norm.

By Definition \ref{def: S_K A_K}, we have a.s. $\|\Psi^\eps\|_{L^2(0,T;U)}\leq K<\infty$. Thus, recalling \eqref{eq: def G^0}, we have for  a.e. $\om\in\Om$: $u^\eps(\om)=u^{\Psi_\eps(\om)}$, where the latter is the unique strong solution  (Theorem \ref{th: global well posedness skeleton}) to \eqref{eq: skeleton eq} with $\psi=\Psi^\eps(\om)\in S_K$. Furthermore, the maximal regularity estimate \eqref{eq: MR estimate global sol} gives
\begin{align}\label{eq: def esssup N}
N\coloneqq \esssup_{\om\in\Om}\sup_{\eps\in(0,\frac{1}{2})}\|u^\eps(\om)\|_{\MR(0,T)}<\infty.
\end{align}

On the other hand, for $X^\eps$ we do not have a.s. $\sup_{\eps\in(0,\frac{1}{2})}\|X^\eps\|_{C([0,T];H)}<\infty$, but we do have the boundedness in probability from Lemma \ref{lem: X^eps bdd}.
For $\eps\in(0,\frac{1}{2})$ and $n\in\N$, define
\[
E_{n,\eps}\coloneqq \{\|X^\eps\|_{\MR(0,T)}\leq n\}\cap \{\|u^{\eps}\|_{\MR(0,T)}\leq N\}.
\]
By Lemma \ref{lem: yamada-watanabe}, $X^\eps$ is a strong solution to \eqref{eq: X^eps SPDE},
so thanks to Lemma \ref{lem: X^eps bdd} and \eqref{eq: def esssup N},
\[
\P(E_{n,\eps}^c)= \P(\|X^\eps\|_{\MR(0,T)}>n) \leq\frac{C}{n^2},
\]
where $C$ is a constant independent of $\eps$.
Hence, for all $\eps\in(0,\frac{1}{2})$ and $n\in\N$:
\begin{align*}
\P(\|X^\eps- u^{\eps}\|_{\MR(0,T)}>\gamma)
&\leq \P(\{\|X^\eps- u^{\eps}\|_{\MR(0,T)}>\gamma\}\cap E_{n,\eps})+\P(E_{n,\eps}^c)\\
&\leq \P(\{\|X^\eps- u^{\eps}\|_{\MR(0,T)}>\gamma\}\cap E_{n,\eps})+\frac{C}{n^2}.
\end{align*}
Therefore, to have the stated convergence in probability, it suffices to prove that for any $\delta>0$ and any large enough $n\in\N$:
\begin{equation}\label{eq: X^eps weak conv to show 2}
  \lim_{\eps\downarrow 0}\P(\{\|X^\eps- u^{\eps}\|_{\MR(0,T)}>\delta\}\cap E_{n,\eps})=0.
\end{equation}

Let $n\geq N$ be arbitrary, where $N$ is given by \eqref{eq: def esssup N}.
We prove \eqref{eq: X^eps weak conv to show 2}.
By the It\^o formula \eqref{eq: Ito pardoux},
we have for all $t\in[0,T]$:
  \begin{align*}
    \|X^\eps(t)-u^{\eps}(t)\|_H^2
    &=2\int_0^t\<-A(s,X^\eps(s))+A(s,u^{\eps}(s)),X^\eps(s)-u^{\eps}(s)\?\dd s\\
    &\qquad+2\int_0^t\<\big(B(s,X^\eps(s))-B(s,u^{\eps}(s))\big)\Psi^\eps(s),X^\eps(s)-u^{\eps}(s)\?\dd s\\
    &\qquad+\eps\int_0^t \nn B(s,X^\eps(s))\nn_H^2\dd s\\
    &\qquad+2\sqrt{\eps}\int_0^t \<X^\eps(s)-u^{\eps}(s),B(s,X^\eps(s))\dd W(s)\?\\
    &\eqqcolon {I}_1^\eps(t)+{I}_2^\eps(t)+{I}_3^\eps(t)+{I}_4^\eps(t).
    \end{align*}
    Below we derive an estimate of the form
    \begin{equation}\label{eq: to show I^1 I^2}
    {I}_1^\eps(t)+{I}_2^\eps(t)\leq -\theta_{n,T}\|X^\eps-u^{\eps}\|_{L^2(0,t;V)}^2+\int_0^t |h_{n,\eps}(s)|\|X^\eps(s)-u^{\eps}(s)\|_H^2 \dd s
    \end{equation}
    that holds a.s.\ on the set $E_{n,\eps}$, for every $t\in[0,T]$ and $\eps\in(0,\frac{1}{2})$. Here, $\theta_{n,T}$ is a constant and a.s.\ $h_{n,\eps}\in L^1(0,T)$, with $\alpha_n\coloneqq \sup_{\eps\in(0,\frac{1}{2})}\esssup_\Om \|h_{n,\eps}\one_{E_{n,\eps}}\|_{L^1(0,T)}<\infty$.
    Then, a.s.\ on $E_{n,\eps}$,
    \[
    \|X^\eps(t)-u^{\eps}(t)\|_H^2\leq -\theta_{n,T}\|X^\eps-u^{\eps}\|_{L^2(0,t;V)}^2+{I}_3^\eps(t)+\sup_{r\in[0,t]}{I}_4^\eps(r)+\int_0^t |h_{n,\eps}(s)|\|X^\eps(s)-u^{\eps}(s)\|_H^2 \dd s,
    \]
    so Lemma \ref{lem: gronwall consequence} (Gronwall)
    gives pointwise in a.e. $\omega\in E_{n,\eps}$:
    \begin{equation*}
    \|X^\eps-u^{\eps}\|_{C([0,T];H)}^2+\|X^\eps-u^{\eps}\|_{L^2(0,T;V)}^2\leq (1+{\theta_{n,T}}^{-1})\exp(\alpha_n)\Big({I}_3^\eps(T)+\sup_{t\in[0,T]}|{I}_4^\eps(t)|\Big).
    \end{equation*}
Putting $c_{n}\coloneqq 2(1+{\theta_{n,T}}^{-1})\exp({\alpha_n})$, we thus have $\|X^\eps-u^{\eps}\|_{\MR(0,T)}^2\leq c_{n}\Big({I}_3^\eps(T)+\sup_{t\in[0,T]}|{I}_4^\eps(t)|\Big)$ a.s.\ on $E_{n,\eps}$,  and therefore,
    \[
    \P(\{\|X^\eps- u^{\eps}\|_{\MR(0,T)}^2>\delta\}\cap E_{n,\eps})\leq \sum_{i=3}^4\P(\{\sup_{t\in[0,T]}|{I}_i^\eps(t)|>\frac{\delta}{2c_n}\}\cap E_{n,\eps}).
    \]
Hence, after we have proved \eqref{eq: to show I^1 I^2}, for \eqref{eq: X^eps weak conv to show 2}, it suffices to prove two convergences in probability:
    \begin{equation}\label{eq: to show I^3}
    \lim_{\eps\downarrow 0}\P({I}_3^\eps(T)> \delta)=0 \text{ for any } \delta>0,
    \end{equation}
    \begin{equation}\label{eq: to show I^4}
    \lim_{\eps\downarrow 0}\P(\sup_{t\in[0,T]}|{I}_4^\eps(t)|>\delta)=0 \text{ for any } \delta>0.
    \end{equation}
    All in all, recalling that we reduced the original problem to proving \eqref{eq: X^eps weak conv to show 2}, by the reasoning above it remains to establish \eqref{eq: to show I^1 I^2}, \eqref{eq: to show I^3} and \eqref{eq: to show I^4}.

    Let us prove \eqref{eq: to show I^1 I^2}. Recall that  $A(t,v)=A_0(t,v)v-F(t,v)-f$ and  $B(t,v)=B_0(t,v)v+G(t,v)+g$, see Assumption \ref{it:ass1}. We have pointwise on $E_{n,\eps}$, for all $\eps\in(0,\frac{1}{2})$:
    \begin{align}
      \frac{1}{2}({I}_1^\eps(t)&+{I}_2^\eps(t))=\int_0^t \<-A_0(s,u^{\eps}(s))(X^\eps(s)-u^{\eps}(s)),X^\eps(s)-u^{\eps}(s)\?\dd s\notag\\
    &\qquad+ \int_0^t\<\big(A_0(s,u^{\eps}(s))-A_0(s,X^\eps(s))\big)X^\eps(s),X^\eps(s)-u^{\eps}(s)\?\dd s\notag\\
    &\qquad + \int_0^t\<F(X^\eps(s))-F(u^{\eps}(s)),X^\eps(s)-u^{\eps}(s)\?\dd s\notag\\
    &\qquad+\int_0^t \<B_0(s,u^{\eps}(s))(X^\eps(s)-u^{\eps}(s))\Psi^\eps(s),X^\eps(s)-u^{\eps}(s)\?\dd s\notag\\
    &\qquad+\int_0^t \<\big(B_0(s,X^\eps(s))-B_0(s,u^{\eps}(s))\big)X^\eps(s)\Psi^\eps(s),X^\eps(s)-u^{\eps}(s)\?\dd s\notag\\
    &\qquad+\int_0^t\<\big(G(X^\eps(s))-G(u^{\eps}(s))\big)\Psi^\eps(s),X^\eps(s)-u^{\eps}(s)\?\dd s\notag\\
    &\leq \int_0^t \<-A_0(s,u^{\eps}(s))(X^\eps(s)-u^{\eps}(s)),X^\eps(s)-u^{\eps}(s)\?\dd s\notag\\
    &\qquad+\int_0^t\frac{1}{2} \nn B_0(s,u^{\eps}(s))(X^\eps(s)-u^{\eps}(s))\nn_H^2+\frac{1}{2}\|\Psi^\eps(s)\|_U^2\|X^\eps(s)-u^{\eps}(s)\|_H^2\dd s\notag\\
    &\qquad+ \int_0^t\|\big(A_0(s,u^{\eps}(s))-A_0(s,X^\eps(s))\big)X^\eps(s)\|_{V^*}\|X^\eps(s)-u^{\eps}(s)\|_V\dd s\notag\\
    &\qquad + \int_0^t\|F(X^\eps(s))-F(u^{\eps}(s))\|_{V^*}\|X^\eps(s)-u^{\eps}(s)\|_V\dd s\notag\\
    &\qquad+\int_0^t \nn\big(B_0(s,X^\eps(s))-B_0(s,u^{\eps}(s))\big)X^\eps(s)\nn_H\|\Psi^\eps(s)\|_U\|X^\eps(s)-u^{\eps}(s)\|_H\dd s\notag\\
    &\qquad+\int_0^t\nn G(X^\eps(s))-G(u^{\eps}(s))\nn_H\|\Psi^\eps(s)\|_U \|X^\eps(s)-u^{\eps}(s)\|_H\dd s\notag\\
    &\leq \int_0^t -\theta_{n,T}\|X^\eps(s)-u^{\eps}(s)\|_V^2\dd s\notag\\
    &\qquad+\int_0^t \left(M_{n,T}+\frac{1}{2}\|\Psi^\eps(s)\|_U^2\right)\|X^\eps(s)-u^{\eps}(s)\|_H^2\dd s\notag\\
    &\qquad+ \int_0^tC_\sigma\|\big(A_0(s,u^{\eps}(s))-A_0(s,X^\eps(s))\big)X^\eps(s)\|_{V^*}^2+\sigma\|X^\eps(s)-u^{\eps}(s)\|_V^2\dd s\notag\\
    &\qquad + \int_0^tC_\sigma\|F(X^\eps(s))-F(u^{\eps}(s))\|_{V^*}^2+\sigma\|X^\eps(s)-u^{\eps}(s)\|_V^2\dd s\notag\\
    &\qquad+\int_0^t \frac{1}{2} \nn\big(B_0(s,X^\eps(s))-B_0(s,u^{\eps}(s))\big)X^\eps(s)\nn_H^2+\frac{1}{2}\|\Psi^\eps(s)\|_U^2\|X^\eps(s)-u^{\eps}(s)\|_H^2\dd s\notag\\
    &\qquad+\frac{1}{2}\int_0^t\nn G(X^\eps(s))-G(u^{\eps}(s))\nn_H^2+\frac{1}{2}\|\Psi^\eps(s)\|_U^2 \|X^\eps(s)-u^{\eps}(s)\|_H^2\dd s\notag\\
    &\eqqcolon  -\theta_{n,T}\|X^\eps-u^{\eps}\|_{L^2(0,t;V)}^2+{J}_{1}^{\eps}(t)+{J}_{2}^{\eps,\sigma}(t)+{J}_{3}^{\eps,\sigma}(t)+{J}_{4}^{\eps}(t)+{J}_{5}^{\eps}(t) \label{eq: prep 1 pointwise gronwall X^eps-u^psi}
    \end{align}
for any $\sigma>0$, where $C_\sigma\coloneqq \frac{1}{4\sigma}$ from Young's inequality and $\theta_{n,T}$ and $M_{n,T}$ are the constants from the local coercivity of $(A_0,B_0)$ in Assumption \ref{it:ass2}.

Next, we estimate the terms of \eqref{eq: prep 1 pointwise gronwall X^eps-u^psi}.
    ${J}_1^\eps$ is already in the desired form for application of Gronwall's inequality. Moreover,
    Lemma \ref{lem: embeddings rho_j}\ref{it:emb3} yields
    \begin{align}
    J_2^{\eps,\sigma}(t)
    &\leq  C_\sigma C_{n,T}^2 \int_0^t\|X^\eps(s)-u^{\eps}(s)\|_H^2\|X^\eps(s)\|_V^2\dd s +\sigma \|X^\eps-u^{\eps}\|_{L^2(0,t;V)}^2, \label{eq: gronwall est J_2}\\
    J_4^\eps(t)&\leq  \frac{1}{2} \int_0^t\|X^\eps(s)-u^{\eps}(s)\|_H^2\big( C_{n,T}^2\|X^\eps(s)\|_V^2+\|\Psi^\eps(s)\|_U^2\big)\dd s.\label{eq: gronwall est J_4}
    \end{align}
    Similarly,
    Lemma \ref{lem: embeddings rho_j}\ref{it:emb5} gives for any $\tilde{\sigma}>0$:
    \begin{align}\label{eq: gronwall est J_3}
    J_3^{\eps,\sigma}(t)
    &\leq  C_\sigma C_{n,T,\tilde{\sigma}} \int_0^t\|X^\eps(s)-u^{\eps}(s)\|_H^2\big(1+\|X^\eps(s)\|_V^2+\|u^{\eps}(s)\|_V^2\big)\dd s \notag\\
    &\qquad+C_\sigma\tilde{\sigma}C_{n,T}^2\|X^\eps-u^{\eps}\|_{L^2(0,t;V)}^2+\sigma \|X^\eps-u^{\eps}\|_{L^2(0,t;V)}^2
    \end{align}
    for some constant $C_{n,T,\tilde{\sigma}}>0$,
    and
    \begin{align}\label{eq: gronwall est J_5}
    J_5^\eps(t)&\leq   \frac{1}{2}\int_0^t\|X^\eps(s)-u^{\eps}(s)\|_H^2\left(C_{n,T,\tilde{\sigma}}(1+\|X^\eps(s)\|_V^2+\|u^{\eps}(s)\|_V^2)+\|\Psi^\eps(s)\|_U^2\right)\dd s \notag\\
    &\qquad+\frac{1}{2}\tilde{\sigma}C_{n,T}^2\|X^\eps-u^{\eps}\|_{L^2(0,t;V)}^2.
    \end{align}
    Now we fix $\sigma\coloneqq \frac{\theta_{n,T}}{8}$. Then we fix $\tilde{\sigma}\coloneqq \frac{\theta_{n,T}}{8C_{n,T}^2(C_\sigma\vee\frac{1}{2})}$.
    Combining estimates \eqref{eq: gronwall est J_2}-\eqref{eq: gronwall est J_5} with \eqref{eq: prep 1 pointwise gronwall X^eps-u^psi}, we obtain
    \begin{align*}
     {I}_1^\eps(t)+{I}_2^\eps(t)&\leq -\theta_{n,T}\|X^\eps-u^{\eps}\|_{L^2(0,t;V)}^2+\int_0^t h_{n,\eps}(s)\|X^\eps(s)-u^{\eps}(s)\|_H^2\dd s,
    \end{align*}
    where $h_{n,\eps}$ is of the form
    \begin{align}
    h_{n,\eps}(s)&=
    C_{n,T,\sigma,\tilde{\sigma}}\left(1+\|\Psi^\eps(s)\|_U^2+\|X^\eps(s)\|_V^2+\|u^{\eps}(s)\|_V^2\right), \label{eq: bound on h_n,eps}
    \end{align}
    for a constant $C_{n,T,\sigma,\tilde{\sigma}}>0$. Now, a.s.\ $\Psi^\eps\in L^2(0,T;U)$ and a.s.\ $X^\eps,u^{\eps}\in L^2(0,T;V)$, thus a.s.\ $h_{n,\eps}\in L^1(0,T)$.  By definition of $E_{n,\eps}$, by \eqref{eq: bound on h_n,eps} and since $(\Psi^\eps)\subset \mathcal{A}_K$, we have a.s.\ $\|h_{n,\eps}\one_{E_{n,\eps}}\|_{L^1(0,T)}\leq C_{n,T,\sigma,\tilde{\sigma}}(T+K^2+2n^2)$ for every $\eps\in(0,\frac{1}{2})$. Thus, $h_{n,\eps}$ has all required properties and \eqref{eq: to show I^1 I^2} is satisfied a.s.\ on the set $E_{n,\eps}$, for every $\eps\in(0,\frac{1}{2})$, as desired.

    Regarding \eqref{eq: to show I^3}, by Lemma \ref{lem: X^eps bdd} we have for any (fixed) $\delta>0$:
    \begin{align*}
    \lim_{\eps\downarrow 0}\P({I}_3^\eps(T)>\delta)
    =\lim_{\eps\downarrow 0}\P( \| B(\cdot,X^\eps(\cdot))\|_{L^2(0,T;\UH)}^2>{\delta}{\eps}^{-1})
    \leq \lim_{\eps\downarrow 0}{C\eps}{\delta}^{-1}=0.
    \end{align*}

     It remains to prove \eqref{eq: to show I^4}.
    Note that ${I}_4^\eps$ is a continuous local martingale (using Lemma \ref{lem: X^eps bdd}) starting at zero, with $[{I}_4^\eps](T)=\int_0^T\eps\|\<X^\eps(s)-u^{\eps}(s),B(s,X^\eps(s))(\cdot)\?\|_{\mathcal{L}_2(U,\R)}^2\dd s$, where $[I_4^\eps]$ denotes the  quadratic variation.
    Thus, by  \cite[Prop.\ 18.6]{kallenberg21}, \eqref{eq: to show I^4} is equivalent to
    \begin{equation}\label{eq: I^4 sufficient convergence}
    \lim_{\eps\downarrow0}\P\Big(\eps\int_0^T\|\<X^\eps(s)-u^{\eps}(s),B(s,X^\eps(s))(\cdot)\?\|_{\mathcal{L}_2(U,\R)}^2\dd s>\delta\Big)=0 \text{ for all } \delta>0.
    \end{equation}
    We prove the latter. We have for all $\delta>0$ and $\eps\in(0,\frac{1}{2})$:
    \begin{align}
    \P\Big(\eps\int_0^T\|\<X^\eps(s)-&u^{\eps}(s),B(s,X^\eps(s))(\cdot)\?\|_{\mathcal{L}_2(U,\R)}^2\dd s>\delta\Big)\notag\\
    &\leq\P\Big(\int_0^T \|X^\eps(s)-u^{\eps}(s)\|_H^2\nn B(s,X^\eps(s))\nn_H^2\dd s>{\delta}{\eps}^{-1}\Big)\notag\\
    &\leq \P\Big(\|X^\eps-u^{\eps}\|_{C([0,T];H)}^2 \int_0^T \nn B(s,X^\eps(s))\nn_H^2\dd s>{\delta}{\eps}^{-1}\Big)\notag\\
    &\leq \P\big(\|X^\eps-u^{\eps}\|_{C([0,T];H)}>({\delta}{\eps}^{-1})^{\frac{1}{4}}\big)+ \P\big(\| B(\cdot,X^\eps(\cdot))\|_{L^2(0,T;\UH)} >({\delta}{\eps}^{-1})^{\frac{1}{4}}\big).\label{eq: weak L1 estimates}
    \end{align}
    Due to Lemma \ref{lem: X^eps bdd}, we have
    \begin{align}
     \P\big(\|X^\eps-u^{\eps}\|_{C([0,T];H)}>({\delta}{\eps}^{-1})^{\frac{1}{4}}\big)
      & \leq \P\big(\|X^\eps\|_{C([0,T];H)}>\tfrac{1}{2}({\delta}{\eps}^{-1})^{\frac{1}{4}}\big)+\P\big(\|u^{\eps}\|_{C([0,T];H)}>\tfrac{1}{2}({\delta}{\eps}^{-1})^{\frac{1}{4}}\big)\notag\\
     & \leq 4C({\eps}{\delta}^{-1})^{\tfrac{1}{2}}+\P\big(\|u^{\eps}\|_{C([0,T];H)}>\frac{1}{2}({\delta}{\eps}^{-1})^{\frac{1}{4}}\big),
      \label{eq: fact difference}
      \end{align}
      \begin{align}
    \P\big(\|B(\cdot,X^\eps(\cdot))\|_{L^2(0,T;\UH)}> ({\delta}{\eps}^{-1})^{\frac{1}{4}}\big)\leq C({\eps}{\delta}^{-1})^{\frac{1}{2}}.\hspace{5cm} \label{eq: fact B(X^eps)}
    \end{align}
    Note that  $\P\big(\|u^{\eps}\|_{C([0,T];H)}>\tfrac{1}{2}({\delta}{\eps}^{-1})^{\frac{1}{4}}\big)=0$ for all $\eps\in(0,\frac{\delta}{16N^4}\wedge \frac{1}{2})$ by \eqref{eq: def esssup N}.
Thus, combining \eqref{eq: fact difference}, \eqref{eq: fact B(X^eps)} and continuing from \eqref{eq: weak L1 estimates}, we see that for all $\eps\in(0,\frac{\delta}{16N^4}\wedge \frac{1}{2})$:
\[
\P\Big(\eps\int_0^T\|\<X^\eps(s)-u^{\eps}(s),B(s,X^\eps(s))(\cdot)\?\|_{\mathcal{L}_2(U,\R)}^2\dd s>\delta\Big)\leq 5C({\eps}{\delta}^{-1})^{\frac{1}{2}}.
\]
Letting $\eps\downarrow 0$ we arrive at \eqref{eq: I^4 sufficient convergence}.
\end{proof}

\subsection{Proof of Theorem \ref{th: main LDP theorem}}

Proving Theorem \ref{th: main LDP theorem} is now only a matter of combining.

\begin{proof}
  We verify the criteria of Theorem \ref{th: sufficient LDP criterion}. Note that $\mathcal{E}\coloneqq \MR(0,T)$ is Polish.
  Define $\mathcal{G}^0$ by \eqref{eq: def G^0} and for $\eps>0$, let $\mathcal{G}^\eps$ be the measurable map from Lemma \ref{lem: yamada-watanabe}.
  Now, \ref{it:suf1} holds by Lemma \ref{lem: yamada-watanabe}, \ref{it:suf2} holds by Proposition \ref{prop:cpt sublevel sets} and \ref{it:suf3} holds by Proposition \ref{prop: X^eps weak conv}. The proof is complete.
\end{proof}

Lastly, a small remark.

\begin{remark}\label{rem: LLN}
The LDP of Theorem \ref{th: main LDP theorem} implies the following Strong Law of Large Numbers: we have $Y^\eps\to Y^0$ a.s.\ as $\eps\downarrow0$, where $Y^0$ solves \eqref{eq: SPDE Y^eps} with $\eps=0$, i.e. with only the drift term. This follows from the Borel-Cantelli lemma and the fact that the rate function has a unique zero at $Y^0$. Indeed, $I(Y^0)=I(u^0)=0$ and if $I(z)=0$, one finds $(\psi_n)\subset L^2(0,T;U)$ with $z=u^{\psi_n}$ and $\|\psi_n\|_{L^2(0,T;U)}\to 0$. Then, $u^{\psi_n}\to z$ in $\MR(0,T)$ and by Proposition \ref{prop:cpt sublevel sets}, $u^{\psi_n}\to u^0=Y^0$ in $\MR(0,T)$, thus $\{z\in \MR(0,T):I(z)=0\}=\{Y^0\}$.
\end{remark}

\section{Application to fluid dynamics}\label{ss:fluid}
In this subsection, we apply our results to an abstract fluid dynamics model considered in several earlier works. We closely follow the presentation of \cite{AVsurvey,chueshovmillet10} and focus on what the large deviation principle of Theorem \ref{th: main LDP theorem} becomes in this setting. Afterwards, we specialize to the Navier--Stokes equations with gradient noise to make our results even more concrete.  

\subsection{Abstract model}
The abstract form of the problem we consider is as follows
\begin{equation}\label{eq:abstractfluid}
\left\{
\begin{aligned}
&\dd Y^{\varepsilon}(t) + A_0(t) Y^{\varepsilon} \,\dd t = \Phi(Y^{\varepsilon}(t), Y^{\varepsilon}(t)) \,\dd t + \sqrt{\varepsilon} \big(B_0(t) Y^{\varepsilon}(t) + G(t,Y^{\varepsilon}(t))\big) \,\dd W(t), 
\\ &u(0) = x.
\end{aligned}\right.
\end{equation}
Here, $\Phi$ is supposed to take care of the typical bilinear term appearing in equations in fluid dynamics. In particular, all of the following models can be included in the abstract framework below: 2D Navier--Stokes, 2D Boussinesq equations, quasigeostrophic equations, 2D magneto-hydrodynamic equations, 2D magnetic B\'enard problem, 3D Leray $\alpha$-model for Navier--Stokes equations and shell models of turbulence. 

To put this problem in the setting of \eqref{eq: original stoch ev eq} and Assumption \ref{ass: crit var set}, we assume the following. 
\begin{assumption}\label{ass:bilinearFluid}
\
\begin{enumerate}
\item\label{it1:bilinearFluid} $A_0\colon\R_+\to \calL(V, V^*)$ and $B_0\colon\R_+\to \calL(V,\calL_2(U,H))$ are measurable and for all $T>0$, $\sup_{t\in[0,T]}\|A_0(t)\|_{\calL(V, V^*)}<\infty$ and $\sup_{t\in[0,T]}\|B_0(t)\|_{\calL(V, \calL_2(U,H))}<\infty$. Moreover, for all $T>0$, there exist $\theta>0$ and $M\geq 0$ such that for all $v\in V$  and $t\in [0,T]$, 
    \[\lb v, A_0(t)v\rb  - \tfrac{1}{2} \|B_0(t) v\|_{\calL_2(U,H)}^2 \geq \theta\|v\|_V^2- M \|v\|_H^2.\]

\item\label{it2:bilinearFluid} For some $\beta_1\in(\frac{1}{2},\frac{3}{4}]$, $\Phi\colon V_{\beta_1}\times V_{\beta_1}\to V^*$ is bilinear and satisfies
\[\|\Phi(u, v)\|_{V^*}\leq C\|u\|_{\beta_1} \|v\|_{\beta_1}, \quad \lb u,\Phi(u,u)\rb = 0, \qquad u,v\in V.\]

\item\label{it3:bilinearFluid} For some $\beta_2\in (\frac12, 1)$, $G\colon\R_+\times V_{\beta_2}\to \calL_2(U,H)$  is measurable and satisfies the following Lipschitz conditions: for all $T>0$, there exists a constant $C$ such that for all $u,v\in V_{\beta_2}$ and $t\in [0,T]$, 
    \[\|G(t,u) - G(t,v)\|_{\calL_2(U,H)}\leq C \|u-v\|_{V_{\beta_2}}  \ \ \text{and} \ \ \|G(t,u)\|_{\calL_2(U,H)}\leq C(1+\|u\|_{V_{\beta_2}}).\]
\end{enumerate}
\end{assumption}

The associated skeleton equation is given by
\begin{equation}\label{eq: skeleton eqfluid}
  \begin{cases}
  &(u^\psi)'(t) + A_0(t) u^\psi(t) = \Phi(u^{\psi}(t), u^{\psi}(t)) + \big(B_0(t)u^\psi(t) + G(t,u^\psi(t)) \big)\psi(t), \quad t\in[0,T],\\
  &u^\psi(0)=x.
\end{cases}
\end{equation}

\begin{theorem}\label{thm:fluidabstract}
Suppose that Assumption \ref{ass:bilinearFluid} holds,
Then for every $x\in H$ and $\varepsilon\in (0,1]$, the problem \eqref{eq:abstractfluid} has a unique global solution 
\[
Y^{\varepsilon} \in L^2_{\rm loc}([0,\infty);V)\cap C([0,\infty);H) \ \text{a.s.}
\]
Moreover, for every $T>0$, $(Y^\eps)$ satisfies the LDP on $L^2(0,T;V)\cap C([0,T];H)$ with rate function $I\colon L^2_{\rm loc}(0,T;V)\cap C([0,T];H)\to [0,+\infty]$ given by
\[
I(z)=\frac{1}{2}\inf\Big\{\int_0^T\|\psi(s)\|_U^2\dd s : \psi\in L^2(0,T;U), z=u^{\psi}\Big\},
\]
where $\inf\varnothing\coloneqq +\infty$ and $u^\psi$ is the strong solution to \eqref{eq: skeleton eqfluid}.
\end{theorem} 
\begin{proof}
  In \cite[Th.\ 7.10]{AVsurvey} it is shown that Assumption \ref{ass:bilinearFluid} is satisfied, noting that the arguments also work for the time-dependent setting. Thus well-posedness follows from Theorem \ref{th: original stoch ev eq global well-posedness} and the large deviation principle follows from Theorem \ref{th: main LDP theorem}. 
\end{proof}  

\subsection{LDP for Navier--Stokes equations with gradient noise}\label{ss:NS}
Next we specialize the result to the 2D Navier--Stokes equations  on an arbitrary open set $\Dom\subseteq \R^2$ (possibly unbounded), and we let the noise term contain a transport/gradient term. The large deviation principle is new even for the case $\Dom = \R^2$. Indeed, as explained in the introduction, previous results in the literature either contain a gap, or do not have gradient noise, or assume boundedness of the domain $\Dom$.  
 
For simplicity we only consider the case of It\^o noise. For details on Stratonovich noise, see \cite[App.\ A]{AV24SNS}. We follow the presentation of \cite[\S 7.3.4]{AVsurvey}.  

Consider the following Navier--Stokes system with no-slip condition on domain $\Dom$:
\begin{equation}
\label{eq:Navier_Stokes}
\left\{
\begin{aligned}
&\dd Y^{\varepsilon} =\big[\nu \Delta Y^{\varepsilon} -(Y^{\varepsilon}\cdot \nabla)Y^{\varepsilon} -\nabla P^{\varepsilon}  \big] \,\dd  t  +\sqrt{\varepsilon} \sum_{n\geq 1}\big[(b_{n}\cdot\nabla) Y^{\varepsilon} +g_n(\cdot,Y^{\varepsilon}) -\nabla \wt{P}_n^{\varepsilon}\big] \,\dd W_t^n,
\\
&\div \,Y^{\varepsilon}=0,
\\ &Y^{\varepsilon}=0 \quad \text{on $\partial \Dom$},
\\ &Y^{\varepsilon}(0,\cdot)=u_0.
\end{aligned}\right.
\end{equation}
Here, $Y^{\varepsilon}\coloneqq(Y^{\varepsilon,1},Y^{\varepsilon,2})\colon[0,\infty)\times \O\times \Dom\to \R^2$ denotes the unknown velocity field, $P^{\varepsilon},\wt{P}_n^{\varepsilon}\colon[0,\infty)\times \O\times \Dom\to \R$ the unknown pressures, $(W_t^n:t\geq 0)_{n\geq 1}$ a given sequence of independent standard Brownian motions and
\begin{equation*}
(b_{n}\cdot\nabla) u\coloneqq\Big(\sum_{j\in \{1,2\}} b_n^j \partial_j u^k\Big)_{k=1,2},
\qquad (u\cdot \nabla ) u\coloneqq\Big(\sum_{j\in \{1,2\}} u^j \partial_j u^k\Big)_{k=1,2}.
\end{equation*}

\begin{assumption}\label{ass:SNS}
Let $d=2$. Let $b^j = (b^j_{n})_{n\geq 1}:\R_+\times \Dom\to \ell^2$ be measurable and bounded and suppose that for every $T>0$ there exists a $\mu\in (0,\nu)$
such that for all $x\in \Dom$ and $t\in [0,T]$, 
\begin{align*}
\frac{1}{2}\sum_{n\geq 1} \sum_{i,j\in \{1,2\}} b_n^i(x) b_n^j(x) \xi_i \xi_j \leq \mu |\xi|^2 \ \ \text{ for all }\xi\in \R^d.
\end{align*}
Moreover, $g^1, g^2\colon\R_+\times\Dom\times\R^2\to \ell^2$ and for every $T>0$ there exists a constant $L_g$ such that
\begin{align*}
\|g^j(t,x,y) - g^j(t,x,y')\|_{\ell^2}&\leq L_g |y-y'| \\ 
\|g^j (t,x,y)\|_{\ell^2}& \leq L_g(1+|y|), \ \ \ x\in \Dom, y,y'\in \R^2, \ t\in [0,T], \ j\in \{1, 2\}.  
\end{align*}
\end{assumption}

As in \cite[\S 7.3.4]{AVsurvey}, we can use the Helmholtz projection $\P$  to rewrite \eqref{eq:Navier_Stokes} as \eqref{eq:abstractfluid}. To this end, let $\mathcal{U} = \ell^2$ with standard basis $(e_n)_{n\geq 1}$ and let
\[
H = \Ls^2(\Dom), \quad  V = \Hs^1_0(\Dom) = H^1_0(\Dom;\R^2)\cap \Ls^2(\Dom), \quad   V^* \coloneqq \Hs^{-1}(\Dom) = (\Hs^1_0(\Dom))^*,
\]
where $\Ls^2(\Dom)$ denotes the range of the Helmholtz projection in $L^2(\Dom;\R^2)$. 
By the divergence free condition, $(u\cdot \nabla)u=\div(u\otimes u)$, where $u\otimes u$ is the matrix with components $u_{j} u_k$.
Assuming $x\in \Ls^2(\Dom)$, after applying the Helmholtz projection $\pr$ to \eqref{eq:Navier_Stokes}, we can write \eqref{eq:Navier_Stokes} in the form \eqref{eq:abstractfluid} with \[A_0 = -\nu \pr\Delta,\quad \Phi(u,v) =- \pr\div[u\otimes v], \quad (B_0 u)e_n = \pr[(b_{n}\cdot\nabla) u], \quad G(u) e_n = \pr g_n(\cdot, u).\]

For $\psi \in L^2(0,T;\ell^2)$, consider the following skeleton equation on $\Dom$: 
\begin{equation}
\label{eq:Navier_Stokes-skeleton}
\left\{
\begin{aligned}
&\dd u^{\psi} =\big[\nu \Delta u^{\psi} - \P\div(u^{\psi}\otimes u^{\psi})\big] \,\dd  t
+\sum_{n\geq 1} \big(\P[(b_{n}\cdot\nabla) u^{\psi}] +\P g_n(\cdot,u^{\psi})\big)\psi_n,
\\ &u^{\psi}=0 \quad \text{on $\partial \Dom$},
\\ &u^{\psi}(0,\cdot)=u_0.
\end{aligned}\right.
\end{equation}

In \cite[\S 7.3.4]{AVsurvey} it is verified that  Assumption \ref{ass:bilinearFluid} is fulfilled for the above setting. Thus we obtain the next result immediately from Theorem \ref{thm:fluidabstract}.

\begin{theorem}[LDP for the 2D Navier--Stokes equations with transport noise]\label{thm:SNS}
Let $d=2$. Suppose that Assumption \ref{ass:SNS} holds,
Then for every $x\in \Ls^2(\Dom)$ and $\varepsilon\in (0,1]$, there exists a unique global solution $Y^{\varepsilon}\in L^2_{\rm loc}([0,\infty);\Hs^1_0(\Dom))\cap C([0,\infty);\Ls^2(\Dom))$ to \eqref{eq:Navier_Stokes}. Moreover, for every $T>0$, $(Y^\eps)$ satisfies the LDP on $\MR(0,T)\coloneqq L^2(0,T;\Hs^1_0(\Dom))\cap C([0,T];\Ls^2(\Dom))$ with rate function $I\colon \MR(0,T)\to [0,+\infty]$ given by
\[
I(z)=\frac{1}{2}\inf\Big\{\int_0^T\|\psi(s)\|_{\ell^2}^2 \dd s : \psi\in L^2(0,T;\ell^2), z=u^{\psi}\Big\},
\]
where $\inf\varnothing\coloneqq +\infty$ and $u^\psi$ is the strong solution to \eqref{eq:Navier_Stokes-skeleton}.
\end{theorem}

\appendix

\section{}\label{appendix}

For convenience we state some tools that are used repeatedly. To begin, let us state a direct consequence of Gronwall's inequality.

\begin{lemma}[Gronwall]\label{lem: gronwall consequence}
Let $T>0$ and let $F,G,H,K\colon [0,T]\to\R_+$ with $F$ and $G$ continuous, $K$ non-decreasing and $H\in L^1(0,T)$.
Suppose that $F(t)\leq -G(t)+K(t)+\int_0^t F(s)H(s)\dd s$ for all $t\in[0,T]$.
Then
\[
\sup_{t\in[0,T]}F(t)\vee \sup_{t\in[0,T]}G(t)\leq K(T)\exp[\|H\|_{L^1(0,T)}].
\]
\end{lemma}

The following special case of a chain rule from \cite{pardoux} is useful, since it applies to $L^2(0,T;V^*)+L^1(0,T;H)$-valued integrands.

\begin{lemma}{\cite[Lem.\ 2.2 p.\ 30]{pardoux}}
Let $(V,H,V^*)$ be a Gelfand triple of Hilbert spaces. Let $x\in H$, $u\in C([0,T];H)\cap L^2(0,T;V)$ and $v\in L^2(0,T;V^*)+L^1(0,T;H)$ be such that
\begin{equation}\label{eq: cond Ito}
u(t)=x+\int_0^t v(s)\dd s \text{ in }  V^*,\quad \text{for all }t\in[0,T].
\end{equation}
Then for all $t\in[0,T]$:
\begin{equation}\label{eq: Ito pardoux deterministic}
\|u(t)\|_H^2=\|x\|_H^2+2\int_0^t \<v(s),u(s)\?\dd s.
\end{equation}
\end{lemma}
\begin{proof}
Note that $u\in L^2(0,T;V)\cap L^\infty(0,T;H)$ and $v\in L^1(0,T;V^*)$. Thus by \eqref{eq: cond Ito}, $u$ is weakly differentiable with $u'=v$ a.e.\ on $[0,T]$, see \cite[Lem.\ 2.5.8]{HNVWvolume1}. Hence, $u'=v\in L^2(0,T;V^*)+L^1(0,T;H)$. Also, \eqref{eq: cond Ito} implies absolute continuity of $u:[0,T]\to V^*$.   Now \cite[Lem.\ 2.2  p.\ 30, $p=2$]{pardoux} gives $\frac{\dd}{\dd t}\|u(t)\|_H^2=2\<u'(t),u(t)\?=2\<v(t),u(t)\?$ a.e., proving  \eqref{eq: Ito pardoux deterministic}.
\end{proof}

Stochastic versions of the chain rule, or It\^o formula, are  also given in \cite{pardoux}. The following special case is suited for random, $L^2(0,T;V^*)+L^1(0,T;H)$-valued integrands.  We recall that the class of integrable processes for a $U$-cylindrical Brownian motion (Definition \ref{def: cylindrical BM}) is given by
\begin{align}\label{eq: def class integrable processes}
\mathcal{N}(0,T)\coloneqq& \Bigl\{\Phi\colon [0,T]\times \Om\to \UH: \Phi \text{ strongly progressively measurable, }\\[-0.15cm]
&\hspace{5.7cm}\P(\|{\Phi}\|_{L^2(0,T;\UH)}<\infty)=1\Bigr\}.\notag
\end{align}

\begin{lemma}{\cite[Th.\ 3.1 p.\ 57, Th.\ 3.3 p.\ 59]{pardoux}}\label{lem: Ito pardoux}
  Let $(V,H,V^*)$ be a Gelfand triple of Hilbert spaces and let $(\Om,\F,\P,(\F_t)_{t\in\R_+})$ be a filtered probability space. Suppose that
  \begin{enumerate}[label=\emph{(\roman*)},ref=\textup{(\roman*)}]
    \item $u\in L^0(\Om;L^2(0,T;V))$, $u_0\in L^0(\Om,\F_0,\P;H)$,
    \item $v\in L^0(\Om;L^1(0,T;H))+L^0(\Om;L^{2}(0,T;V^*))$, $v$ is adapted,
    \item $\Phi\in\mathcal{N}(0,T)$, $W$ is a $U$-cylindrical Brownian motion,
    \item a.s. for all $t\in[0,T]$: $u(t)=u_0+\int_0^t v(s)\dd s+\int_0^t \Phi(s)\dd W(s)$.
  \end{enumerate}
  Then, $u\in L^0(\Om;C([0,T];H))$ and a.s.\ for all  $t\in[0,T]$:
  \begin{equation}\label{eq: Ito pardoux}
  \|u(t)\|_H^2 =\|u_0\|_H^2+2\int_0^t\<v(s),u(s)\?\dd s+2\int_0^t\<u(s),\Phi(s) \dd W(s)\?+\int_0^t \nn \Phi(s)\nn_H^2 \dd s.
  \end{equation}
\end{lemma}

Finally, we relate the $U$-cylindrical Brownian motion $W$ of Definition \ref{def: cylindrical BM} to the $\R^\infty$-Brownian motion $\tilde{W}$ of Definition \ref{def: sequence independent BM}, as well as their stochastic integrals constructed in \cite{NVW15} and \cite{liurockner15}, respectively.

An $\R^\infty$-Brownian motion $\tilde{W}=((\beta_k)_{k\in\N},(e_k)_{k\in\N})$ in $U$ corresponds to a Wiener process $\tilde{W}_1$ in a larger space $U_1$, with trace class covariance. That is, for any Hilbert-Schmidt embedding $J\colon U\into U_1$, the $U_1$-valued process given by
\begin{equation}\label{eq: Q-Wiener appendix}
\tilde{W}_1(t)\coloneqq \sum_{k=1}^{\infty}\beta_k(t)J e_k, \qquad t\in[0,T],
\end{equation}
defines a $Q_1$-Wiener process on $U_1$, with $Q_1\coloneqq JJ^*\in\mathcal{L}(U_1,U_1)$ nonnegative definite, symmetric and of trace class \cite[Prop.\ 2.5.2]{liurockner15}.

It is well-known that $\mathcal{N}(0,T)$ from  \eqref{eq: def class integrable processes} is the class of integrable processes for both $\tilde{W}$ and $W$,   see \cite[p.\ 52, p.\ 53]{liurockner15}, \cite[Prop.\ 2.13]{AV22nonlinear1} and the  proof in \cite[p.\ 306, \S 5.4 $(p=0)$]{NVW15}.
The next proposition relates the stochastic integrals corresponding to $\tilde{W}$ and $W$.

\begin{proposition}\label{prop: equiv cyl BM}
For any $U$-cylindrical Brownian motion $W\in\mathcal{L}(L^2(\R_+;U),L^2(\Omega))$ and any orthonormal basis $(e_k)_{k\in\N}$ of $U$, there exists an $\R^\infty$-Brownian motion $\tilde{W}=((\beta_k)_{k\in\N},(e_k)_{k\in\N})$ with
\begin{equation}\label{eq: assump beta_k}
W(\one_{(0,t]}\otimes e_k)=\beta_k(t) \text{\quad in }L^2(\Om),\quad\text{ for all } k\in\N \text{ and } t\in \R_+.
\end{equation}
The sequence $(\beta_k)_{k\in\N}$ in $\tilde{W}$ is unique up to indistinguishability.

Reversely, given an $\R^\infty$-Brownian motion $\tilde{W}=((\beta_k)_{k\in\N},(e_k)_{k\in\N})$, there exists a unique $U$-cylindrical Brownian motion $W\in\mathcal{L}(L^2(\R_+;U),L^2(\Omega))$ that satisfies \eqref{eq: assump beta_k}.

If \eqref{eq: assump beta_k} holds, then for  any $\Phi\in \mathcal{N}(0,T)$ and $t\in[0,T]$, we have $\P$-a.s.\ in $C([0,T];H)$:
\begin{equation}\label{eq: identity stoch integrals defs}
\int_0^t \Phi(s)\dd W(s)=\int_0^t \Phi(s)\circ J^{-1}\dd \tilde{W}_1(s)\eqqcolon \int_0^t \Phi(s)\dd \tilde{W}(s),\qquad  t\in[0,T],
\end{equation}
with $\tilde{W}_1$ as in \eqref{eq: Q-Wiener appendix}. Here, the integral on the left-hand side is the one constructed in \cite{NVW15} and the middle and right integral are those constructed in \cite[\S 2.3, \S 2.5]{liurockner15}.
\end{proposition}

\printbibliography

\end{document}